\documentclass[graybox]{svmult}
\usepackage{amsmath,amsfonts,amssymb}
\usepackage{mathptmx}
\usepackage{helvet}
\usepackage{courier}
\usepackage{mathtools}
\usepackage{paralist}
\usepackage{type1cm}
\usepackage{makeidx}
\usepackage{graphicx}
\usepackage{multicol}
\usepackage[bottom]{footmisc}

\usepackage{cite}
\usepackage{comment}
\PassOptionsToPackage{hyphens}{url}\usepackage[hidelinks]{hyperref}
\usepackage{mdframed}
\usepackage{nicefrac}
\usepackage{xfrac}
\usepackage[dvipsnames]{xcolor}
\usepackage{pgfplots,pgfplotstable}
\usepackage{pmat}
\usepackage{subcaption}
\usepackage{tabularx}
\usetikzlibrary{external}

\makeindex

\def\imagebox#1#2{\vtop to #1{\vfill\hbox{#2}\vfill}}

\newcommand{\eqbydef}{\mathrel{\mathop:}=}

\DeclareMathOperator{\card}{card}

\newcommand{\SCAL}{{\cdot}}
\newcommand{\GRAD}{\vec{\nabla}}
\newcommand{\DIV}{\vec{\nabla}{\cdot}}
\newcommand{\LAPL}{{\triangle}}
\newcommand{\GRADh}{\GRAD_h}

\newcommand{\st}{\; | \;}

\newcommand{\closure}[1]{\overline{#1}}
\newcommand{\restrto}[2]{#1{}_{|#2}}

\newcommand{\norm}[2][]{\|#2\|_{#1}}
\newcommand{\seminorm}[2][]{|#2|_{#1}}

\newcommand{\meas}[2][d]{|#2|_{#1}}

\newcommand{\term}{\mathfrak{T}}

\newtheorem{assumption}{Assumption}

\newcommand{\Real}{\mathbb{R}}
\newcommand{\Natural}{\mathbb{N}}
\newcommand{\Poly}[1]{\mathbb{P}^{#1}}

\newcommand{\Mh}[1][h]{\mathcal{M}_{#1}}
\newcommand{\Th}[1][h]{\mathcal{T}_{#1}}
\newcommand{\Fh}[1][h]{\mathcal{F}_{#1}}
\newcommand{\Fhi}{\Fh^{{\rm i}}}
\newcommand{\Fhb}{\Fh^{{\rm b}}}
\newcommand{\fTh}[1][h]{\mathfrak{T}_{#1}}
\newcommand{\fFh}[1][h]{\mathfrak{F}_{#1}}
\newcommand{\Fhh}[2]{\Fh[#1,#2]} 
\newcommand{\Thh}[2]{\Th[#1,#2]} 

\newcommand{\normal}{\vec{n}}

\newcommand{\lproj}[2][h]{\pi_{#1}^{0,#2}}
\newcommand{\eproj}[2][T]{\pi_{#1}^{1,#2}}
\newcommand{\vlproj}[2][h]{\vec{\pi}_{#1}^{0,#2}}

\newcommand{\Osw}[1][k+1]{\mathcal{I}_{h}^{#1}}
\newcommand{\argmin}{\operatornamewithlimits{arg\,min}}

\newcommand{\UT}[1][k]{\underline{U}_T^{#1}}
\newcommand{\Uh}[1][k]{\underline{U}_h^{#1}}
\newcommand{\UhD}[1][k]{\underline{U}_{h,0}^{#1}}
\newcommand{\IT}[1][k]{\underline{I}_T^{#1}}
\newcommand{\Ih}[1][k]{\underline{I}_h^{#1}}

\newcommand{\DT}[1][k]{\underline{U}_{\partial T}^{#1}}
\newcommand{\DpT}[1][k]{\underline{\Delta}_{\partial T}^k}

\newcommand{\uu}[1][T]{\underline{u}_{#1}}
\newcommand{\uhu}[1][h]{\hat{\underline{u}}_{#1}}
\newcommand{\uv}[1][T]{\underline{v}_{#1}}
\newcommand{\uphi}[1][T]{\underline{\varphi}_{#1}}

\newcommand{\cu}[1][T]{\check{u}_{#1}}
\newcommand{\cv}[1][T]{\check{v}_{#1}}
\newcommand{\cw}[1][T]{\check{w}_{#1}}

\newcommand{\ual}[1][\partial T]{\underline{\alpha}_{#1}}

\newcommand{\RpT}[1][k]{\underline{R}_{\partial T}^{#1}}

\newcommand{\sU}[1][h]{\mathsf{U}_{#1}}
\newcommand{\sV}[1][h]{\mathsf{V}_{#1}}
\newcommand{\sA}[1][h]{\mathsf{A}_{#1}}
\newcommand{\sB}[1][h]{\mathsf{B}_{#1}}

\newcommand{\trans}{^{\rm T}}

\newcommand{\pT}[1][k+1]{p_T^{#1}}
\newcommand{\ph}[1][k+1]{p_h^{#1}}
\newcommand{\GT}[1][k]{\vec{G}_T^{#1}}
\newcommand{\Gh}[1][k]{\vec{G}_h^{#1}}

\newcommand{\jump}[2][F]{[#2]_{#1}}

\newcommand{\est}[2][T]{\eta_{{\rm #2},#1}}

\newcommand{\vel}{\vec{\beta}}
\newcommand{\diff}[1][]{\kappa_{#1}}

\newcommand{\reac}{\mu}
\newcommand{\Pe}{\mathrm{Pe}}

\newcommand{\Gvel}[2][k]{\vec{G}_{\vel,#2}^{#1}}
\newcommand{\tauref}[1][T]{\hat{\tau}_{#1}}
\newcommand{\velref}[1][T]{\hat{\vel}_{#1}}
\newcommand{\uw}[1][T]{\underline{w}_{#1}}

\newcommand{\logLogSlopeTriangle}[5]
{
    \pgfplotsextra
    {
        \pgfkeysgetvalue{/pgfplots/xmin}{\xmin}
        \pgfkeysgetvalue{/pgfplots/xmax}{\xmax}
        \pgfkeysgetvalue{/pgfplots/ymin}{\ymin}
        \pgfkeysgetvalue{/pgfplots/ymax}{\ymax}

        \pgfmathsetmacro{\xArel}{#1}
        \pgfmathsetmacro{\yArel}{#3}
        \pgfmathsetmacro{\xBrel}{#1-#2}
        \pgfmathsetmacro{\yBrel}{\yArel}
        \pgfmathsetmacro{\xCrel}{\xArel}

        \pgfmathsetmacro{\lnxB}{\xmin*(1-(#1-#2))+\xmax*(#1-#2)} 
        \pgfmathsetmacro{\lnxA}{\xmin*(1-#1)+\xmax*#1} 
        \pgfmathsetmacro{\lnyA}{\ymin*(1-#3)+\ymax*#3} 
        \pgfmathsetmacro{\lnyC}{\lnyA+#4*(\lnxA-\lnxB)}
        \pgfmathsetmacro{\yCrel}{\lnyC-\ymin)/(\ymax-\ymin)}

        \coordinate (A) at (rel axis cs:\xArel,\yArel);
        \coordinate (B) at (rel axis cs:\xBrel,\yBrel);
        \coordinate (C) at (rel axis cs:\xCrel,\yCrel);

        \draw[#5]   (A)-- node[pos=0.5,anchor=north] {\scriptsize{1}}
                    (B)-- 
                    (C)-- node[pos=0.,anchor=west] {\scriptsize{#4}} 
                    cycle;
    }
}

\newcommand{\reverseLogLogSlopeTriangle}[5]
{
    \pgfplotsextra
    {
        \pgfkeysgetvalue{/pgfplots/xmin}{\xmin}
        \pgfkeysgetvalue{/pgfplots/xmax}{\xmax}
        \pgfkeysgetvalue{/pgfplots/ymin}{\ymin}
        \pgfkeysgetvalue{/pgfplots/ymax}{\ymax}

        \pgfmathsetmacro{\xArel}{#1}
        \pgfmathsetmacro{\yArel}{#3}
        \pgfmathsetmacro{\xBrel}{#1-#2}
        \pgfmathsetmacro{\yBrel}{\yArel}
        \pgfmathsetmacro{\xCrel}{\xBrel}

        \pgfmathsetmacro{\lnxB}{\xmin*(1-(#1-#2))+\xmax*(#1-#2)} 
        \pgfmathsetmacro{\lnxA}{\xmin*(1-#1)+\xmax*#1} 
        \pgfmathsetmacro{\lnyA}{\ymin*(1-#3)+\ymax*#3} 
        \pgfmathsetmacro{\lnyC}{\lnyA+#4*(\lnxA-\lnxB)}
        \pgfmathsetmacro{\yCrel}{\lnyC-\ymin)/(\ymax-\ymin)}

        \coordinate (A) at (rel axis cs:\xArel,\yArel);
        \coordinate (B) at (rel axis cs:\xBrel,\yBrel);
        \coordinate (C) at (rel axis cs:\xCrel,\yCrel);

        \draw[#5]   (A)-- node[pos=0.5,anchor=north] {\scriptsize{1}}
                    (B)-- 
                    (C) -- node[pos=0.,anchor=east] {\scriptsize{#4}}
                    cycle;
    }
}

\begin{document}

\title*{An introduction to Hybrid High-Order methods}
\author{Daniele A. Di Pietro and Roberta Tittarelli}
\institute{Daniele A. Di Pietro \and Roberta Tittarelli \at Institut Montpelli\'{e}rain Alexander Grothendieck, CNRS, Univ. de Montpellier. \email{daniele.di-pietro@umontpellier.fr}, \email{roberta.tittarelli@umontpellier.fr}}
\maketitle

\abstract{This chapter provides an introduction to Hybrid High-Order (HHO) methods.
These are new generation numerical methods for PDEs with several advantageous features: the support of arbitrary approximation orders on general polyhedral meshes, the reproduction at the discrete level of relevant continuous properties, and a reduced computational cost thanks to static condensation and compact stencil.
After establishing the discrete setting, we introduce the basics of HHO methods using as a model problem the Poisson equation. We describe in detail the construction, and prove a priori convergence results for various norms of the error as well as a posteriori estimates for the energy norm.
We then consider two applications: the discretization of the nonlinear $p$-Laplace equation and of scalar diffusion-advection-reaction problems.
The former application is used to introduce compactness analysis techniques to study the convergence to minimal regularity solution.
The latter is used to introduce the discretization of first-order operators and the weak enforcement of boundary conditions.
Numerical examples accompany the exposition.}


\section{Introduction}\label{sec:introduction}

This chapter provides an introduction to Hybrid High-Order (HHO) methods.
The material is closely inspired by a series of lectures given by the first author at Institut Henri Poincar\'{e} in September 2016 within the thematic quarter \emph{Numerical Methods for PDEs} (see~\url{http://tinyurl.com/IHP-quarter-nmpdes}).

HHO methods, introduced in~\cite{Di-Pietro.Ern.ea:14,Di-Pietro.Ern:15}, are discretization methods for Partial Differential Equations (PDEs) with relevant features that set them apart from classical techniques such as finite elements or finite volumes.
These include, in particular:
\begin{enumerate}[(i)]
\item The support of general polytopal meshes in arbitrary space dimension, paving the way to a seamless treatment of complex geometric features and unified 1d-2d-3d implementations;
\item The possibility to select the approximation order which, possibly combined with adaptivity, leads to a reduction of the simulation cost for a given precision or better precision for a given cost;
\item The compliance with the physics, including robustness with respect to the variations of physical coefficients and reproduction at the discrete level of key continuous properties such as local balances and flux continuity;
\item A reduced computational cost thanks to their compact stencil along with the possibility to perform static condensation.
\end{enumerate}

As of today, HHO methods have been successfully applied to the discretization of several linear and nonlinear problems of engineering interest including:
variable diffusion~\cite{Di-Pietro.Ern.ea:14,Di-Pietro.Ern.ea:16,Di-Pietro.Ern:16},
quasi incompressible linear elasticity~\cite{Di-Pietro.Ern:15,Di-Pietro.Ern:15*1},
locally degenerate diffusion-advection-reaction~\cite{Di-Pietro.Droniou.ea:15},
poroelasticity~\cite{Boffi.Botti.ea:16},
creeping flows~\cite{Aghili.Boyaval.ea:15} possibly driven by volumetric forces with large irrotational part~\cite{Di-Pietro.Ern.ea:16*1},
electrostatics~\cite{Di-Pietro.Specogna:16},
phase separation problems governed by the Cahn--Hilliard equation~\cite{Chave.Di-Pietro.ea:16},
Leray--Lions type elliptic problems~\cite{Di-Pietro.Droniou:16,Di-Pietro.Droniou:17}.
More recent applications also include steady incompressible flows governed by the Navier--Stokes equations~\cite{Di-Pietro.Krell:16} and nonlinear elasticity~\cite{Botti.Di-Pietro.ea:16}.
Generalizations of HHO methods and comparisons with other (new generation or classical) discretization methods for PDEs can be found in~\cite{Cockburn.Di-Pietro.ea:16,Boffi.Di-Pietro:16}.
Implementation tools based on advanced programming techniques have been recently discussed in~\cite{Cicuttin.Di-Pietro.ea:17}.

Discretization methods that support polytopal meshes and, possibly, arbitrary approximation orders have experienced a vigorous development over the last decade.
Novel approaches to the analysis and the design have been developed borrowing ideas from other branches of mathematics (such as topology and geometry), or expanding past their initial limits the original ideas underlying finite element or finite volume methods.
A brief state-of-the-art is provided in what follows.

Several lowest-order methods for diffusive problems have been proposed to circumvent the strict conditions of mesh-data compliance required for the consistency of classical (two-points) finite volume schemes; see~\cite{Droniou:14} for a comprehensive review.
We mention here, in particular, the Mixed and Hybrid Finite Volume methods of~\cite{Droniou.Eymard:06,Eymard.Gallouet.ea:10}.
These methods possess local conservation properties on the primal mesh, and enable an explicit identification of equilibrated numerical fluxes.
Their relation with the lowest-order version of HHO methods has been studied in~\cite[Section~2.5]{Di-Pietro.Ern.ea:14} for pure diffusion and in~\cite[Section~5.4]{Di-Pietro.Droniou.ea:15} for advection-diffusion-reaction.
Other families of lowest-order methods have been obtained by reproducing at the discrete level salient features of the continuous problem.
Mimetic Finite Difference methods are derived by emulating the Stokes theorems to formulate counterparts of differential operators and of $L^2$-products; cf.~\cite{Brezzi.Lipnikov.ea:05} and~\cite{Droniou.Eymard.ea:10} for a study of their relation with Mixed and Hybrid Finite Volume methods.
In the Discrete Geometric Approach of~\cite{Codecasa.Specogna.ea:10} as well as in Compatible Discrete Operators~\cite{Bonelle.Ern:14}, formal links with the continuous operators are expressed in terms of Tonti diagrams.
To different extents, the aforementioned methods owe to the seminal ideas of Whitney on geometric integration~\cite{Whitney:57}.
A different approach to lowest-order schemes on general meshes consists in extending classical properties of nonconforming and penalized finite elements as in the Cell Centered Galerkin~\cite{Di-Pietro:12} and generalized Crouzeix--Raviart~\cite{Di-Pietro.Lemaire:15} methods.
We also cite here~\cite{Vohralik.Wohlmuth:13} concerning the use of classical mixed finite elements on polyhedral meshes (see, in particular, Section 7 therein).
Further investigations have recently lead to unifying frameworks that encompass the above (and other) methods. We mention, in particular, the Gradient Schemes discretizations of~\cite{Droniou.Eymard.ea:13}.
Finally, the methods discussed here can often be regarded as lowest-order versions of more recent technologies.

Methods that support the possibility to increase the approximation order have received a considerable amount of attention over the last few years.
High-order discretizations on general meshes that are possibly physics-compliant can be obtained by the discontinuous Galerkin approach; cf., e.g.,~\cite{Arnold.Brezzi.ea:02,Di-Pietro.Ern:12} and also~\cite{Bassi.Botti.ea:12}.
Discontinuous Galerkin methods, however, have some practical limitations.
For problems in incompressible fluid mechanics, e.g., a key ingredient for inf-sup stability is a reduction map that can play the role of a Fortin interpolator.
Unfortunately, such an interpolator is not available for discontinuous Galerkin methods on non-standard elements.
Additionally, in particular for modal implementations on general meshes, the number of unknowns can become unbearably large.
This has motivated the introduction of Hybridizable Discontinuous Galerkin methods~\cite{Castillo.Cockburn.ea:00,Cockburn.Gopalakrishnan.ea:09}, which mainly focus on standard meshes (the extension to general meshes is possible in some cases); see also the very recent $M$-decomposition techniques~\cite{Cockburn.Fu:17}.
High-order discretization methods that support general meshes also include Virtual Element methods; cf.~\cite{Beirao-da-Veiga.Brezzi.ea:13*1} for an introduction.
In short, Virtual Element methods are finite element methods where explicit expressions for the basis functions are not available at each point, and computable approximations thereof are used instead.
This provides the extra flexibility required, e.g., to handle polyhedral elements.
Links between HHO and the nonconforming Virtual Element method have been pointed out in~\cite[Section~2.4]{Cockburn.Di-Pietro.ea:16}; see also~\cite{Boffi.Di-Pietro:16} concerning HHO and mixed Virtual Element methods.

We next describe in detail the content of this chapter.
We start in Section~\ref{sec:setting} by presenting the discrete setting: we introduce the notion of polytopal mesh (Section~\ref{sec:setting:mesh}), formulate assumptions on the way meshes are refined that are suitable to carry out a $h$-convergence analysis (Section~\ref{sec:setting:regular.mesh}), introduce the local polynomial spaces (Section~\ref{sec:setting:spaces}) and projectors (Section~\ref{sec:setting:projectors}) that lie at the heart of the HHO construction.

In Section~\ref{sec:basics} we present the basic principles of HHO methods using as a model problem the Poisson equation.
While the material in this section is mainly adapted from~\cite{Di-Pietro.Ern.ea:14}, some results are new and the arguments have been shortened or made more elegant.
In Section~\ref{sec:basics:local.construction} we introduce the local space of degrees of freedom (DOFs) and discuss the main ingredients upon which HHO methods rely, namely:
\begin{enumerate}[(i)]
\item Reconstructions of relevant quantities obtained by solving small, embarassingly parallel problems on each element; 
\item High-order stabilization terms obtained by penalizing cleverly designed residuals.
\end{enumerate}
In Section~\ref{sec:basics:discrete.problem} we show how to combine these ingredients to formulate local contributions, which are then assembled element-by-element as in standard finite elements.
The construction is conceived so that only face-based DOFs are globally coupled, which paves the way to efficient practical implementations where element-based DOFs are statically condensed in a preliminary step.
In Sections~\ref{sec:basics:apriori.error.analysis}~and~\ref{sec:basics:aposteriori.error.analysis} we discuss, respectively, optimal a priori estimates for various norms and seminorms of the error, and residual-based a posteriori estimates for the energy-norm of the error. 
Finally, some numerical examples are provided in Section~\ref{sec:basics:num.ex} to demonstrate the theoretical results.

In Section~\ref{sec:plap} we consider the HHO discretization of the $p$-Laplace equation.
The material is inspired by~\cite{Di-Pietro.Droniou:16,Di-Pietro.Droniou:17}, where more general Leray--Lions operators are considered.
When dealing with nonlinear problems, regularity for the exact solution is often difficult to prove and can entail stringent assumptions on the data.
For this reason, the $h$-convergence analysis can be carried out in two steps:
in a first step, convergence to minimal regularity solutions is proved by a compactness argument;
in a second step, convergence rates are estimated for smooth solutions (and smooth data).
Convergence by compactness typically requires discrete counterparts of functional analysis results relevant for the study of the continuous problem.
In our case, two sets of discrete functional analysis results are needed: discrete Sobolev embeddings (Section~\ref{sec:plap:sobolev.embeddings}) and compactness for sequences of HHO functions uniformly bounded in a $W^{1,p}$-like seminorm (Section~\ref{sec:plap:compactness}).
The interest of both results goes beyond the specific method and problem considered here.
As an example, in~\cite{Di-Pietro.Krell:16} they are used for the analysis of a HHO discretization of the steady incompressible Navier--Stokes equations.
The HHO method for the $p$-Laplacian stated in Section~\ref{sec:plap:discrete} is designed according to similar principles as for the Poisson problem.
Convergence results are stated in Section~\ref{sec:plap:convergence}, and numerical examples are provided in Section~\ref{sec:plap:num}.

Following~\cite{Di-Pietro.Droniou.ea:15}, in Section~\ref{sec:adr} we extend the HHO method to diffusion-advection-reaction problems.
In this context, a crucial property from the numerical point of view is robustness in the advection-dominated regime.
In Section~\ref{sec:adr:weak.bc} we modify the diffusive bilinear form introduced in Section~\ref{sec:basics:discrete.problem} to incorporate weakly enforced boundary conditions.
The weak enforcement of boundary conditions typically improves the behaviour of the method in the presence of boundary layers, since the discrete solution is not constrained to a fixed value on the boundary.
In Section~\ref{sec:adr:advection} we introduce the HHO discretization of first-order terms based on two novel ingredients: a local advective derivative reconstruction and an upwind penalty term. The former is used to formulate the consistency terms, while the role of the latter is to confer suitable stability properties to the advective-reactive bilinear form.
The HHO discretization is finally obtained in Section~\ref{sec:adr:global.problem} combining the diffusive and advective-reactive contributions, and its stability with respect to an energy-like norm including an advective derivative contribution is studied.
In Section~\ref{sec:adr:convergence} we state an energy-norm error estimate which accounts for the dependence of the error contribution of each mesh element on a local P\'{e}clet number.
A numerical illustration is provided in Section~\ref{sec:adr:numerical.example}.


\section{Discrete setting}\label{sec:setting}

Let $\Omega\subset\Real^d$, $d\in\Natural^*$, denote a bounded connected open polyhedral domain with Lipschitz boundary and outward normal $\normal$.
We assume that $\Omega$ does not have cracks, i.e., it lies on one side of its boundary.
In what follows, we introduce the notion of polyhedral mesh of $\Omega$, formulate assumptions on the way meshes are refined that enable to prove useful geometric and functional results, and introduce functional spaces and projectors that will be used in the construction and analysis of HHO methods.

\subsection{Polytopal mesh}\label{sec:setting:mesh}

The following definition enables the treatment of meshes as general as the ones depicted in Fig.~\ref{fig:meshes}.
\begin{figure}\centering
  \begin{minipage}{0.24\textwidth}\centering
    \imagebox{3.5cm}{\includegraphics[height=2.5cm]{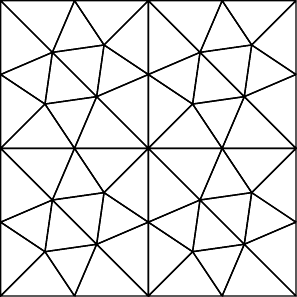}}
    \subcaption{Matching triangular\label{fig:meshes:triangular}}
  \end{minipage}
  \begin{minipage}{0.24\textwidth}\centering
    \imagebox{3.5cm}{\includegraphics[height=2.5cm]{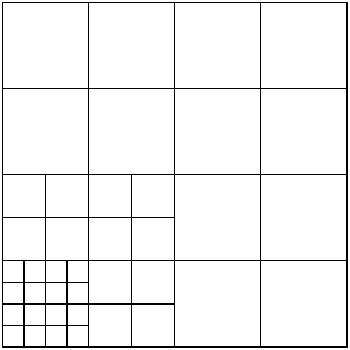}}
    \subcaption{Nonconforming\label{fig:meshes:nonconforming}}
  \end{minipage}
  \begin{minipage}{0.24\textwidth}\centering
    \imagebox{3.5cm}{\includegraphics[height=2.5cm]{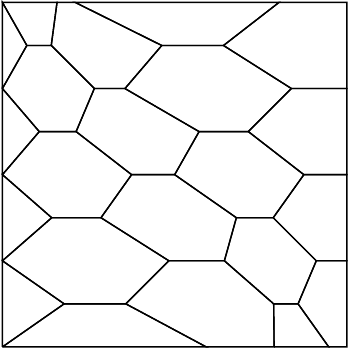}}
    \subcaption{Polygonal\label{fig:meshes:polygonal}}
  \end{minipage}
  \begin{minipage}{0.26\textwidth}\centering  
    \includegraphics[height=3.5cm]{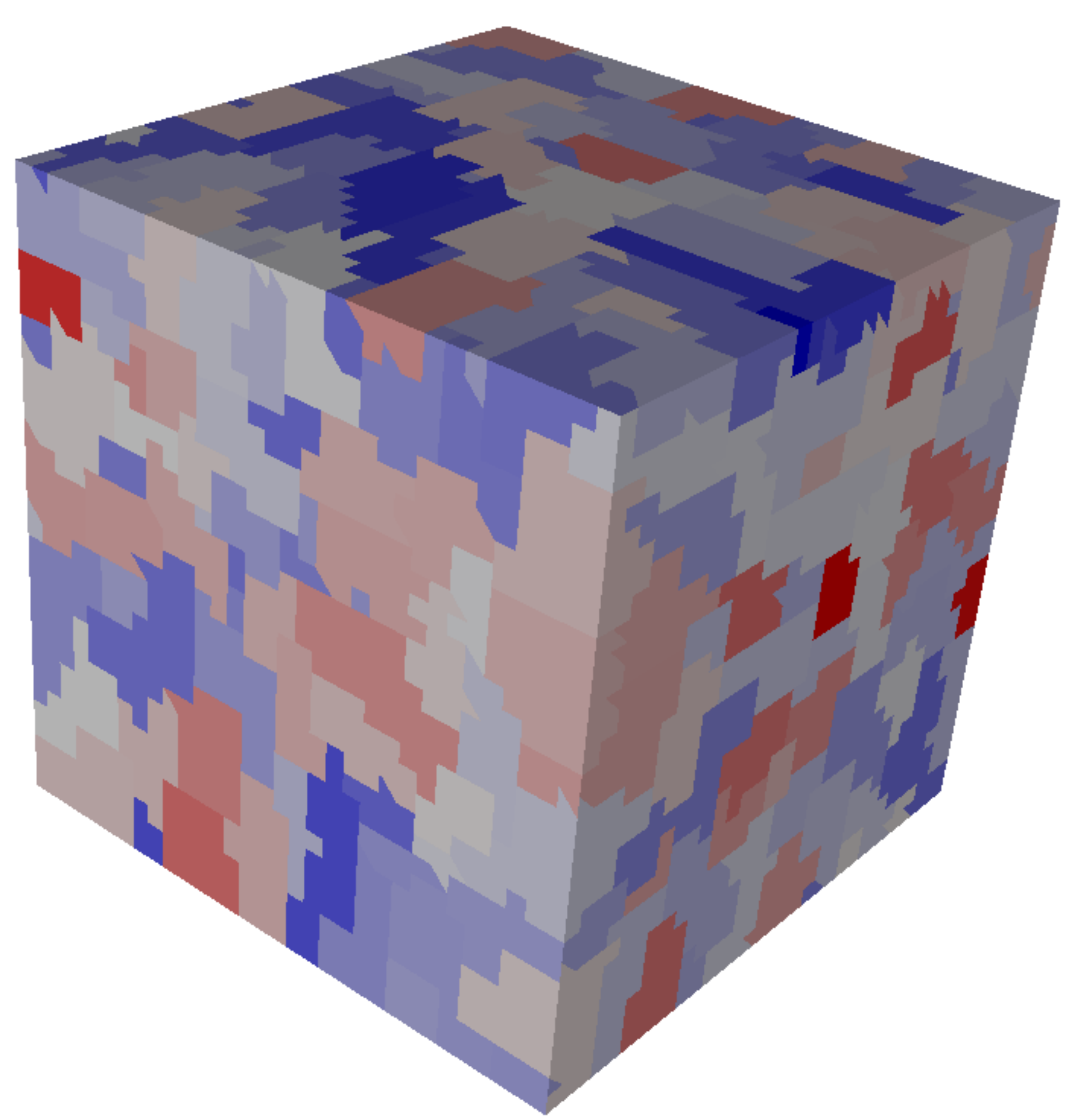}
    \subcaption{Agglomerated\label{fig:meshes:agglomerated}}
  \end{minipage}
  \caption{Examples of polytopal meshes in two and three space dimensions. The triangular and nonconforming meshes are taken from the FVCA5 benchmark~\cite{Herbin.Hubert:08}, the polygonal mesh family from~\cite[Section~4.2.3]{Di-Pietro.Lemaire:15}, and the agglomerated polyhedral mesh from~\cite{Di-Pietro.Specogna:16}.\label{fig:meshes}}
\end{figure}
\begin{definition}[Polytopal mesh]\label{def:mesh}
  A polytopal mesh of $\Omega$ is a couple $\Mh=(\Th,\Fh)$ where:
  
  \begin{inparaenum}[(i)]
  \item The set of \emph{mesh elements} $\Th$ is a finite collection of nonempty disjoint open polytopes $T$ with boundary $\partial T$ and diameter $h_T$
    such that the \emph{meshsize} $h$ satisfies $h=\max_{T\in\Th} h_T$ and it holds that
    $\closure{\Omega}=\bigcup_{T\in\Th}\closure{T}.$

  \item The set of \emph{mesh faces} $\Fh$ is a finite collection of disjoint subsets of $\overline{\Omega}$
  such that, for any $F\in\Fh$, $F$ is an open subset of a hyperplane of $\Real^d$,
  the $(d{-}1)$-dimensional Hausdorff measure
  of $F$ is strictly positive, and the $(d-1)$-dimensional Hausdorff measure
  of its relative interior $\closure{F}\backslash F$ is zero.
  Moreover,
  \begin{inparaenum}[(a)]
  \item for each $F\in \Fh$, either
    there exist two distinct mesh elements $T_1,T_2\in\Th $ such that $F\subset\partial T_1\cap\partial T_2$ and $F$ is called an \emph{interface} or 
    there exists one mesh element $T\in\Th$ such that $F\subset\partial T\cap\partial\Omega$ and $F$ is called a \emph{boundary face};
  \item the set of faces is a partition of the mesh skeleton, i.e.,
    $\bigcup_{T\in\Th}\partial T = \bigcup_{F\in\Fh}\closure{F}.$
  \end{inparaenum}
  \end{inparaenum}
\end{definition}
  
Interfaces are collected in the set $\Fhi$ and boundary faces in $\Fhb$, so that $\Fh=\Fhi\cup\Fhb$.  
For any mesh element $T\in\Th$, $$\Fh[T]\eqbydef\{F\in\Fh\st F\subset\partial T\}$$ denotes the set of faces contained in $\partial T$.
Similarly, for any mesh face $F\in\Fh$, $$\Th[F]\eqbydef\{T\in\Th\st F\subset\partial T\}$$ is the set of mesh elements sharing $F$.
Finally, for all $F\in\Fh[T]$, $\normal_{TF}$ is the unit normal vector to $F$ pointing out of $T$.

\begin{figure}\centering
  \includegraphics[height=3.5cm]{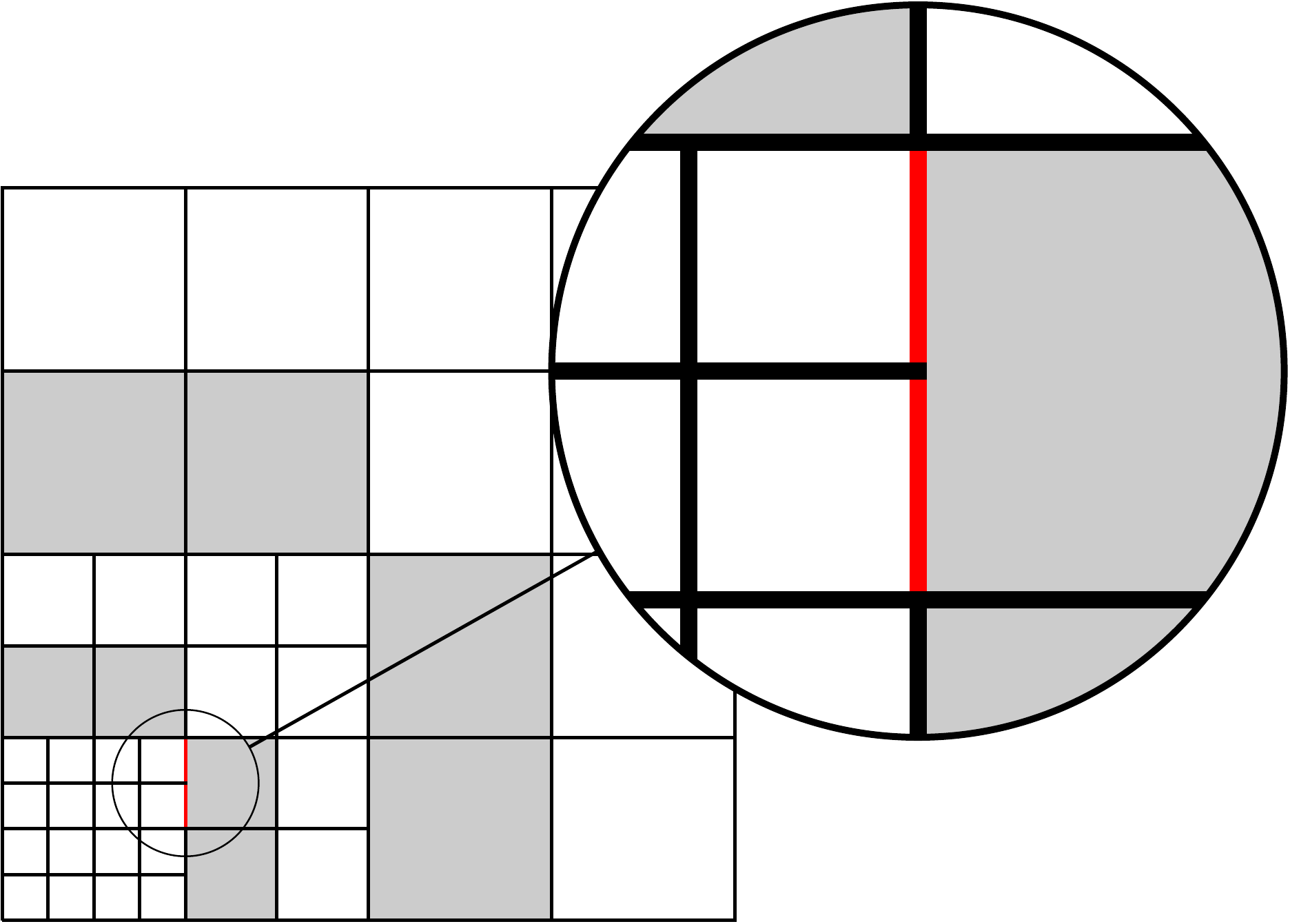}
  \caption{Treatment of a nonconforming junction (red) as multiple coplanar faces. Gray elements are pentagons with two coplanar faces, white elements are squares.\label{fig:nonconforming}}
\end{figure}
\begin{remark}[Nonconforming junctions]
  Meshes including nonconforming junctions such as the one depicted in Fig.~\ref{fig:nonconforming} are naturally supported provided that each face containing hanging nodes is treated as multiple coplanar faces.
\end{remark}

\subsection{Regular mesh sequences}\label{sec:setting:regular.mesh}

When studying the convergence of HHO methods with respect to the meshsize $h$, one needs to make assumptions on how the mesh is refined.
The ones provided here are closely inspired by~\cite[Chapter~1]{Di-Pietro.Ern:12}, and refer to the case of isotropic meshes with non-degenerate faces.
Isotropic means here that we do not consider the case of elements that become more and more stretched when refining.
Non-degenerate faces means, on the other hand, that the diameter of each mesh face is uniformly comparable to that of the element(s) it belongs to; see~\eqref{eq:hT_hF} below.

\begin{definition}[Matching simplicial submesh]\label{def:matching.submesh}
  Let $\Mh=(\Th,\Fh)$ be a polytopal mesh of $\Omega$. We say that $\fTh$ is a matching simplicial submesh of $\Mh$ if
  \begin{inparaenum}[(i)]
  \item $\fTh$ is a matching simplicial mesh of $\Omega$;
  \item for all simplices $\tau\in\fTh$, there is only one mesh element $T\in\Th$ such that $\tau\subset T$;
  \item for all $\sigma\in\fFh$, the set collecting the simplicial faces of $\fTh$, there is at most one face $F\in\Fh$ such that $\sigma\subset F$.
  \end{inparaenum}
\end{definition}

If $\Th$ itself is matching simplicial and $\Fh$ collects the corresponding simplicial faces, we can simply take $\fTh=\Th$, so that $\fFh=\Fh$.
The notion of regularity for refined mesh sequences is made precise by the following
\begin{definition}[Regular mesh sequence]\label{def:mesh.reg}
  Denote by ${\cal H}\subset \Real_*^+ $ a countable set of meshsizes having $0$ as its unique accumulation point.
  A sequence of refined meshes $(\Mh)_{h \in {\cal H}}$ is said to be \emph{regular} if there exists a real number $\varrho\in(0,1)$ such that, for all $h\in{\cal H}$, there exists a matching simplicial submesh $\fTh$ of $\Mh$ and 
  \begin{inparaenum}[(i)]
  \item for all simplices $\tau\in\fTh$ of diameter $h_{\tau}$ and inradius $r_{\tau}$, $\varrho h_{\tau}\le r_{\tau}$;
  \item for all mesh elements $T\in\Th$ and all simplices $\tau\in\fTh$ such that $\tau\subset T$, $\varrho h_T \le h_{\tau}$.
  \end{inparaenum}
\end{definition}
\begin{remark}[Role of the simplicial submesh]
  The simplicial submesh introduced in Definition~\ref{def:mesh.reg} is merely a theoretical tool, and needs not be constructed in practice.
\end{remark}

Geometric bounds on regular mesh sequences can be proved as in~\cite[Section 1.4.2]{Di-Pietro.Ern:12} (the definition of mesh face is slightly different therein since planarity is not required, but the proofs are based on the matching simplicial submesh and one can check that they carry out unchanged).
We recall here, in particular, that the number of faces of one mesh element is uniformly bounded: There is $N_\partial\ge d+1$ such that
\begin{equation}\label{eq:Np}
  \max_{h\in{\cal H}}\max_{T\in\Th}\card(\Fh[T])\le N_\partial.
\end{equation}
Morevover, according to~\cite[Lemma~1.42]{Di-Pietro.Ern:12}, 
for all $h\in{\cal H}$, all $T \in\Th$, and all $F\in \Fh[T]$ 
\begin{equation}\label{eq:hT_hF}
  \varrho^2 h_T \le h_F \le h_T.
\end{equation}
Discrete functional analysis results for arbitrary-order methods on regular mesh sequences can be found in~\cite[Chapter~1]{Di-Pietro.Ern:12} and~\cite{Di-Pietro.Droniou:16,Di-Pietro.Droniou:17}.
We also refer the reader to~\cite{Eymard.Gallouet.ea:00} for a first theorization of discrete functional analysis in the context of lowest-order finite volume methods, as well as to the subsequent extensions of~\cite{Eymard.Gallouet.ea:10,Droniou.Eymard.ea:17}.

Throughout the rest of this work, it is tacitly understood that we work on regular mesh sequences.

\subsection{Local and broken spaces}\label{sec:setting:spaces}

Throughout the rest of this chapter, for any $X\subset\closure{\Omega}$, we denote by $(\cdot,\cdot)_X$ and $\norm[X]{{\cdot}}$ the standard $L^2(X)$-product and norm, with the convention that the subscript is omitted whenever $X=\Omega$.
The same notation is used for the vector-valued space $L^2(X)^d$.

Let now the set $X$ be a mesh element or face.
For an integer $l\ge 0$, we denote by $\Poly{l}(X)$ the space spanned by the restriction to $X$ of scalar-valued, $d$-variate polynomials of total degree $l$.
We note the following trace inequality (see~\cite[Lemma~1.46]{Di-Pietro.Ern:12}):
There is a real number $C>0$ only depending on $d$, $\varrho$, and $l$ such that, for all $h\in{\cal H}$, all $T\in\Th$, all $v\in\Poly{l}(T)$, and all $F\in\Fh[T]$,
\begin{equation}\label{eq:trace.disc}
  \norm[F]{v}\le Ch_T^{-\nicefrac12}\norm[T]{v}.
\end{equation}
At the global level, we define the broken polynomial space
$$
\Poly{l}(\Th)\eqbydef\left\{
v_h\in L^2(\Omega)\st\restrto{v_h}{T}\in\Poly{l}(T)\quad\forall T\in\Th
\right\}.
$$
Functions in $\Poly{l}(\Th)$ belong to the broken Sobolev space
$$
W^{1,1}(\Th)\eqbydef\left\{
v\in L^1(\Omega)\st\restrto{v}{T}\in W^{1,1}(T)\quad\forall T\in\Th
\right\}.
$$
We denote by $\GRADh:W^{1,1}(\Th)\to L^1(\Omega)^d$ the usual broken gradient operator such that, for all $v\in W^{1,1}(\Th)$,
$$
\restrto{(\GRADh v)}{T} = \GRAD\restrto{v}{T}\qquad\forall T\in\Th.
$$

\subsection{Projectors on local polynomial spaces}\label{sec:setting:projectors}
Projectors on local polynomial spaces play a key role in the design and analysis of HHO methods.

\subsubsection{$L^2$-orthogonal projector}
Let $X$ denote a mesh element or face.
The $L^2$-orthogonal projector (in short, $L^2$-projector) $\lproj[X]{l}:L^1(X)\to\Poly{l}(X)$ is defined as follows: For all $v\in L^1(X)$, $\lproj[X]{l}$ is the unique polynomial in $\Poly{l}(X)$ that satisfies
\begin{equation}\label{eq:lproj}
  (\lproj[X]{l}v-v,w)_X=0\qquad\forall w\in\Poly{l}(X).
\end{equation}
Existence and uniqueness of $\lproj[X]{l}v$ follow from the Riesz representation theorem in $\Poly{l}(X)$ for the standard $L^2(X)$-inner product.
Moreover, we have the following characterization:
$$
\lproj[X]{l} v = \argmin_{w\in\Poly{l}(X)}\norm[X]{w-v}^2.
$$
In what follows, we will also need the vector-valued $L^2$-projector denoted by $\vlproj[X]{l}$ and obtained by applying $\lproj[X]{l}$ component-wise.
The following $H^s$-boundedness result is a special case of~\cite[Corollary~3.7]{Di-Pietro.Droniou:16}:
For any $s\in\{0,\ldots,l+1\}$, there exists a real number $C>0$ depending only on $d$, $\varrho$, $l$, and $s$ such that, for all $h\in{\cal H}$, all $T\in\Th$, and all $v\in H^s(T)$,
\begin{equation}\label{eq:stab.lproj}
  \seminorm[H^s(T)]{\lproj[T]{l}v}\le C\seminorm[H^s(T)]{v}.
\end{equation}
At the global level, we denote by $\lproj{l}:L^1(\Omega)\to\Poly{l}(\Th)$ the $L^2$-projector on the broken polynomial space $\Poly{l}(\Th)$ such that, for all $v\in L^1(\Omega)$,
$$
\restrto{(\lproj{l} v)}{T}\eqbydef\lproj[T]{l}\restrto{v}{T}.
$$

\subsubsection{Elliptic projector}
For any mesh element $T\in\Th$, we also define the elliptic projector $\eproj{l}:W^{1,1}(T)\to\Poly{l}(T)$ as follows:
For all $v\in W^{1,1}(T)$, $\eproj{l}v$ is a polynomial in $\Poly{l}(T)$ that satisfies
\begin{subequations}\label{eq:eproj}
\begin{equation}\label{eq:eproj:1}
  (\GRAD(\eproj{l}v-v),\GRAD w)_T=0\qquad\forall w\in\Poly{l}(T).
\end{equation}
By the Riesz representation theorem in $\GRAD\Poly{l}(T)$ for the $L^2(T)^d$-inner product,
this relation defines a unique element $\GRAD\eproj{l}v$, and thus a polynomial
$\eproj{l}v$ up to an additive constant.
This constant is fixed by writing
\begin{equation}\label{eq:eproj:2}
  (\eproj{l}v-v,1)_T=0.
\end{equation}
\end{subequations}
Observing that~\eqref{eq:eproj:1} is trivially verified when $l=0$, it follows from~\eqref{eq:eproj:2} that $\eproj[T]{0}=\lproj[T]{0}$.
Finally, the following characterization holds:
$$
\eproj{l} v = \argmin_{w\in\Poly{l}(T),\,(w-v,1)_T=0}\norm[L^2(T)^d]{\GRAD(w-v)}^2.
$$

\subsubsection{Approximation properties}
On regular mesh sequences, both $\lproj[T]{l}$ and $\eproj[T]{l}$ have optimal approximation properties in $\Poly{l}(T)$, as summarized by the following result (for a proof, see Theorem~1, Theorem~2, and Lemma~13 in~\cite{Di-Pietro.Droniou:16}):
For any $\alpha\in\{0,1\}$ and $s\in\{\alpha,\ldots,l+1\}$, there exists a real number $C>0$  depending only on $d$, $\varrho$, $l$, $\alpha$, and $s$ such that, for all $h\in{\cal H}$, all $T\in\Th$, and all $v\in H^s(T)$, 
\begin{subequations}\label{eq:approx.approx.trace}
  \begin{equation}\label{eq:approx}
    \seminorm[H^m(T)]{v - \pi_T^{\alpha,l} v }
    \le 
    C h_T^{s-m} 
    \seminorm[H^s(T)]{v}
    \qquad \forall m \in \{0,\ldots,s\},
  \end{equation}
  and, if $s\ge 1$,
  \begin{equation}\label{eq:approx.trace}
    \seminorm[{H^m(\Fh[T])}]{v - \pi_T^{\alpha,l} v}
    \le 
    C h_T^{s-m-\frac12} 
    \seminorm[H^s(T)]{v}
    \qquad \forall m \in \{0,\ldots,s-1\},
  \end{equation}
  where $H^m(\Fh[T])\eqbydef\left\{ v\in L^2(\partial T)\st \restrto{v}{F}\in H^m(F)\quad\forall F\in\Fh[T]\right\}$.
\end{subequations}


\section{Basic principles of Hybrid High-Order methods}\label{sec:basics}

To fix the main ideas and notation, we study in this section the HHO discretization of the Poisson problem:
Find $u:\Omega\to\Real$ such that
\begin{subequations}\label{eq:poisson:strong}
  \begin{alignat}{2}\label{eq:poisson:strong:pde}
    -\LAPL u &= f &\qquad&\text{in $\Omega$},
    \\\label{eq:poisson:strong:bc}
    u &= 0 &\qquad&\text{on $\partial\Omega$},
  \end{alignat}
\end{subequations}
where $f\in L^2(\Omega)$ is a given volumetric source term.
More general boundary conditions can replace~\eqref{eq:poisson:strong:bc}, but we restrict the discussion to the homogeneous Dirichlet case for the sake of simplicity.
A detailed treatment of more general boundary conditions including also variable diffusion coefficients can be found in~\cite{Di-Pietro.Ern.ea:16}.

The starting point to devise a HHO discretization is the following weak formulation of problem~\eqref{eq:poisson:strong}: Find $u\in H_0^1(\Omega)$ such that
\begin{equation}\label{eq:poisson:weak}
  a(u,v) = (f,v)\qquad\forall v\in H_0^1(\Omega),
\end{equation}
where the bilinear form $a:H^1(\Omega)\times H^1(\Omega)\to\Real$ is such that
\begin{equation}\label{eq:poisson:a}
  a(u,v)\eqbydef (\GRAD u, \GRAD v).
\end{equation}
In what follows, the quantities $u$ and $-\GRAD u$ will be referred to, respectively, as the potential and the flux.

\subsection{Local construction}\label{sec:basics:local.construction}

Throughout this section, we fix a polynomial degree $k\ge 0$ and a mesh element $T\in\Th$.
We introduce the local ingredients underlying the HHO construction: the DOFs, the potential reconstruction operator, and the discrete counterpart of the restriction to $T$ of the global bilinear form $a$ defined by~\eqref{eq:poisson:a}.

\subsubsection{Computing the local elliptic projection from $L^2$-projections}\label{sec:basics:local:ibp}

Consider a function $v\in H^1(T)$.
We note the following integration by parts formula, valid for all $w\in C^\infty(\overline{T})$:
\begin{equation}\label{eq:ipp}
  (\GRAD v,\GRAD w)_T = -(v, \LAPL w)_T + \sum_{F\in\Fh[T]}(v, \GRAD w\SCAL\normal_{TF})_F.
\end{equation}
Specializing~\eqref{eq:ipp} to $w\in\Poly{k+1}(T)$, we obtain
\begin{subequations}\label{eq:ipp.poly}  
  \begin{equation}\label{eq:ipp.poly:1}
    (\GRAD\eproj[T]{k+1}v,\GRAD w)_T = -(\lproj[T]{k-1}v, \LAPL w)_T + \sum_{F\in\Fh[T]}(\lproj[F]{k} v, \GRAD w\SCAL\normal_{TF})_F,
  \end{equation}
  where we have used~\eqref{eq:eproj} to insert $\eproj[T]{k+1}$ into the left-hand side and~\eqref{eq:lproj} to insert $\lproj[T]{k-1}$ and $\lproj[F]{k}$ into the right-hand side after  observing that $\LAPL w\in\Poly{k-1}(T)\subset\Poly{k}(T)$ and $\restrto{(\GRAD w)}{F}\SCAL\normal_{TF}\in\Poly{k}(F)$ for all $F\in\Fh[T]$.
  Moreover, recalling~\eqref{eq:eproj:2} and using the definition~\eqref{eq:lproj} of the $L^2$-projector, we infer that
  \begin{equation}\label{eq:ipp.poly:2}
    (v-\lproj[T]{0}v,1)_T
    =(\eproj[T]{k+1}v-\lproj[T]{\max(0,k-1)}v,1)_T
    =0.
  \end{equation}
\end{subequations}
The relations~\eqref{eq:ipp.poly} show that computing the elliptic projection $\eproj[T]{k+1}v$ does not require a full knowledge of the function $v$. All that is required is
\begin{enumerate}[(i)]
\item $\lproj[T]{\max(0,k-1)}v$, the $L^2$-projection of $v$ on the polynomial space $\Poly{\max(0,k-1)}(T)$. Clearly, one could also choose $\lproj[T]{k}v$ instead, which has the advantage of not requiring a special treatment of the case $k=0$;
\item for all $F\in\Fh[T]$, $\lproj[F]{k}\restrto{v}{F}$, the $L^2$-projection of the trace of $v$ on $F$ on the polynomial space $\Poly{k}(F)$.
\end{enumerate}

\subsubsection{Local space of degrees of freedom}

The remark at the end of the previous section motivates the introduction of the following space of DOFs (see Fig.~\ref{fig:dofspho}):
\begin{equation}\label{eq:UT}
  \UT\eqbydef\Poly{k}(T)\times\left(
  \bigtimes_{F\in\Fh[T]}\Poly{k}(F)
  \right).
\end{equation}
\begin{figure}\centering
  \includegraphics{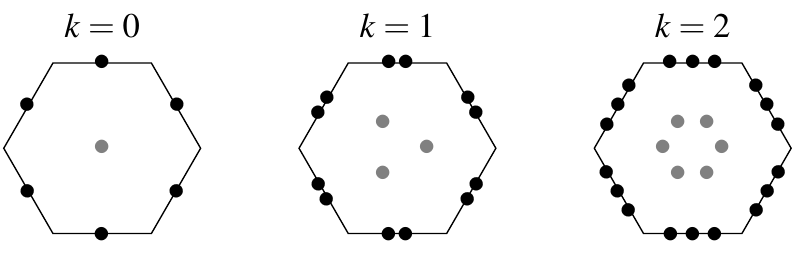}
  \caption{DOFs in $\UT$ for $k\in\{0,1,2\}$.\label{fig:dofspho}}
\end{figure}%

Observe that naming $\UT$ space of DOFs involves a shortcut: the actual DOFs can be chosen in several equivalent ways (polynomial moments, point values, etc.), and the specific choice does not affect the following discussion.
For a generic vector of DOFs in $\UT$, we use the underlined notation $\uv = (v_T, (v_F)_{F\in\Fh[T]})$.
On $\UT$, we define the $H^1$-like seminorm $\norm[1,T]{{\cdot}}$ such that, for all $\uv\in\UT$,
\begin{equation}\label{eq:norm.1T}
  \norm[1,T]{\uv}^2\eqbydef\norm[T]{\GRAD v_T}^2 + \seminorm[1,\partial T]{\uv}^2,\qquad
  \seminorm[1,\partial T]{\uv}^2\eqbydef\sum_{F\in\Fh[T]} h_F^{-1}\norm[F]{v_F-v_T}^2,
\end{equation}
where $h_F$ denotes the diameter of $F$.
The negative power of $h_F$ in the second term ensures that both contributions have the same scaling.
The DOFs corresponding to a smooth function $v\in W^{1,1}(T)$ are obtained via the reduction map $\IT:W^{1,1}(T)\to\UT$ such that
\begin{equation}\label{eq:IT}
  \IT v\eqbydef (\lproj[T]{k}v, (\lproj[F]{k}\restrto{v}{F})_{F\in\Fh[T]}).
\end{equation}

\subsubsection{Potential reconstruction operator}

Inspired by formula~\eqref{eq:ipp.poly}, we introduce the potential reconstruction operator $\pT:\UT\to\Poly{k+1}(T)$ such that, for all $\uv\in\UT$,
\begin{subequations}\label{eq:pT}
  \begin{equation}\label{eq:pT:1}
    (\GRAD\pT\uv,\GRAD w)_T = -(v_T,\LAPL w)_T + \sum_{F\in\Fh[T]}(v_F, \GRAD w\SCAL\normal_{TF})_F
    \quad\forall w\in\Poly{k+1}(T)
  \end{equation}
  and
  \begin{equation}\label{eq:pT:2}
    (\pT\uv-v_T,1)_T = 0.
  \end{equation}
\end{subequations}
Notice that $\pT\uv$ is a polynomial function on $T$ one degree higher than the element-based DOFs $v_T$.
By definition, for all $v\in W^{1,1}(T)$ it holds that
\begin{equation}\label{eq:pT:commuting}
  (\pT\circ\IT) v = \eproj[T]{k+1} v,
\end{equation}
i.e., the composition of the potential reconstruction operator with the reduction map gives the elliptic projector on $\Poly{k+1}(T)$.
An immediate consequence of~\eqref{eq:pT:commuting} together with~\eqref{eq:approx.approx.trace} is that $\pT\circ\IT$ has optimal approximation properties in $\Poly{k+1}(T)$.

\subsubsection{Local contribution}

We approximate the restriction $\restrto{a}{T}:H^1(T)\times H^1(T)\to\Real$ to $T$ of the continuous bilinear form $a$ defined by~\eqref{eq:poisson:a} by the discrete bilinear form $\mathrm{a}_T:\UT\times\UT\to\Real$ such that
\begin{equation}\label{eq:aT}
  \mathrm{a}_T(\uu,\uv) \eqbydef (\GRAD\pT\uu,\GRAD\pT\uv)_T + \mathrm{s}_T(\uu,\uv),
\end{equation}
where the first term in the right-hand side is the usual Galerkin contribution, while the second is a stabilization contribution for which we consider the following design conditions, originally proposed in~\cite{Boffi.Di-Pietro:16}:

\begin{assumption}[Local stabilization bilinear form $\mathrm{s}_T$]\label{ass:sT}
  The local stabilization bilinear form $\mathrm{s}_T:\UT\times\UT\to\Real$ satisfies the following properties:
  \begin{enumerate}[(S1)]
  \item \emph{Symmetry and positivity.} $\mathrm{s}_T$ is symmetric and positive semidefinite;
  \item \emph{Stability.} There is a real number $\eta>0$ independent of $h$ and of $T$, but possibly depending on $d$, $\varrho$, and $k$, such that
    \begin{equation}\label{eq:stab.T}
      \eta^{-1}\norm[1,T]{\uv}^2\le 
      \mathrm{a}_T(\uv,\uv)\le \eta\norm[1,T]{\uv}^2
      \qquad\forall \uv\in\UT;
    \end{equation}
  \item \emph{Polynomial consistency.} For all $w\in\Poly{k+1}(T)$ and all $\uv\in\UT$, it holds that
    \begin{equation}\label{eq:poly.cons.T}
      \mathrm{s}_T(\IT w,\uv) = 0.
    \end{equation}
  \end{enumerate}
\end{assumption}

These requirements suggest that $\mathrm{s}_T$ can be obtained penalizing in a least square sense residuals that vanish for reductions of polynomial functions in $\Poly{k+1}(T)$.
Paradigmatic examples of such residuals are provided by the operators $\delta_T^k:\UT\to\Poly{k}(T)$ and, for all $F\in\Fh[T]$, $\delta_{TF}^k:\UT\to\Poly{k}(F)$ such that, for all $\uv\in\UT$,
\begin{equation}\label{eq:residuals}
  \delta_T^k\uv\eqbydef\lproj[T]{k}(\pT\uv-v_T),\qquad
  \delta_{TF}^k\uv\eqbydef\lproj[F]{k}(\pT\uv-v_F)\quad\forall F\in\Fh[T].
\end{equation}
To check that $\delta_T^k$ vanishes when $\uv=\IT w$ with $w\in\Poly{k+1}(T)$, we observe that
$$
\delta_T^k\IT w
= \lproj[T]{k}(\pT\IT w - \lproj[T]{k} w)
= \lproj[T]{k}(\eproj[T]{k+1} w - w)
= \lproj[T]{k}(w - w) = 0,
$$
where we have used the definition of $\delta_T^k$ in the first equality, the relation~\eqref{eq:pT:commuting} to replace $\pT\IT$ by $\eproj[T]{k+1}$ and the fact that $\lproj[T]{k}w\in\Poly{k}(T)$ to cancel $\lproj[T]{k}$ from the second term in parentheses, and the fact that $\eproj[T]{k+1}$ leaves polynomials of total degree up to $(k+1)$ unaltered as a projector to conclude.
A similar argument shows that $\delta_{TF}^k\IT w=0$ for all $F\in\Fh[T]$ whenever $w\in\Poly{k+1}(T)$.

Accounting for dimensional homogeneity with the Galerkin term, one possible expression for $\mathrm{s}_T$ is thus
\begin{equation}\label{eq:sT.vem}
  \mathrm{s}_T(\uu,\uv)\eqbydef h_T^{-2}(\delta_T^k\uu,\delta_T^k\uv)_T + \sum_{F\in\Fh[T]}h_F^{-1}(\delta_{TF}^k\uu,\delta_{TF}^k\uv)_F.
\end{equation}
This choice, inspired by the Virtual Element literature \cite{Beirao-da-Veiga.Brezzi.ea:13}, differs from the original HHO stabilization of~\cite{Di-Pietro.Ern.ea:14}, where the following expression is considered instead:
\begin{equation}\label{eq:sT.hho}
  \mathrm{s}_T(\uu,\uv)\eqbydef \sum_{F\in\Fh[T]}h_F^{-1}(\delta_{TF}^k\uu-\delta_T^k\uu,\delta_{TF}^k\uv-\delta_T^k\uv)_F.
\end{equation}
In this case, only quantities at faces are penalized.
Both of the above expressions match the design conditions (S1)--(S3) and are essentially equivalent in terms of implementation.
A detailed proof for $\mathrm{s}_T$ as in~\eqref{eq:sT.hho} can be found in~\cite[Lemma~4]{Di-Pietro.Ern.ea:14}.
Yet another example of stabilization bilinear form used in the context of HHO methods is provided by~\cite[Eq.~(3.24)]{Aghili.Boyaval.ea:15}.
This expression results from the hybridization of the Mixed High-Order method of~\cite{Di-Pietro.Ern:16}.

\begin{remark}[Original HDG stabilization]
  The following stabilization bilinear form is used in the original Hybridizable Discontinuous Galerkin (HDG) method of \cite{Castillo.Cockburn.ea:00,Cockburn.Gopalakrishnan.ea:09}:
  $$
  \mathrm{s}_T(\uu,\uv) 
  = \sum_{F\in\Fh[T]}h_F^{-1} (u_F-u_T,v_F-v_T)_F.
  $$
  While this choice obviously satisfies the properties (S1)-(S2), it fails to satisfy (S3) (it is only consistent for polynomials of degree up to $k$).
  As a result, up to one order of convergence is lost with respect to the estimates of Theorems \ref{thm:poisson:en.err.est} and \ref{thm:poisson:l2.err.est} below.
  For a discussion including fixes that restore optimal orders of convergence in HDG see \cite{Cockburn.Di-Pietro.ea:16}.
\end{remark}

\subsubsection{Consistency properties of the stabilization for smooth functions}
In the following proposition we study the consistency properties of $\mathrm{s}_T$ when its arguments are reductions of a smooth function.
We give a detailed proof since this result is a new extension of the bound in~\cite[Theorem~8]{Di-Pietro.Ern.ea:14} (see, in particular, Eq.~(45) therein) to more general stabilization bilinear forms.
\begin{proposition}[Consistency of $\mathrm{s}_T$]\label{prop:approx.sT}
  Let $\mathrm{s}_T$ denote a stabilization bilinear form satisfying assumptions (S1)--(S3).
  Then, there is a real number $C>0$ independent of $h$, but possibly depending on $d$, $\varrho$, and $k$, such that, for all $T\in\Th$ and all $v\in H^{k+2}(T)$, it holds that
  \begin{equation}\label{eq:approx.sT}
    \mathrm{s}_T(\IT v, \IT v)^{\nicefrac12}\le C h_T^{k+1}\norm[H^{k+2}(T)]{v}.
  \end{equation}
\end{proposition}
\begin{proof}
  We set, for the sake of brevity, $\cv\eqbydef\eproj[T]{k+1}v$ and abridge as $A\lesssim B$ the inequality $A\le cB$ with multiplicative constant $c>0$ having the same dependencies as $C$ in~\eqref{eq:approx.sT}.
  Using (S2) and (S3) we infer that
  \begin{equation}\label{eq:approx.sT:0}
    \mathrm{s}_T(\IT v, \IT v)^{\nicefrac12}
    = \mathrm{s}_T(\IT (v-\cv), \IT (v-\cv))^{\nicefrac12}
    \le\eta\norm[1,T]{\IT (v-\cv)}.
  \end{equation}
  Recalling~\eqref{eq:norm.1T}, we have that
  \begin{multline}\label{eq:approx.sT:1}
    \norm[1,T]{\IT (v-\cv)}^2
    = \\
    \norm[T]{\GRAD\lproj[T]{k}(v-\cv)}^2
    + \sum_{F\in\Fh[T]}h_F^{-1}\norm[F]{\lproj[F]{k}(v-\cv-\lproj[T]{k}(v-\cv))}^2.
  \end{multline}
  Using the $H^1(T)$-boundedness of $\lproj[T]{k}$ resulting from~\eqref{eq:stab.lproj} with $l=k$, and $s=1$ followed by the optimal approximation properties~\eqref{eq:approx} of $\cv$ (with $\alpha=1$, $l=k+1$, $s=k+2$, and $m=1$), it is inferred that
  \begin{equation}\label{eq:approx.sT:2}
    \norm[T]{\GRAD\lproj[T]{k}(v-\cv)}\lesssim\norm[T]{\GRAD(v-\cv)}\lesssim h^{k+1}\norm[H^{k+2}(T)]{v}.
  \end{equation}
  On the other hand, for all $F\in\Fh[T]$ it holds that
  \begin{equation}\label{eq:approx.sT:3}
    \begin{aligned}
      h_F^{-\nicefrac12}\norm[F]{\lproj[F]{k}(v-\cv-\lproj[T]{k}(v-\cv))}
      &\lesssim h_T^{-1}\norm[T]{v-\cv-\lproj[T]{k}(v-\cv)}
      \\
      &\lesssim \norm[T]{\GRAD(v-\cv)}
      \\
      &\lesssim h_T^{k+1}\norm[H^{k+2}(T)]{v},
    \end{aligned}
  \end{equation}
  where we have used the $L^2(F)$-boundedness of $\lproj[F]{k}$ together with~\eqref{eq:hT_hF} and the discrete trace inequality~\eqref{eq:trace.disc} in the first line, a local Poincar\'{e} inequality resulting from~\eqref{eq:approx} with $\alpha=0$, $l=k$, $s=1$, and $m=0$  to pass to the second line, and the optimal approximation properties of $\cv$ expressed by~\eqref{eq:approx} with $\alpha=1$, $l=k+1$, $s=k+2$, and $m=1$ to conclude.
  Plugging~\eqref{eq:approx.sT:2} and~\eqref{eq:approx.sT:3} into~\eqref{eq:approx.sT:1}, recalling that $\card(\Fh[T])\lesssim 1$ (see~\eqref{eq:Np}), and using the resulting bound to estimate~\eqref{eq:approx.sT:0},~\eqref{eq:approx.sT} follows.\qed
\end{proof}

\subsection{Discrete problem}\label{sec:basics:discrete.problem}
We now show how to formulate the discrete problem from the local contributions introduced in the previous section.

\subsubsection{Global spaces of degrees of freedom}
We define the following global space of DOFs with single-valued interface unknowns:
$$
\Uh\eqbydef\left(\bigtimes_{T\in\Th}\Poly{k}(T)\right)\times\left(\bigtimes_{F\in\Fh}\Poly{k}(F)\right).
$$
Notice that single-valued means here that interface values match from one element to the adjacent one.
For a generic element $\uv[h]\in\Uh$, we use the underlined notation $\uv[h]=((v_T)_{T\in\Th},(v_F)_{F\in\Fh})$ and, for all $T\in\Th$, we denote by $\uv[T]=(v_T,(v_F)_{F\in\Fh[T]})\in\UT$ its restriction to $T$.
We also define the broken polynomial function $v_h\in\Poly{k}(\Th)$ such that
$$
\restrto{v_h}{T}\eqbydef v_T\qquad\forall T\in\Th.
$$
The DOFs corresponding to a smooth function $v\in W^{1,1}(\Omega)$ are obtained via the reduction map $\Ih:W^{1,1}(\Omega)\to\Uh$ such that
$$
\Ih v\eqbydef ((\lproj[T]{k}\restrto{v}{T})_{T\in\Th}, (\lproj[F]{k}\restrto{v}{F})_{F\in\Fh}).
$$
We define on $\Uh$ the seminorm $\norm[1,h]{{\cdot}}$ such that, for all $\uv[h]\in\Uh$,
\begin{equation}\label{eq:norm1h}
  \norm[1,h]{\uv[h]}^2\eqbydef\sum_{T\in\Th}\norm[1,T]{\uv}^2,
\end{equation}
with local seminorm $\norm[1,T]{{\cdot}}$ defined by~\eqref{eq:norm.1T}.
To account for the homogeneous Dirichlet boundary condition~\eqref{eq:poisson:strong:bc} in a strong manner, we introduce the subspace
$$
\UhD\eqbydef\left\{
\uv[h]\in\Uh\st v_F\equiv 0\quad\forall F\in\Fhb
\right\}.
$$
We recall the following discrete Poincar\'{e} inequality proved in~\cite[Proposition~5.4]{Di-Pietro.Droniou:16}:
There exists a real number $C_{\rm P}>0$ independent of $h$, but possibly depending on $\Omega$, $\varrho$, and $k$, such that, for all $\uv[h]\in\UhD$,
\begin{equation}\label{eq:poincare}
  \norm{v_h}\le C_{\rm P}\norm[1,h]{\uv[h]}.
\end{equation}
\begin{proposition}[{Norm $\norm[1,h]{{\cdot}}$}]
  The map $\norm[1,h]{{\cdot}}$ defines a norm on $\UhD$.
\end{proposition}
\begin{proof}
  The seminorm property being evident, it suffices to prove that, for all $\uv[h]\in\UhD$, $\norm[1,h]{\uv[h]}=0\implies\uv[h]=\underline{0}_h$.
  Let $\uv[h]\in\UhD$ be such that $\norm[1,h]{\uv[h]}=0$.
  By~\eqref{eq:poincare}, we have $\norm{v_h}=0$, hence $v_T\equiv 0$ for all $T\in\Th$.
  From the definition~\eqref{eq:norm.1T} of the norm $\norm[1,T]{{\cdot}}$, we also have that $\norm[F]{v_F-v_T}=0$ for all $T\in\Th$ and all $F\in\Fh[T]$, hence $v_F=v_T\equiv 0$.
  Since any mesh face belongs to the set $\Fh[T]$ for at least one mesh element $T\in\Th$, this concludes the proof.\qed
\end{proof}
\subsubsection{Global bilinear form}
We define the global bilinear forms $\mathrm{a}_h:\Uh\times\Uh\to\Real$ and $\mathrm{s}_h:\Uh\times\Uh\to\Real$ by element-by-element assembly setting, for all $\uu[h],\uv[h]\in\Uh$,
\begin{equation}\label{eq:ah}
  \mathrm{a}_h(\uu[h],\uv[h])\eqbydef\sum_{T\in\Th}\mathrm{a}_T(\uu,\uv),\qquad
  \mathrm{s}_h(\uu[h],\uv[h])\eqbydef\sum_{T\in\Th}\mathrm{s}_T(\uu,\uv).
\end{equation}
\begin{lemma}[Properties of $\mathrm{a}_h$]
  The bilinear form $\mathrm{a}_h$ enjoys the following properties:
  \begin{enumerate}[(i)]
  \item \emph{Stability.} For all $\uv[h]\in\UhD$ it holds with $\eta$ as in~\eqref{eq:stab.T} that
    \begin{equation}\label{eq:stab.h}
      \eta^{-1}\norm[1,h]{\uv[h]}^2\le\norm[\mathrm{a},h]{\uv[h]}^2\eqbydef\mathrm{a}_h(\uv[h],\uv[h])\le\eta\norm[1,h]{\uv[h]}^2.
    \end{equation}
  \item \emph{Consistency.} There is a real number $C>0$ independent of $h$, but possibly depending on $d$, $\varrho$, and $k$, such that, for all $w\in H_0^1(\Omega)\cap H^{k+2}(\Omega)$,
    \begin{equation}\label{eq:consistency.h}
      \sup_{\uv[h]\in\UhD,\norm[1,h]{\uv[h]}=1}{\cal E}_h(w;\uv[h])\le C h^{k+1}\norm[H^{k+2}(\Omega)]{w},
    \end{equation}
    with linear form ${\cal E}_h(w;\cdot):\Uh\to\Real$ representing the conformity error such that, for all $\uv[h]\in\Uh$,
    \begin{equation}\label{eq:Eh}
      {\cal E}_h(w;\uv[h])\eqbydef -(\LAPL w,v_h) - \mathrm{a}_h(\Ih w, \uv[h]).
    \end{equation}
  \end{enumerate}
\end{lemma}
\begin{proof}
  \begin{asparaenum}[(i)]
  \item \emph{Stability.} Summing inequalities~\eqref{eq:stab.T} over $T\in\Th$,~\eqref{eq:stab.h} follows.
  \item \emph{Consistency.}
    Let $\uv[h]\in\UhD$ be such that $\norm[1,h]{\uv[h]}=1$.
    Throughout the proof, we abridge as $A\lesssim B$ the inequality $A\le cB$ with multiplicative constant $c>0$ having the same dependecies as $C$ in~\eqref{eq:consistency.h}.
    For the sake of brevity, we also let $\cw\eqbydef\pT\IT w=\eproj[T]{k+1}w$ (cf.~\eqref{eq:pT:commuting}) for all $T\in\Th$.
    Integrating by parts element-by-element, we infer that
    \begin{equation}\label{eq:ah:consistency:1}
      -(\LAPL w,v_h) = \sum_{T\in\Th}\left(
      (\GRAD w,\GRAD v_T)_T + \sum_{F\in\Fh[T]}(\GRAD w\SCAL\normal_{TF}, v_F-v_T)_F
      \right).
    \end{equation}
    To insert $v_F$ into the second term in parentheses in~\eqref{eq:ah:consistency:1}, we have used the fact that $v_F\equiv 0$ for all $F\in\Fhb$ while, for all $F\in\Fhi$ such that $F\subset\partial T_1\cap\partial T_2$ for distinct mesh elements $T_1,T_2\in\Th$,
    $\restrto{(\GRAD w)}{T_1}\SCAL\normal_{T_1F} + \restrto{(\GRAD w)}{T_2}\SCAL\normal_{T_2F} = 0$ (since $w\in H^{k+2}(\Omega)$), so that
    $$
    \sum_{T\in\Th}\sum_{F\in\Fh[T]}(\GRAD w\SCAL\normal_{TF},v_F)_F
    = \sum_{F\in\Fhi}(\sum_{T\in\Th[F]}\restrto{(\GRAD w)}{T}\SCAL\normal_{TF},v_F)_F
    + \sum_{F\in\Fhb}(\GRAD w\SCAL\normal, v_F)_F
    = 0.
    $$
    On the other hand, plugging the definition~\eqref{eq:aT} of $\mathrm{a}_T$ into~\eqref{eq:ah}, and expanding $\pT\uv$ according to~\eqref{eq:pT} with $w=\cw$, it is inferred that
    \begin{equation}\label{eq:ah:consistency:2}
      \mathrm{a}_h(\Ih w,\uv[h]) = \sum_{T\in\Th}\left(
      (\GRAD\cw,\GRAD v_T)_T
      + \sum_{F\in\Fh[T]}(\GRAD\cw\SCAL\normal_{TF}, v_F-v_T)_F + \mathrm{s}_T(\IT w,\uv)
      \right).
    \end{equation}
    Subtracting~\eqref{eq:ah:consistency:2} from~\eqref{eq:ah:consistency:1}, using the definition~\eqref{eq:eproj} of $\eproj[T]{k+1}$ to cancel the first terms in parentheses, and taking absolute values, we get
    $$
    \begin{aligned}
      |{\cal E}_h(w;\uv[h])|
      &=\left| \sum_{T\in\Th}\left(
      \sum_{F\in\Fh[T]}(\GRAD(w-\cw)\SCAL\normal_{TF}, v_F-v_T)_F + \mathrm{s}_T(\IT w,\uv)
      \right)
      \right|
      \\
      &\le\left[\sum_{T\in\Th}\left(
      h_T\norm[\partial T]{\GRAD(w-\cw)}^2 + \mathrm{s}_T(\IT w,\IT w)
      \right)
      \right]^{\nicefrac12}
      \\
      &\qquad\times\left[\sum_{T\in\Th}\left(
      \seminorm[1,\partial T]{\uv}^2 + \mathrm{s}_T(\uv,\uv)
      \right)\right]^{\nicefrac12}.
    \end{aligned}
    $$
    Using~\eqref{eq:approx.trace} with $\alpha=1$, $l=k+1$, $s=k+2$, and $m=1$ together with~\eqref{eq:approx.sT} for the first factor, and the seminorm equivalence~\eqref{eq:stab.T} together with the fact that $\norm[1,h]{\uv[h]}=1$ for the second, we infer the bound
    $$
    |{\cal E}_h(w;\uv[h])|\lesssim h^{k+1} \norm[H^{k+2}(\Omega)]{w}.
    $$
    Since $\uv[h]$ is arbitrary, this yields~\eqref{eq:consistency.h}.\qed
  \end{asparaenum}
\end{proof}
\subsubsection{Discrete problem and well-posedness}
The discrete problem reads:
Find $\uu[h]\in\UhD$ such that
\begin{equation}\label{eq:poisson:discrete}
  \mathrm{a}_h(\uu[h],\uv[h]) = (f,v_h)\qquad\forall\uv[h]\in\UhD.
\end{equation}

\begin{lemma}[Well-posedness]
  Problem~\eqref{eq:poisson:discrete} is well-posed, and we have the following a priori bound for the unique discrete solution $\uu[h]\in\UhD$:
  $$
  \norm[1,h]{\uu[h]}\le\eta C_{\rm P}\norm{f}.
  $$
\end{lemma}
\begin{proof}
  We check the assumptions of the Lax--Milgram lemma~\cite{Lax.Milgram:54}
  on the finite-dimensional space $\UhD$ equipped with the norm $\norm[1,h]{{\cdot}}$.
  The bilinear form $\mathrm{a}_h$ is coercive and continuous owing to~\eqref{eq:stab.h} with coercivity constant equal to $\eta^{-1}$.
  The linear form $\uv[h]\mapsto (f,v_h)$ is continuous owing to~\eqref{eq:poincare} with continuity constant equal to $C_{\rm P}$.
  \qed
\end{proof}

\subsubsection{Implementation}\label{sec:basics:local:implementation}

Let a basis $\mathcal{B}_h$ for the space $\UhD$ be fixed such that every basis function is supported by only one mesh element or face.
For a generic element $\uv[h]\in\UhD$, denote by $\sV$ the corresponding vector of coefficients in $\mathcal{B}_h$ partitioned as
$$\sV=\begin{pmat}[{}]\sV[\Th]\cr\- \sV[\Fh]\cr\end{pmat},$$
where the subvectors $\sV[\Th]$ and $\sV[\Fh]$ collect the coefficients associated to element-based and face-based DOFs, respectively.
Denote by $\sA$ the matrix representation of the bilinear form $\mathrm{a}_h$ and by $\sB$ the vector representation of the linear form $\uv[h]\mapsto (f,v_h)$, both partitioned in a similar way.
The algebraic problem corresponding to~\eqref{eq:poisson:discrete} reads
\begin{equation}\label{eq:poisson:algebraic}
  \underbrace{\begin{pmat}[{|}]
      \sA[\Th\Th] & \sA[\Th\Fh] \cr\-
      \sA[\Th\Fh]\trans & \sA[\Fh\Fh] \cr
  \end{pmat}}_{\sA}
  \underbrace{\begin{pmat}[{}]
      \sU[\Th]\cr\-\sU[\Fh]\vphantom{\sU[\Fh]\trans}\cr
  \end{pmat}}_{\sU}=
  \underbrace{
  \begin{pmat}[{}]
      \sB[\Th]\cr\-\mathsf{0}_{\Fh}\vphantom{\sU[\Fh]\trans}\cr
    \end{pmat}}_{\sB}.
\end{equation}
The submatrix $\sA[\Th\Th]$ is block-diagonal and symmetric positive definite, and is therefore inexpensive to invert.
In the practical implementation, this remark can be exploited by solving the linear system~\eqref{eq:poisson:algebraic} in two steps (see, e.g.,~\cite[Section 2.4]{Cockburn.Di-Pietro.ea:16}):
\begin{subequations}
  \begin{enumerate}[(i)]
  \item First, element-based coefficients in $\sU[\Th]$ are expressed in terms of $\sB[\Th]$ and $\sU[\Fh]$ by the inexpensive solution of the first block equation:
    \begin{equation}\label{eq:static.cond:1}
      \sU[\Th]=\sA[\Th\Th]^{-1}\left(
      \sB[\Th] - \sA[\Th\Fh]\sU[\Fh]
      \right).
    \end{equation}      
    This step is referred to as {\em static condensation} in the finite element literature;
  \item Second, face-based coefficients in $\sU[\Fh]$ are obtained solving the global skeletal (i.e., involving unknowns attached to the mesh skeleton) problem
    \begin{equation}\label{eq:static.cond:2}
      \left(\sA[\Fh\Fh]-\sA[\Th\Fh]\trans\sA[\Th\Th]^{-1}\sA[\Th\Fh]\right)
      \sU[\Fh]
      =
      \sA[\Th\Fh]\trans\sA[\Th\Th]^{-1}\sB[\Th].
    \end{equation}
    This computationally more intensive step requires to invert the matrix in parentheses in the above expression.
    This symmetric positive definite matrix, whose stencil is the same as that of $\sA[\Fh\Fh]$ and only involves neighbours through faces, has size $N_{\rm dof}\times N_{\rm dof}$ with
    \begin{equation}\label{eq:Ndof}
      N_{\rm dof} = \card(\Fhi) \times {k + d - 1\choose k}.
    \end{equation}
  \end{enumerate}
\end{subequations}

\subsubsection{Local conservation and flux continuity}

At the continuous level, the solution of problem~\eqref{eq:poisson:weak} satisfies the following local balance for all $T\in\Th$ and all $v_T\in\Poly{k}(T)$:
\begin{subequations}\label{eq:balance.equilibrium}
  \begin{equation}\label{eq:balance}
    (\GRAD u,\GRAD v_T)_T - \sum_{F\in\Fh[T]}(\GRAD u\SCAL\normal_{TF}, v_T)_F = (f, v_T)_T,
  \end{equation}
  and the normal flux traces are continuous in the sense that, for all $F\in\Fhi$ such that $F\subset\partial T_1\cap \partial T_2$ with distinct mesh elements $T_1,T_2\in\Th$, it holds (see, e.g.,~\cite[Lemma~4.3]{Di-Pietro.Ern:12})
  \begin{equation}\label{eq:equilibrium}
    \restrto{(\GRAD u)}{T_1}\SCAL\normal_{T_1F} + \restrto{(\GRAD u)}{T_2}\SCAL\normal_{T_2F}=0.
  \end{equation}
\end{subequations}
We show in this section that a discrete counterpart of the relations~\eqref{eq:balance.equilibrium} holds for the discrete solution.
This property is relevant both from the engineering and mathematical points of view, and it can be exploited to derive a posteriori error estimators by flux equilibration.
It was originally highlighted in~\cite{Di-Pietro.Ern:15*1} and, using different techniques, in~\cite{Cockburn.Di-Pietro.ea:16} for the stabilization bilinear form $\mathrm{s}_T$ defined by~\eqref{eq:sT.hho}.
Here, using yet a different approach, we extend these results to more general stabilization bilinear forms.

Let a mesh element $T\in\Th$ be fixed.
We define the space
\begin{equation}\label{eq:DT}
  \DT\eqbydef\bigtimes_{F\in\Fh[T]}\Poly{k}(F),
\end{equation}
as well as the boundary difference operator $\DpT:\UT\to\DT$ such that, for all $\uv\in\UT$,
\begin{equation}\label{eq:DpT}
\DpT\uv = (\Delta_{TF}^k\uv)_{F\in\Fh[T]} \eqbydef (v_F - \restrto{v_T}{F})_{F\in\Fh[T]}.
\end{equation}
A useful remark is that, for all $\uv\in\UT$, it holds
\begin{equation}\label{eq:flux.magic}
  \uv-\IT v_T
  = (v_T - \lproj[T]{k}v_T, (v_F - \lproj[F]{k}\restrto{v_T}{F})_{F\in\Fh[T]})
  = (0,\DpT\uv),
\end{equation}
where the conclusion follows observing that, for all $T\in\Th$ and all $F\in\Fh[T]$, $\lproj[T]{k}v_T=v_T$ and $\lproj[F]{k}\restrto{v_T}{F}=\restrto{v_T}{F}$ since $v_T\in\Poly{k}(T)$ and $\restrto{v_T}{F}\in\Poly{k}(F)$.

We show in the next proposition that any stabilization bilinear form with a suitable dependence on its arguments can be reformulated in terms of boundary differences.
\begin{proposition}[Reformulation of the stabilization bilinear form]\label{prop:sT'}
  Let $T\in\Th$, and assume that $\mathrm{s}_T$ is a stabilization bilinear form that satisfies assumptions (S1)--(S3) and that depends on its arguments only through the residuals defined by~\eqref{eq:residuals}.
  Then, it holds for all $\uu,\uv\in\UT$ that
  \begin{equation}\label{eq:sT'}
    \mathrm{s}_T(\uu,\uv) =
    \mathrm{s}_T((0,\DpT\uu),(0,\DpT\uv)).
  \end{equation}
\end{proposition}
\begin{proof}
  It suffices to show that, for all $\uv\in\UT$,
  $$
  \delta_T^k\uv=\delta_T^k(0,\DpT\uv),\qquad
  \delta_{TF}^k\uv=\delta_{TF}^k(0,\DpT\uv)\quad\forall F\in\Fh[T].
  $$
  Let us start by $\delta_T^k$.
  Since $v_T\in\Poly{k}(T)$, $\pT\IT v_T = \eproj[T]{k+1}v_T = v_T$.
  Hence,
  $$
  \begin{aligned}
    \delta_T^k\uv
    &= \lproj[T]{k}(\pT\uv - v_T)
    \\
    &= \lproj[T]{k}(\pT\uv - \pT\IT v_T)
    \\
    &= \lproj[T]{k}\pT(\uv - \IT v_T)
    = \delta_T^k(0,\DpT\uv),
  \end{aligned}
  $$
  where we have used the linearity of $\pT$ to pass to the third line and~\eqref{eq:flux.magic} to conclude.
  Let now $F\in\Fh[T]$ and consider $\delta_{TF}^k$. We have
  $$
  \begin{aligned}
    \delta_{TF}^k\uv
    &= \lproj[F]{k}(\pT\uv - v_F)
    \\
    &= \lproj[F]{k}(\pT\uv - \pT\IT v_T + v_T - v_F)
    \\
    &= \lproj[F]{k}(\pT(0,\DpT\uv) - \Delta_{TF}^k\uv)
    = \delta_{TF}^k(0,\DpT\uv),
  \end{aligned}
  $$
  where we have introduced $v_T - \pT\IT v_T = 0$ in the second line, used the linearity of $\pT$ together with~\eqref{eq:flux.magic} and the definition~\eqref{eq:DT} of $\DpT$ in the third line, and concluded recalling the definition~\eqref{eq:residuals} of $\delta_{TF}^k$. \qed
\end{proof}
Define the boundary residual operator $\RpT:\UT\to\DT$ such that, for all $\uv\in\UT$, $\RpT\uv = (R_{TF}^k\uv)_{F\in\Fh[T]}$ satisfies for all $\ual=(\alpha_{TF})_{F\in\Fh[T]}\in\DT$
\begin{equation}\label{eq:RpT}
  -\sum_{F\in\Fh[T]}(R_{TF}^k\uv,\alpha_{TF})_F = \mathrm{s}_T((0,\DpT\uv),(0,\ual)).
\end{equation}
Problem~\eqref{eq:RpT} is well-posed, and computing $R_{TF}^k\uv$ requires to invert the boundary mass matrix.

\begin{svgraybox}
  \begin{lemma}[Local balance and flux continuity]\label{lem:balance.equilibrium.disc}
    Under the assumptions of Proposition~\ref{prop:sT'}, denote by $\uu[h]\in\UhD$ the unique solution of problem~\eqref{eq:poisson:discrete} and, for all $T\in\Th$ and all $F\in\Fh[T]$, define the numerical trace of the flux
    $$
    S_{TF}(\uu[T])
    \eqbydef -\GRAD\pT\uu[T]\SCAL\normal_{TF}
    + R_{TF}^k\uu[T]
    $$
    with $R_{TF}^k$ defined by~\eqref{eq:RpT}.
    Then, for all $T\in\Th$ we have the following discrete counterpart of the local balance~\eqref{eq:balance}:
    For all $v_T\in\Poly{k}(T)$,
    \begin{subequations}\label{eq:balance.equilibrium.disc}
      \begin{equation}\label{eq:balance.disc}
        (\GRAD\pT\uu,\GRAD v_T)_T + \sum_{F\in\Fh[T]} (S_{TF}(\uu[T]), v_T)_F
        = (f,v_T)_T,
      \end{equation}
      and, for any interface $F\in\Fhi$ such that $F\subset\partial T_1\cap\partial T_2$ with distinct mesh elements $T_1,T_2\in\Th$, the numerical fluxes are continuous in the sense that (compare with~\eqref{eq:equilibrium}):
      \begin{equation}\label{eq:equilibrium.disc}
        S_{T_1F}(\uu[T_1]) + S_{T_2F}(\uu[T_2])=0.
      \end{equation}
    \end{subequations}
  \end{lemma}
\end{svgraybox}
\begin{proof}
  Let $\uv[h]\in\UhD$.
  Plugging the definition~\eqref{eq:aT} of $\mathrm{a}_T$ into~\eqref{eq:ah}, using for all $T\in\Th$ the definition of $\pT\uv[T]$ with $w=\pT\uu[T]$, and recalling the reformulation~\eqref{eq:sT'} of $\mathrm{s}_T$ together with the definition~\eqref{eq:RpT} of $\RpT$ to write
  \begin{equation}\label{eq:sT.RTF}
    \mathrm{s}_T(\uu,\uv)
    = -\sum_{F\in\Fh[T]}(R_{TF}^k\uu,v_F-v_T)_F\qquad\forall T\in\Th,
  \end{equation}
  we infer from the discrete problem~\eqref{eq:poisson:discrete} that
  $$
  \sum_{T\in\Th}\left(
  (\GRAD\pT\uu,\GRAD v_T)_T
  + \sum_{F\in\Fh[T]}(\GRAD\pT\uu\SCAL\normal_{TF} - R_{TF}^k\uu,v_F-v_T)_F
  \right)=(f,v_h).
  $$
  Selecting $\uv[h]$ such that $v_T$ spans $\Poly{k}(T)$ for a selected mesh element $T\in\Th$ while $v_{T'}\equiv 0$ for all $T'\in\Th\setminus\{T\}$ and $v_F\equiv 0$ for all $F\in\Fh$, we obtain~\eqref{eq:balance.disc}.
  On the other hand, selecting $\uv[h]$ such that $v_T\equiv 0$ for all $T\in\Th$, $v_F$ spans $\Poly{k}(F)$ for a selected interface $F\in\Fhi$ such that $F\subset\partial T_1\cap\partial T_2$ for distinct mesh elements $T_1,T_2\in\Th$, and $v_{F'}\equiv 0$ for all $F'\in\Fh\setminus\{F\}$ yields~\eqref{eq:equilibrium.disc}.
\end{proof}

\begin{remark}[Interpretation of the discrete problem]
  Lemma~\ref{lem:balance.equilibrium.disc} and its proof provide further insight into the structure of the discrete problem~\eqref{eq:poisson:discrete}, which consists of the local balances~\eqref{eq:balance.disc} (corresponding to the local block equations~\eqref{eq:static.cond:1}) and a global transmission condition enforcing the continuity~\eqref{eq:equilibrium.disc} of numerical fluxes (corresponding to the global skeletal problem~\eqref{eq:static.cond:2}).
\end{remark}

\subsection{A priori error analysis}\label{sec:basics:apriori.error.analysis}

Having proved that the discrete problem~\eqref{eq:poisson:discrete} is well-posed, it remains to determine the convergence of the discrete solution towards the exact solution, which is precisely the goal of this section.

\subsubsection{Energy error estimate}

We start by deriving a basic convergence result. The error is measured as the difference between the exact solution and the global reconstruction obtained from the discrete solution through the operator $\ph:\Uh\to\Poly{k+1}(\Th)$ such that, for all $\uv[h]\in\Uh$,
\begin{equation}\label{eq:rh}
  \restrto{(\ph\uv[h])}{T}\eqbydef\pT\uv\qquad\forall T\in\Th.
\end{equation}

\begin{svgraybox}
  \begin{theorem}[Energy error estimate]\label{thm:poisson:en.err.est}
    Let a polynomial degree $k\ge 0$ be fixed.
    Let $u\in H_0^1(\Omega)$ denote the unique solution to~\eqref{eq:poisson:weak}, for which we assume the additional regularity $u\in H^{k+2}(\Omega)$.
    Let $\uu[h]\in\UhD$ denote the unique solution to~\eqref{eq:poisson:discrete} with stabilization bilinear form $\mathrm{s}_T$ in~\eqref{eq:aT} satisfying assumptions (S1)--(S3) for all $T\in\Th$.
    Then, there exists a real number $C>0$ independent of $h$, but possibly depending on $d$, $\varrho$, and $k$, such that
    \begin{equation}\label{eq:poisson:en.err.est}
      \norm{\GRADh(\ph\uu[h]-u)} + \seminorm[\mathrm{s},h]{\uu[h]}\le C h^{k+1}\norm[H^{k+2}(\Omega)]{u},
    \end{equation}
    where $\seminorm[\mathrm{s},h]{{\cdot}}$ is the seminorm defined by the bilinear form $\mathrm{s}_h$ on $\Uh$.
  \end{theorem}
\end{svgraybox}
\begin{proof}
  Let, for the sake of brevity, $\uhu\eqbydef\Ih u$ and $\cu[h]\eqbydef\ph\uhu$.
  We abridge as $A\lesssim B$ the inequality $A\le cB$ with multiplicative constant $c>0$ having the same dependencies as $C$ in~\eqref{eq:poisson:en.err.est}.
  Using the triangle and Cauchy--Schwarz inequalities, it is readily inferred that
  \begin{equation}\label{eq:en.err.est:basic}
    \norm{\GRADh(\ph\uu[h]-u)} + \seminorm[\mathrm{s},h]{\uu[h]}
    \le
    \underbrace{\vphantom{\Big(}\norm[\mathrm{a},h]{\uu[h]-\uhu}}_{\term_1} +
    \underbrace{\Big(
      \norm{\GRADh(\cu[h]-u)}^2 + \seminorm[\mathrm{s},h]{\uhu}^2.
      \Big)^{\nicefrac12}}_{\term_2}.
  \end{equation}
  We have that
  $$
  \begin{aligned}
    \term_1^2 
    &= \mathrm{a}_h(\uu[h],\uu[h]-\uhu[h]) - \mathrm{a}_h(\uhu[h],\uu[h]-\uhu[h])
    \\
    &= (f,u_h-\hat{u}_h) - \mathrm{a}_h(\uhu[h],\uu[h]-\uhu[h])
    = {\cal E}_h(u;\uu[h]-\uhu[h]),
  \end{aligned}
  $$
  where we have used the definition~\eqref{eq:stab.h} of the $\norm[\mathrm{a},h]{{\cdot}}$-norm together with the linearity of $\mathrm{a}_h$ in its first argument in the first line, the discrete problem~\eqref{eq:poisson:discrete} to pass to the second line, and the definition~\eqref{eq:Eh} of the conformity error to conclude.
  As a consequence, assuming $\uu[h]\neq\uhu[h]$ (the other case is trivial), we have that
  $$
  |\term_1|
  \le {\cal E}_h\left(u;\frac{\uu[h]-\uhu[h]}{\norm[\mathrm{a},h]{\uu[h]-\uhu[h]}}\right)
  \le \eta^{\nicefrac12}{\cal E}_h\left(u;\frac{\uu[h]-\uhu[h]}{\norm[1,h]{\uu[h]-\uhu[h]}}\right)
  \le\eta^{\nicefrac12}\hspace{-1em}\sup_{\uv[h]\in\UhD,\norm[1,h]{\uv[h]}=1}\hspace{-1em}{\cal E}_h(u;\uv[h]),
  $$
  where we have used the linearity of ${\cal E}_h(u;\cdot)$, the first bound in~\eqref{eq:stab.h}, and a passage to the supremum to conclude.
  Recalling~\eqref{eq:consistency.h}, we arrive at
  \begin{equation}\label{eq:en.err.est:T1}
    |\term_1|\lesssim h^{k+1}\norm[H^{k+2}(\Omega)]{u}.
  \end{equation}
  On the other hand, using the approximation properties~\eqref{eq:approx.approx.trace} of $\cu$ with $\alpha=1$, $l=k+1$, $s=k+2$, and $m=1$ together with the approximation properties~\eqref{eq:approx.sT} of $\mathrm{s}_T$, it is inferred for the second term
  \begin{equation}\label{eq:en.err.est:T2}
    |\term_2|\lesssim h^{k+1}\norm[H^{k+2}(\Omega)]{u}.
  \end{equation}
  Using~\eqref{eq:en.err.est:T1} and \eqref{eq:en.err.est:T2} to bound the right-hand side of~\eqref{eq:en.err.est:basic},~\eqref{eq:poisson:en.err.est} follows.\qed
\end{proof}

\subsubsection{Convergence of the jumps}
Functions in $H^1(\Th)\eqbydef\left\{v\in L^2(\Omega)\st\restrto{v}{T}\in H^1(T)\quad\forall T\in\Th\right\}$ are in $H_0^1(\Omega)$ if their jumps vanish a.e. at interfaces and their trace is zero a.e. on $\partial\Omega$; see, e.g.,~\cite[Lemma~1.23]{Di-Pietro.Ern:12}.
Thus, a measure of the nonconformity is provided by the jump seminorm $\seminorm[{\rm J},h]{{\cdot}}$ such that, for all $v\in H^1(\Th)$,
\begin{equation}\label{eq:seminorm.J}
  \seminorm[{\rm J},h]{v}^2\eqbydef 
  \sum_{F\in\Fh}h_F^{-1}\norm[F]{\lproj[F]{k}\jump{v}}^2,
\end{equation}
where $\jump{{\cdot}}$ denotes the usual jump operator such that, for all faces $F\in\Fh$ and all functions $v:\bigcup_{T\in\Th[F]}T\to\Real$ smooth enough,
\begin{equation}\label{eq:def.jump}
  \jump{v}\eqbydef\begin{cases}
  \restrto{v}{T_1}-\restrto{v}{T_2} & \forall F\in\Fh[T_1]\cap\Fh[T_2],
  \\
  v & \forall F\in\Fhb.
  \end{cases}
\end{equation}
A natural question is whether the jump seminorm of $\ph\uu[h]$ converges to zero.
The answer is provided by the following lemma.
\begin{lemma}[Convergence of the jumps]\label{lem:jump.est}
  Under the assumptions and notations of Theorem~\ref{thm:poisson:en.err.est}, and further supposing, for the sake of simplicity, that the local stabilization bilinear form $\mathrm{s}_T$ is given by~\eqref{eq:sT.vem}, there is a real number $C>0$ independent of $h$, but possibly depending on $d$, $\varrho$, and $k$, such that
  \begin{equation}\label{eq:jump.est}
    \seminorm[{\rm J},h]{\ph\uu[h]}\le C h^{k+1}\norm[H^{k+2}(\Omega)]{u}.
  \end{equation}
\end{lemma}

\begin{proof}
  Inserting $u_F$ inside the jump and using the triangle inequality for every interface $F\in\Fhi$, and recalling that $v_F=0$ on every boundary face $F\in\Fhb$, it is inferred that
  $$
  \begin{aligned}
    \sum_{F\in\Fh}h_F^{-1}\norm[F]{\lproj[F]{k}\jump{\ph\uu[h]}}^2
    &\le 2\sum_{F\in\Fh}\sum_{T\in\Th[F]}h_F^{-1}\norm[F]{\lproj[F]{k}(\pT\uu - u_F)}^2
    \\
    &\le 2\sum_{T\in\Th}\sum_{F\in\Fh[T]}h_F^{-1}\norm[F]{\lproj[F]{k}(\pT\uu - u_F)}^2
    \le 2\seminorm[\mathrm{s},h]{\uu[h]}^2.
  \end{aligned}
  $$
  Using~\eqref{eq:poisson:en.err.est} to bound the right-hand side yields~\eqref{eq:jump.est}.
\end{proof}

\subsubsection{$L^2$-error estimate}

To close this section, we state a result concerning the convergence of the error in the $L^2$-norm.
Optimal error estimates require in this context further regularity for the continuous operator.
More precisely, we assume that, for all $g\in L^2(\Omega)$, the unique solution of the problem:
Find $z\in H_0^1(\Omega)$ such that
$$
a(z,v) = (g,v)\qquad\forall v\in H_0^1(\Omega)
$$
satisfies the a priori estimate
$$
\norm[H^2(\Omega)]{z}\le C\norm{g},
$$
with real number $C$ depending only on $\Omega$.
Elliptic regularity holds when the domain $\Omega$ is convex; see, e.g.,~\cite{Grisvard:92}.
The following result, whose detailed proof is omitted, can be obtained using the arguments of~\cite[Theorem~10]{Di-Pietro.Ern.ea:14} and~\cite[Corollary~4.6]{Aghili.Boyaval.ea:15}.

\begin{svgraybox}
  \begin{theorem}[$L^2$-error estimate]\label{thm:poisson:l2.err.est}
    Under the assumptions and notations of Theorem~\ref{thm:poisson:en.err.est}, and further assuming elliptic regularity and that $f\in H^1(\Omega)$ if $k=0$, $f\in H^k(\Omega)$ if $k\ge 1$, there exists a real number $C>0$ independent of $h$, but possibly depending on $\Omega$, $d$, $\varrho$, and $k$, such that
    \begin{equation}\label{eq:poisson:l2.err.est}
      \norm{\ph\uu[h]-u}\le \begin{cases}
        C h^2\norm[H^{1}(\Omega)]{f} & \text{if $k=0$},
        \\
        C h^{k+2}\left(\norm[H^{k+2}(\Omega)]{u} + \norm[H^k(\Omega)]{f}\right) & \text{if $k\ge 1$}.
      \end{cases}
    \end{equation}
  \end{theorem}
\end{svgraybox}

\begin{remark}[Supercloseness of element DOFs]\label{rem:supercloseness}
  An intermediate step in the proof of the estimate~\eqref{eq:poisson:l2.err.est} (see~\cite[Theorem~10]{Di-Pietro.Ern.ea:14}) consists in showing that the element DOFs are superclose to the $L^2$-projection of the exact solution on $\Poly{k}(\Th)$:
  \begin{equation}\label{eq:supercloseness}
    \norm{\lproj{k}u-u_h}\le \begin{cases}
      C h^2\norm[H^{1}(\Omega)]{f} & \text{if $k=0$},
      \\
      C h^{k+2}\left(\norm[H^{k+2}(\Omega)]{u} + \norm[H^k(\Omega)]{f}\right) & \text{if $k\ge 1$}.
      \end{cases}
  \end{equation}
  This is done adapting to the HHO framework the classical Aubin--Nitsche technique.
\end{remark}

\subsection{A posteriori error analysis}\label{sec:basics:aposteriori.error.analysis}

For smooth enough exact solutions, it is classically expected that increasing the polynomial degree $k$ will reduce the computational time required to achieve a desired precision; see, e.g., the numerical test in Section~\ref{sec:poisson:num:3d.smooth} below and, in particular, Fig.~\ref{fig:poisson.3d.smooth:err.vs.t_tot}.
However, when the regularity requirements detailed in Theorems~\ref{thm:poisson:en.err.est} and~\ref{thm:poisson:l2.err.est} are not met, the order of convergence is limited by the regularity of the solution instead of the polynomial degree. 
To restore optimal orders of convergence, local mesh adaptation is required.
This is typically done using a posteriori error estimators to mark the elements where the error is larger, and locally refine the computational mesh based on this information. 
Here, we present energy-norm upper and lower bounds for the HHO method~\eqref{eq:poisson:discrete} inspired by the residual-based approach of~\cite{Di-Pietro.Specogna:16}.

\subsubsection{Error upper bound} 

We start by proving an upper bound of the discretization error in terms of quantities whose computation does not require the knowledge of the exact solution.
We will need the following local Poincar\'{e} and Friedrichs inequalities, valid for all $T\in\Th$ and all $\varphi\in H^1(T)$:
\begin{align}
  \norm[T]{\varphi - \lproj[T]{0}\varphi} &\le C_{{\rm P},T} h_T \norm[T]{\GRAD\varphi}, \label{eq:poincare_loc}
  \\
  \norm[\partial T]{\varphi - \lproj[T]{0}\varphi} &\le C_{{\rm F},T}^{\nicefrac12} h_T^{\nicefrac12} \norm[T]{\GRAD\varphi}. \label{eq:friedrichs}
\end{align}
In~\eqref{eq:poincare_loc}, $C_{{\rm P},T}$ is a constant equal to $\pi^{-1}$ if $T$ is convex~\cite{Payne.Weinberger:60,Bebendorf:03}, and for which upper bounds on nonconvex elements can be found in~\cite{Vohralik:07}.
In~\eqref{eq:friedrichs}, $C_{{\rm F},T}$ is a constant which, if $T$ is a simplex, can be estimated as $C_{{\rm F},T}=C_{{\rm P},T}(h_T \meas[d-1]{\partial T}/{\meas{T}})(2/d+C_{{\rm P},T})$ (see~\cite[Section~5.6.2.2]{Di-Pietro.Ern:12}).
\begin{svgraybox}
  \begin{theorem}[A posteriori error upper bound]\label{thm:poisson:apost.upper.bound}
    Let $u\in H_0^1(\Omega)$ and $\uu[h]\in\UhD$ denote the unique solutions to problems~\eqref{eq:poisson:weak} and~\eqref{eq:poisson:discrete}, respectively, with local stabilization bilinear form $\mathrm{s}_T$ satisfying the assumptions of Proposition~\ref{prop:sT'} for all $T\in\Th$.
    Let $u_h^*$ be an arbitrary function in $H_0^1(\Omega)$. Then, it holds that
    \begin{equation}\label{eq:poisson:a.post}
      \norm{\GRADh(\ph\uu[h]-u)}\le\left[
        \sum_{T\in\Th}\left(
        \est{nc}^2 + (\est{res}+\est{sta})^2
        \right)
      \right]^{\nicefrac12},
    \end{equation}
    with local nonconformity, residual, and stabilization estimators such that, for all $T\in\Th$,
    \begin{subequations}
      \begin{align}
        \est{nc} &\eqbydef \norm[T]{\GRAD(\pT\uu[T]-u_h^*)}, \label{eq:def:est_nc}
        \\
        \est{res} &\eqbydef C_{{\rm P},T} h_T\norm[T]{(f+\LAPL\pT\uu[T])-\lproj[T]{0}(f+\LAPL\pT\uu[T])}, \label{eq:def:est_res}
        \\
        \est{sta} &\eqbydef C_{{\rm F},T}^{\nicefrac12} h_T^{\nicefrac12}\left(\sum_{F\in\Fh[T]}\norm[F]{R_{TF}^k\uu}^2\right)^{\nicefrac12},\label{eq:def:est_sta}
      \end{align}
    \end{subequations}
    where, for all $F\in\Fh[T]$, the boundary residual $R_{TF}^k$ is defined by~\eqref{eq:RpT}.
  \end{theorem}  
\end{svgraybox}
\begin{remark}[Nonconformity estimator]\label{rem:nonconformity_est}
  To compute the estimator $\est{nc}$, we can obtain a $H_0^1(\Omega)$-conforming function $u_h^*$ by applying a node-averaging operator to $\ph\uu[h]$.
  Let an integer $l\ge 1$ be fixed.
  When $\Th$ is a matching simplicial mesh and $\Fh$ is the corresponding set of simplicial faces, the node-averaging operator   $\Osw[l]  : \Poly{l}(\Th) \to \Poly{l}(\Th) \cap H_{0}^{1}(\Omega)$ is defined by setting for each (Lagrange) interpolation node $N$
  $$
  \Osw[l] v_h (N) \eqbydef\begin{cases}
  \frac{1}{\card(\mathcal{T}_N)} \sum_{T \in \mathcal{T}_N} (v_h)_{|T} (N) & \text{if $N\in\Omega$},
  \\
  0 & \text{if $N\in\partial\Omega$},
  \end{cases}
  $$
  where the set $\mathcal{T}_N \subset \Th$ collects the simplices to which $N$ belongs.  
  We then set
  \begin{equation}\label{eq:osw_u*}
    u_h^* \eqbydef \Osw \ph\uu[h].
  \end{equation}
  The generalization to polytopal meshes can be realized applying the node averaging operator to $\ph\uu[h]$ on a simplicial submesh of $\Th$ (whose existence is guaranteed for regular mesh sequences, see Definition~\ref{def:mesh.reg}).
\end{remark}
\begin{proof}
  Let the equation residual $\mathcal{R}\in H^{-1}(\Omega)$ be such that, for all $\varphi\in H_0^1(\Omega)$, $\langle\mathcal{R},\varphi\rangle_{-1,1}\eqbydef(f,\varphi) - (\GRADh\ph\uu[h],\GRAD\varphi)$.
  The following abstract error estimate descends from~\cite[Lemma 5.44]{Di-Pietro.Ern:12} and is valid for any function $u_h^*\in H_0^1(\Omega)$:
  \begin{equation}\label{eq:upper_bound_residual0}
    \norm{\GRADh (\ph\uu[h]-u)}^2\le 
    \norm{\GRADh(\ph\uu[h] - u_h^*)}^2
    +\left(
    \sup_{\varphi \in H_0^1(\Omega),\norm{\GRAD\varphi}=1} \langle\mathcal{R},\varphi\rangle_{-1,1}
    \right)^2.
  \end{equation}
  Denote by $\term_1$ and $\term_2$ the addends in the right-hand side of~\eqref{eq:upper_bound_residual0}.

  \begin{asparaenum}[(i)]
  \item \emph{Bound of $\term_1$.}
  Recalling the definition~\eqref{eq:def:est_nc} of the nonconformity estimator, it is readily inferred that
  \begin{equation}\label{eq:upper_bound:T1}
    \term_1 = \sum_{T\in\Th}\est{nc}^2.
  \end{equation}
  \item \emph{Bound of $\term_2$.}
  We bound the argument of the supremum in $\term_2$ for a generic function $\varphi \in H_0^1(\Omega)$. 
  Using an element-by-element integration by parts, we obtain 	 
  \begin{equation}\label{eq:upper_bound_residual}
    \langle\mathcal{R},\varphi\rangle_{-1,1}
    =\sum_{T\in\Th} \bigg(
    ( f+\LAPL\pT\uu[T], \varphi )_T - \sum_{F\in\Fh[T]} ( \GRAD \pT\uu[T] \cdot \normal_{TF}, \varphi)_F 
    \bigg).
  \end{equation}
  
  Let now $\uphi[h] \in \UhD$ be such that $\varphi_T = \lproj[T]{0}{\varphi}$ for all $T\in\Th$ and $\varphi_F = \lproj[F]{k}\restrto{\varphi}{F}$ for all $F\in\Fh$. 
  We have that 
  \begin{equation}\label{eq:upper_bound_mean}
    \begin{aligned}
      \sum_{T\in\Th} ( \lproj[T]{0}{( f+\LAPL\pT\uu[T] )}, \varphi )_T
      &= \sum_{T\in\Th} (f+\LAPL\pT\uu[T], \varphi_T)_T \\
      &= \sum_{T\in\Th} \bigg(
      \mathrm{a}_T(\uu,\uphi) + \sum_{F\in\Fh[T]} (  \GRAD \pT\uu[T] \cdot \normal_{TF},  \varphi_T )_F 
      \bigg)
      \\
      &= \sum_{T\in\Th} \bigg(
      \mathrm{s}_T(\uu,\uphi) + \sum_{F\in\Fh[T]} ( \GRAD \pT\uu[T] \cdot \normal_{TF}, \varphi )_F 
      \bigg),    	
    \end{aligned} 
  \end{equation}  	
  where we have used definition~\eqref{eq:lproj} of $\lproj[T]{0}$ in the first line, 
  the discrete problem \eqref{eq:poisson:discrete} with $\uv[h]=\uphi[h]$ and an element-by-element integration by parts together with the fact that $\GRAD\varphi_T\equiv 0$ for all $T\in\Th$ in the second line. 
  In order to pass to the third line, we have expanded $\mathrm{a}_T$ according to its definition \eqref{eq:aT} and used \eqref{eq:pT:1} with $\uv=\uphi$ and $w=\pT\uu$ for the consistency term (in the boundary integral, we can write $\varphi$ instead of $\varphi_F$ using the definition~\eqref{eq:lproj} of $\lproj[F]{k}$).

  Summing~\eqref{eq:upper_bound_mean} and \eqref{eq:upper_bound_residual}, and rearranging the terms, we obtain
  \begin{equation} \label{eq:upper_bound_residual2}
    \langle\mathcal{R},\varphi\rangle_{-1,1}
    = \sum_{T\in\Th}\bigg(
    ( f+\LAPL\pT\uu[T] - \lproj[T]{0}{(f+\LAPL\pT\uu[T])} , \varphi - \varphi_T)_T
    + \mathrm{s}_T(\uu,\uphi)
    \bigg), 
  \end{equation}
  where we have used the definition~\eqref{eq:lproj} of $\lproj[T]{0}$ to insert $\varphi_T$ into the first term.
  Let us estimate the addends inside the summation, hereafter denoted by $\term_{2,1}(T)$ and $\term_{2,2}(T)$.
  Using the Cauchy--Schwarz and local Poincar\'{e} \eqref{eq:poincare_loc} inequalities, and recalling the definition~\eqref{eq:def:est_res} of the residual estimator, we readily infer, for all $T\in\Th$, that
  \begin{equation} \label{eq:upper_bound_residual:I1}
    |\term_{2,1}(T)|\le\est{res}\norm[T]{\GRAD\varphi}.
  \end{equation} 
  On the other hand, recalling the reformulation~\eqref{eq:sT.RTF} of the local stabilization bilinear form $\mathrm{s}_T$ we have, for all $T\in\Th$,
  \begin{equation} \label{eq:upper_bound_residual:I2}
    | \term_{2,2}(T) |
    =  \bigg| \sum_{F\in\Fh[T]}(R_{TF}^k\uu,\varphi -\varphi_T)_F \bigg|
    \le \est{sta} \norm[T]{\GRAD\varphi}, 
  \end{equation} 
  where we have used the fact that $\varphi_F= \lproj[F]{k}{\varphi}$ and $R_{TF}^k\uu \in \Poly{k}(F)$ together with the definition~\eqref{eq:lproj} of $\lproj[F]{k}$ to write $\varphi$ instead of $\varphi_F$ inside the boundary term,
  and the Cauchy--Schwarz and local Friedrichs~\eqref{eq:friedrichs} inequalities followed by definition~\eqref{eq:def:est_sta}  of the stability estimator to conclude.
  Using~\eqref{eq:upper_bound_residual:I1} and~\eqref{eq:upper_bound_residual:I2} to estimate the right-hand side of~\eqref{eq:upper_bound_residual2} followed by a Cauchy--Schwarz inequality, and plugging the resulting bound inside the supremum in $\term_2$, we arrive at
  \begin{equation}\label{eq:upper_bound:T2}
    \term_2\le\sum_{T\in\Th}(\est{res} + \est{sta})^2.
  \end{equation}
  \item \emph{Conclusion.}
    Plugging~\eqref{eq:upper_bound:T1} and~\eqref{eq:upper_bound:T2} into~\eqref{eq:upper_bound_residual0}, the conclusion follows.\qed
  \end{asparaenum}
\end{proof}

\subsubsection{Error lower bound} 

In practice, one wants to make sure that the error estimators are able to correctly localize the error (for use, e.g., in adaptive mesh refinement) and that they do not unduly overestimate it.
We prove in this section that the error estimators defined in Theorem~\ref{thm:poisson:apost.upper.bound} are \emph{locally efficient}, i.e., they are locally controlled by the error.
This shows that they are suitable to drive mesh refinement.
Moreover, they are also \emph{globally efficient}, i.e., the right-hand side of~\eqref{eq:poisson:a.post} is (uniformly) controlled by the discretization error, so that it cannot depart from it.

Let a mesh element $T\in\Th$ be fixed and define the following sets of faces and elements sharing at least one node with $T$:
$$
\Fhh{\mathcal{N}}{T} \eqbydef\{F\in\Fh\st \closure F \cap \partial T \not = \emptyset\},\qquad
\Thh{\mathcal{N}}{T} \eqbydef\{T'\in\Th\st \closure T' \cap \closure T \not = \emptyset\}. 
$$
Let an integer $l\ge 1$ be fixed.
The following result is proved in \cite{Karakashian.Pascal:03} for standard meshes:
There is a real number $C>0$ independent of $h$, but possibly depending on $d$, $\varrho$, and $l$, such that, for all $v_h \in \Poly{l}(\Th)$ and all $T \in \Th$,
\begin{equation}\label{eq:osw}
  \norm[T]{ v_h - \Osw[l] v_h  }^2 \le C \sum_{F\in \Fhh{\mathcal{N}}{T} } h_F \norm[F]{\jump{v_h}}^2,
\end{equation}
with jump operator defined by~\eqref{eq:def.jump}.
Following~\cite[Section~5.5.2]{Di-Pietro.Ern:12},~\eqref{eq:osw} still holds on regular polyhedral meshes when the nodal interpolator is defined on the matching simplicial submesh of Definition~\ref{def:mesh.reg}.
We also note the following technical result:
\begin{proposition}[Estimate of boundary oscillations]
  Let an integer $l\ge 0$ be fixed.
  There is a real number $C>0$ independent of $h$, but possibly depending on $d$, $\varrho$, and $l$, such that, for all mesh elements $T\in\Th$ and all functions $\varphi\in H^1(T)$,
  \begin{equation}\label{eq:magic.boundary}
    h_F^{-\nicefrac12}\norm[F]{\varphi-\lproj[F]{l}\varphi}
    \le C\norm[T]{\GRAD\varphi}.
  \end{equation}
\end{proposition}
\begin{proof}
  We abridge as $A\lesssim B$ the inequality $A\le cB$ with multiplicative constant $c>0$ having the same dependencies as $C$ in~\eqref{eq:magic.boundary}.
  Let $F\in\Fh[T]$ and observe that
  \begin{equation}\label{eq:magic.boundary:1}
    \begin{aligned}
      \norm[F]{\varphi-\lproj[F]{l}\varphi}
      &\le\norm[F]{\varphi-\lproj[T]{l}\varphi} + \norm[F]{\lproj[F]{l}(\lproj[T]{l}\varphi-\varphi)}
      \\
      &\le 2\norm[F]{\varphi-\lproj[T]{l}\varphi}
      \lesssim h_T^{\nicefrac12}\norm[T]{\GRAD\varphi},
    \end{aligned}
  \end{equation}  
  where we have inserted $\pm\lproj[T]{l}\varphi$ and used the triangle inequality to infer the first bound,
  we have used the $L^2(F)$-boundedness of $\lproj[F]{l}$ to infer the second,
  and invoked~\eqref{eq:approx.trace} with $\alpha=0$, $m=0$, and $s=1$ to conclude.
  Using the fact that $h_T/h_F\lesssim 1$ owing to~\eqref{eq:hT_hF} gives the desired result.\qed
\end{proof}

\begin{svgraybox}
  \begin{theorem}[A posteriori error lower bound]\label{thm:poisson:apost.lower.bound}
    Under the assumptions of Theorem~\ref{thm:poisson:apost.upper.bound}, 
    and further assuming, for the sake of simplicity,
    \begin{inparaenum}[(i)]
    \item that the local stabilization bilinear form $\mathrm{s}_T$ is given by~\eqref{eq:sT.vem} for all $T\in\Th$,
    \item that $u_h^*$ is obtained applying the node-averaging operator to $\ph\uu[h]$ on $\Th$ if $\Th$ is matching simplicial or on the simplicial submesh of Definition~\ref{def:mesh.reg} if this is not the case, and
    \item that $f \in \Poly{k+1}(\Th)$,
    \end{inparaenum}
    it holds for all $T\in\Th$,
    \begin{subequations}\label{eq:lower_bound}
      \begin{align}
        \est[T]{nc} &\le C  \left(
        \norm[\mathcal{N},T]{\GRADh (\ph\uu[h] -u)} + \seminorm[\mathrm{s},\mathcal{N},T]{\uu[h]}
        \right), \label{eq:lower_bound:est_nc}
        \\
        \est{res} &\le C \norm[T]{\GRAD (\pT\uu[T]-u_{\st T})},\label{eq:lower_bound:est_res}
        \\
        \est{sta} &\le C \seminorm[\mathrm{s},T]{\uu},\label{eq:lower_bound:est_sta}
      \end{align}
    \end{subequations}
    where $C>0$ is a real number possibly depending on $d$, $\varrho$, and on $k$ but independent of both $h$ and $T$.
    For all $T\in\Th$, $\norm[\mathcal{N},T]{{\cdot}}$ denotes the $L^2$-norm on the union of the elements in $\Thh{\mathcal{N}}{T}$ and we have set
    $$
    \seminorm[\mathrm{s},T]{\uu} = \mathrm{s}_T(\uu,\uu)^{\nicefrac12},\qquad
    \seminorm[\mathrm{s},\mathcal{N},T]{\uu[h]}^2
    \eqbydef \sum_{T'\in\Thh{\mathcal{N}}{T}} \seminorm[\mathrm{s},T']{\uu}^2.
    $$ 
  \end{theorem}  
\end{svgraybox}
\begin{proof}
  Let a mesh element $T\in\Th$ be fixed. 
  In the proof, we abridge as $A\lesssim B$ the inequality $A\le cB$ with multiplicative constant $c>0$ having the same dependencies as $C$ in~\eqref{eq:lower_bound}.

  \begin{asparaenum}[(i)]
  \item \emph{Bound~\eqref{eq:lower_bound:est_nc} on the nonconformity estimator.}
    Using a local inverse inequality (see, e.g.,~\cite[Lemma~1.44]{Di-Pietro.Ern:12}) and the relation \eqref{eq:osw}, we infer from~\eqref{eq:def:est_nc} that 
    \begin{equation}\label{eq:lower_bound:est_nc_1}
      \est[T]{nc}^2 \lesssim h_{T}^{-2} \norm[T]{\pT\uu[T]-u_h^*}^2
      \lesssim \sum_{F\in \Fhh{\mathcal{N}}{T}} h_{F}^{-1} \norm[F]{ \jump{\ph\uu[h]} }^2,
    \end{equation}
    where we have used the fact that, owing to mesh regularity, $h_F\lesssim h_T$ for all $F\in\Fhh{\mathcal{N}}{T}$.
    Using the fact $\jump{u}=0$ for all $F\in\Fh$ (see, e.g.,~\cite[Lemma~4.3]{Di-Pietro.Ern:12}) to write ${\jump{\ph\uu[h]-u}}$ instead of $\jump{\ph\uu[h]}$, inserting $\lproj[F]{k}\jump{\ph\uu[h]}-\lproj[F]{k}\jump{\ph\uu[h]-u}=0$ inside the norm, and using the triangle inequality, we have for all $F\in\Fhh{\mathcal{N}}{T}$,
    $$
    \begin{aligned}
      \norm[F]{ \jump{\ph\uu[h]} } &\le
      \norm[F]{\jump{\ph\uu[h]-u} - \lproj[F]{k}{\jump{\ph\uu[h]-u}}}
      + \norm[F]{\lproj[F]{k}{\jump{\ph\uu[h]}}}
      \\
     &\le \sum_{T\in\Th[F]}\norm[F]{(\pT\uu-u) - \lproj[F]{k}(\pT\uu-u)}
      + \norm[F]{\lproj[F]{k}{\jump{\ph\uu[h]}}},
    \end{aligned}
    $$
    where we have expanded the jump according to its definition~\eqref{eq:def.jump} and used a triangle inequality to pass to the second line.
    Plugging the above bound into~\eqref{eq:lower_bound:est_nc_1}, and using multiple times~\eqref{eq:magic.boundary} with $\varphi=\pT\uu-u$ for $T\in\Thh{\mathcal{N}}{T}$, we arrive at
    $$
    \est[T]{nc}^2 \lesssim
    \norm[\mathcal{N},T]{\GRAD(\pT\uu-u)}^2
    + \sum_{F\in\Fhh{\mathcal{N}}{T}} h_F^{-1}\norm[F]{\lproj[F]{k}{\jump{\ph\uu[h]}}}^2.
    $$
    To conclude, we proceed as in the proof of Lemma~\ref{lem:jump.est} to prove that the last term is bounded by $\seminorm[\mathrm{s},\mathcal{N},T]{\uu[h]}^2$ up to a constant independent of $h$.

  \item \emph{Bound~\eqref{eq:lower_bound:est_res} on the residual estimator.}
    We use classical bubble function techniques, see e.g.~\cite{Verfurth:96}.
    For the sake of brevity, we let $r_T\eqbydef f_{|T}+\LAPL\pT\uu[T]$.
    Denote by $\fTh$ the simplicial submesh of $\Th$ introduced in Definition~\ref{def:mesh.reg}, and let $\fTh[T]\eqbydef\left\{\tau\in\fTh\st\tau\subset T\right\}$, the set of simplices contained in $T$.
    For all $\tau\in\fTh[T]$, we denote by $b_\tau\in H_0^1(\tau)$ the element bubble function equal to the product of barycentric coordinates of $\tau$ and rescaled so as to take the value 1 at the center of gravity of $\tau$.
    Letting $\psi_\tau \eqbydef b_\tau r_T$ for all $\tau\in\fTh[T]$, the following properties hold \cite{Verfurth:96}:\\
    \begin{subequations}
      \noindent\begin{tabularx}{\textwidth}{@{}XXX@{}}
        \begin{equation}
          \text{$\psi_\tau=0$ on $\partial \tau$},\label{eq:bubble:1}
        \end{equation} & 
        \begin{equation}
          \norm[\tau]{ r_T }^2 \lesssim (r_T, \psi_\tau )_\tau,\label{eq:bubble:2}
        \end{equation} & 
        \begin{equation}
          \norm[\tau]{\psi_\tau} \lesssim \norm[\tau]{r_T}.\label{eq:bubble:3}
        \end{equation}
      \end{tabularx}
    \end{subequations}    
    We have that
    \begin{equation}	\label{eq:lower_bound:est_res_1}
      \begin{aligned}
	\norm[T]{r_T}^2
        &= \sum_{\tau\in\fTh[T]} \norm[\tau]{r_T}^2 
	\lesssim \sum_{\tau\in\fTh[T]} (r_T, \psi_\tau )_\tau 
        \\
	&= \sum_{\tau\in\fTh[T]} (\GRAD(u-\pT\uu[T]),\GRAD\psi_\tau)_\tau
        \\
	&\le \norm[T]{\GRAD(u-\pT\uu)} \left(\sum_{\tau\in\fTh[T]} h_\tau^{-2}\norm[\tau]{\psi_\tau}^2\right)^{\nicefrac12}
        \\
	&\lesssim h_T^{-1} \norm[T]{\GRAD(u-\pT\uu)}\norm[T]{r_T}, 
      \end{aligned}
    \end{equation}	
    where we have used property~\eqref{eq:bubble:2}  in the first line,
    the fact that $f=-\LAPL u$ together with an integration by parts and property~\eqref{eq:bubble:1} to pass to the second line,
    the Cauchy--Schwarz inequality together with a local inverse inequality (see, e.g.,~\cite[Lemma~1.44]{Di-Pietro.Ern:12}) to pass to the third line,
    and~\eqref{eq:bubble:3} together with the fact that $h_\tau^{-1}\le(\varrho h_T)^{-1}$ for all $\tau\in\fTh[T]$ (see Definition \ref{def:mesh.reg}) to conclude.
Recalling the definition~\eqref{eq:def:est_res} of the residual estimator, observing that $\norm[T]{r_T-\lproj[T]{0} r_T}\le\norm[T]{r_T}$ as a result of the triangle inequality followed by the $L^2(T)$-boundedness of $\lproj[T]{0}$, and using \eqref{eq:lower_bound:est_res_1}, the bound~\eqref{eq:lower_bound:est_res} follows.  

  \item \emph{Bound~\eqref{eq:lower_bound:est_sta} on the stabilization estimator.}
    Using the definition \eqref{eq:RpT} of the boundary residual operator $\RpT$ with $\uv=\uu$ and $\ual={-h_T\RpT\uu}={(-h_T R_{TF}^k\uu)_{F\in\Fh[T]}}$, the stabilization estimator \eqref{eq:def:est_sta} can be bounded as follows:
    \begin{equation} \label{eq:lower_bound:est_sta1}
      \est{sta}^2 = C_{{\rm F},T} \mathrm{s}_{T}(\uu, (0,-h_T\RpT\uu))  
      \lesssim\seminorm[\mathrm{s},T]{\uu[T]} \seminorm[\mathrm{s},T]{(0,-h_T\RpT\uu)}.
    \end{equation}
    On the other hand, from property (S2) in Assumption \ref{ass:sT}, the relation \eqref{eq:hT_hF}, and the definition \eqref{eq:def:est_sta} of $\est{sta}$, it is inferred that 
    $$
    \seminorm[\mathrm{s},T]{(0,-h_T\RpT\uu)}
    \le \eta^{\nicefrac12} \left(\sum_{F\in\Fh[T]}h_{F}^{-1}\norm[F]{h_T R_{TF}^k\uu}^2\right)^{\nicefrac12}
    \le \eta^{\nicefrac12}\rho^{-1}C_{{\rm F},T}^{-\nicefrac12}\est{sta}.
    $$
    Using this estimate to bound the right-hand side of~\eqref{eq:lower_bound:est_sta1},~\eqref{eq:lower_bound:est_sta} follows.\qed
  \end{asparaenum}
\end{proof} 

\begin{corollary}[Global lower bound] 
  Under the assumptions of Theorem~\ref{thm:poisson:apost.lower.bound}, there exists a constant $C$ independent of $h$, but possibly depending on $d$, $\varrho$ and $k$, such that 
  \begin{equation*} 
    \left[
      \sum_{T\in\Th}\left(
      \est{nc}^2 + (\est{res}+\est{sta})^2
      \right)
      \right]^{\nicefrac12}
    \le C \left(\norm{\GRADh(\ph\uu[h]-u)} + \seminorm[\mathrm{s}, h]{\uu[h]}\right).
  \end{equation*}	
\end{corollary}

\subsection{Numerical examples}\label{sec:basics:num.ex}

We illustrate the numerical performance of the HHO method on a set of model problems.

\subsubsection{Two-dimensional test case}\label{sec:poisson:num:2d}

The first test case, taken from~\cite{Di-Pietro.Ern.ea:14}, aims at demonstrating the estimated orders of convergence in two space dimensions.
We solve the Dirichlet problem in the unit square $\Omega=(0,1)^2$ with
\begin{equation}\label{eq:poisson:num:2d:ex.sol}
  u(\vec{x}) = \sin(\pi x_1)\sin(\pi x_2),
\end{equation}
and corresponding right-hand side $f(\vec{x})=2\pi^2\sin(\pi x_1)\sin(\pi x_2)$ on the triangular and polygonal meshes of Fig.~\ref{fig:meshes:triangular} and~\ref{fig:meshes:polygonal}.
Fig.~\ref{fig:poisson.2d} displays convergence results for both mesh families and polynomial degrees up to 4.
Recalling~\eqref{eq:en.err.est:T1} and~\eqref{eq:supercloseness}, we measure the energy- and $L^2$-errors by the quantities $\norm[\mathrm{a},h]{\Ih u - \uu[h]}$ and $\norm{\lproj{k}u-u_h}$, respectively.
In all cases, the numerical results show asymptotic convergence rates that match those predicted by the theory.

\begin{figure}\centering
  \ref{conv.legend.poisson.2d}
  \vspace{0.5cm}\\
  \begin{minipage}[b]{0.45\linewidth}\centering
    \begin{tikzpicture}[scale=0.65]
      \begin{loglogaxis}[ legend columns=-1, legend to name=conv.legend.poisson.2d ]
        \addplot table[x=meshsize,y=err_Gh] {pho_0_mesh1.dat};
        \addplot table[x=meshsize,y=err_Gh] {pho_1_mesh1.dat};
        \addplot table[x=meshsize,y=err_Gh] {pho_2_mesh1.dat};
        \addplot table[x=meshsize,y=err_Gh] {pho_3_mesh1.dat};
        \addplot table[x=meshsize,y=err_Gh] {pho_4_mesh1.dat};
        \logLogSlopeTriangle{0.90}{0.4}{0.1}{1}{black};
        \logLogSlopeTriangle{0.90}{0.4}{0.1}{2}{black};
        \logLogSlopeTriangle{0.90}{0.4}{0.1}{3}{black};
        \logLogSlopeTriangle{0.90}{0.4}{0.1}{4}{black};
        \logLogSlopeTriangle{0.90}{0.4}{0.1}{5}{black};
        \legend{$k=0$,$k=1$,$k=2$,$k=3$,$k=4$};
      \end{loglogaxis}
    \end{tikzpicture}
    \subcaption{$\norm[\mathrm{a},h]{\Ih u - \uu[h]}$ vs. $h$, triangular mesh\label{fig:convergence:triangular:flux}}
  \end{minipage}
  \hspace{0.5cm}
  \begin{minipage}[b]{0.45\linewidth}\centering
    \begin{tikzpicture}[scale=0.65]
      \begin{loglogaxis}
        \addplot table[x=meshsize,y=err_Gh] {pho_0_pi6_tiltedhexagonal.dat};
        \addplot table[x=meshsize,y=err_Gh] {pho_1_pi6_tiltedhexagonal.dat};
        \addplot table[x=meshsize,y=err_Gh] {pho_2_pi6_tiltedhexagonal.dat};
        \addplot table[x=meshsize,y=err_Gh] {pho_3_pi6_tiltedhexagonal.dat};
        \addplot table[x=meshsize,y=err_Gh] {pho_4_pi6_tiltedhexagonal.dat};
        \logLogSlopeTriangle{0.90}{0.4}{0.1}{1}{black};
        \logLogSlopeTriangle{0.90}{0.4}{0.1}{2}{black};
        \logLogSlopeTriangle{0.90}{0.4}{0.1}{3}{black};
        \logLogSlopeTriangle{0.90}{0.4}{0.1}{4}{black};
        \logLogSlopeTriangle{0.90}{0.4}{0.1}{5}{black};
      \end{loglogaxis}
    \end{tikzpicture}
    \subcaption{$\norm[\mathrm{a},h]{\Ih u - \uu[h]}$ vs. $h$ polygonal mesh\label{fig:convergence:polygonal:flux}}
  \end{minipage}
  \vspace{0.5cm}\\
  \begin{minipage}[b]{0.45\linewidth}\centering
    \begin{tikzpicture}[scale=0.65]
      \begin{loglogaxis}
        \addplot table[x=meshsize,y=err_uh] {pho_0_mesh1.dat};
        \addplot table[x=meshsize,y=err_uh] {pho_1_mesh1.dat};
        \addplot table[x=meshsize,y=err_uh] {pho_2_mesh1.dat};
        \addplot table[x=meshsize,y=err_uh] {pho_3_mesh1.dat};
        \addplot table[x=meshsize,y=err_uh] {pho_4_mesh1.dat};
        \logLogSlopeTriangle{0.90}{0.4}{0.1}{2}{black};
        \logLogSlopeTriangle{0.90}{0.4}{0.1}{3}{black};
        \logLogSlopeTriangle{0.90}{0.4}{0.1}{4}{black};
        \logLogSlopeTriangle{0.90}{0.4}{0.1}{5}{black};
        \logLogSlopeTriangle{0.90}{0.4}{0.1}{6}{black};
      \end{loglogaxis}
    \end{tikzpicture}
    \subcaption{$\norm{\lproj{k}u-u_h}$ vs. $h$, triangular mesh\label{fig:convergence:triangular:L2}}
  \end{minipage}
  \hspace{0.5cm}  
  \begin{minipage}[b]{0.45\linewidth}\centering
    \begin{tikzpicture}[scale=0.65]
      \begin{loglogaxis}
        \addplot table[x=meshsize,y=err_uh] {pho_0_pi6_tiltedhexagonal.dat};
        \addplot table[x=meshsize,y=err_uh] {pho_1_pi6_tiltedhexagonal.dat};
        \addplot table[x=meshsize,y=err_uh] {pho_2_pi6_tiltedhexagonal.dat};
        \addplot table[x=meshsize,y=err_uh] {pho_3_pi6_tiltedhexagonal.dat};
        \addplot table[x=meshsize,y=err_uh] {pho_4_pi6_tiltedhexagonal.dat};
        \logLogSlopeTriangle{0.90}{0.4}{0.1}{2}{black};
        \logLogSlopeTriangle{0.90}{0.4}{0.1}{3}{black};
        \logLogSlopeTriangle{0.90}{0.4}{0.1}{4}{black};
        \logLogSlopeTriangle{0.90}{0.4}{0.1}{5}{black};
        \logLogSlopeTriangle{0.90}{0.4}{0.1}{6}{black};
      \end{loglogaxis}
    \end{tikzpicture}
    \subcaption{$\norm{\lproj{k}u-u_h}$ vs. $h$, polygonal mesh\label{fig:convergence:polygonal:L2}}
  \end{minipage}
  \caption{Error vs. $h$ for the test case of Section~\ref{sec:poisson:num:2d}.\label{fig:poisson.2d}}
\end{figure}
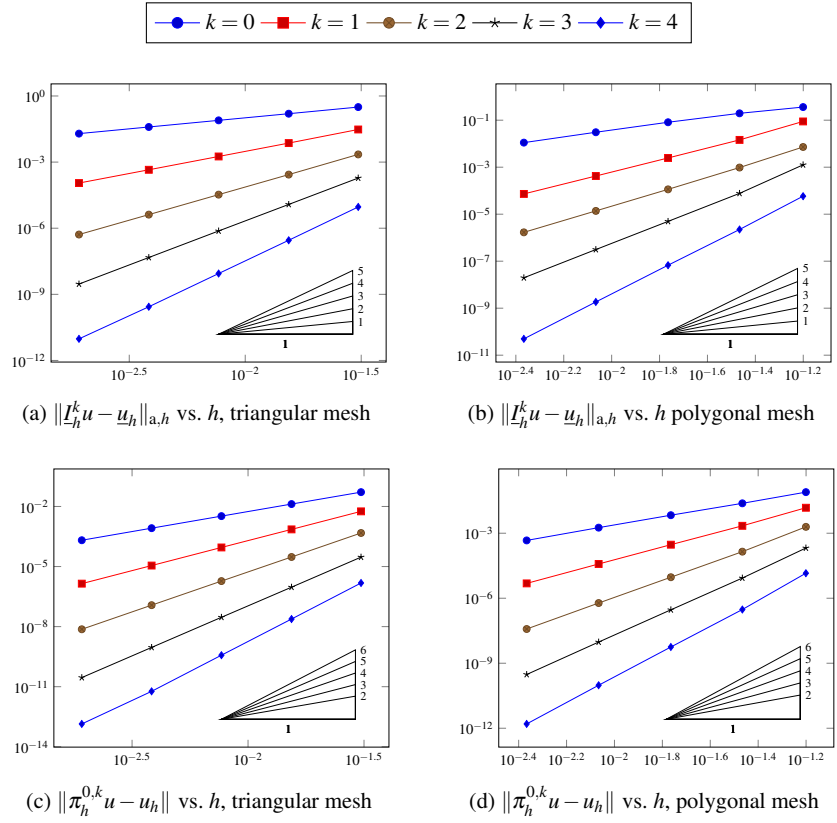

\subsubsection{Three-dimensional test case}\label{sec:poisson:num:3d.smooth}

The second test case, taken from~\cite{Di-Pietro.Specogna:16}, demonstrates the orders of convergence in three space dimensions.
We solve the Dirichlet problem in the unit cube $\Omega=(0,1)^3$ with
$$
u(\vec{x}) = \sin(\pi x_1)\sin(\pi x_2)\sin(\pi x_3),
$$
and corresponding right-hand side $f(\vec{x})=3\pi^2\sin(\pi x_1)\sin(\pi x_2)\sin(\pi x_3)$ on a matching simplicial mesh family for polynomial degrees up to 3.
The numerical results displayed in Fig.~\ref{fig:poisson.3d.smooth:err.vs.h} show asymptotic convergence rates that match those predicted by~\eqref{eq:poisson:en.err.est} and~\eqref{eq:poisson:l2.err.est}.
In Fig.~\ref{fig:poisson.3d.smooth:err.vs.t_tot} we display the error versus the total computational time $t_{\rm tot}$ (including the pre-processing, solution, and post-processing).
It can be seen that the energy- and $L^2$-errors optimally scale as $t_{\rm tot}^{\nicefrac{(k+1)}{d}}$ and $t_{\rm tot}^{\nicefrac{(k+2)}{d}}$ (with $d=3$), respectively.

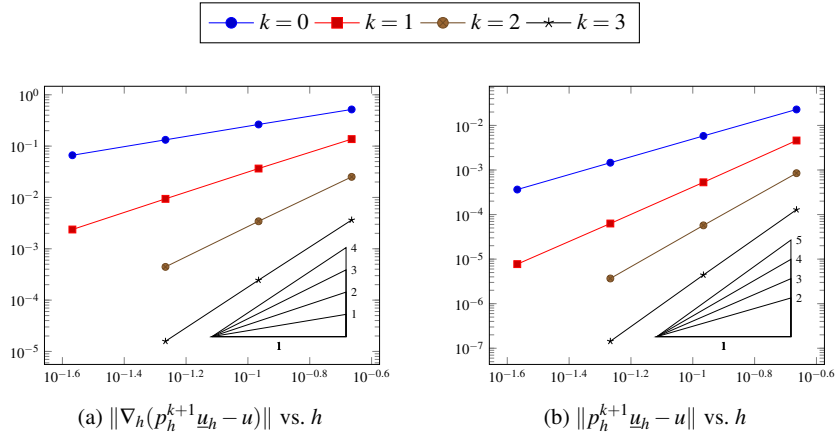
\begin{figure}\centering
  \ref{conv.legend.poisson.3d.smooth}
  \vspace{0.5cm}\\
  \begin{minipage}[b]{0.45\linewidth}\centering
    \begin{tikzpicture}[scale=0.65]
      \begin{loglogaxis}[ legend columns=-1, legend to name=conv.legend.poisson.3d.smooth ]
        \addplot table[x=hmax,y=energy_error] {sinsinsin0.txt};
        \addplot table[x=hmax,y=energy_error] {sinsinsin1.txt};
        \addplot table[x=hmax,y=energy_error] {sinsinsin2.txt};
        \addplot table[x=hmax,y=energy_error] {sinsinsin3.txt};
        \logLogSlopeTriangle{0.90}{0.4}{0.1}{1}{black};
        \logLogSlopeTriangle{0.90}{0.4}{0.1}{2}{black};
        \logLogSlopeTriangle{0.90}{0.4}{0.1}{3}{black};
        \logLogSlopeTriangle{0.90}{0.4}{0.1}{4}{black};
        \legend{$k=0$,$k=1$,$k=2$,$k=3$};
      \end{loglogaxis}
    \end{tikzpicture}
    \subcaption{$\norm{\GRADh(\ph\uu[h]-u)}$ vs. $h$}      
  \end{minipage}
  \hspace{0.5cm}
  \begin{minipage}[b]{0.45\linewidth}\centering
    \begin{tikzpicture}[scale=0.65]
      \begin{loglogaxis}
        \addplot table[x=hmax,y=L2_error] {sinsinsin0.txt};
        \addplot table[x=hmax,y=L2_error] {sinsinsin1.txt};
        \addplot table[x=hmax,y=L2_error] {sinsinsin2.txt};
        \addplot table[x=hmax,y=L2_error] {sinsinsin3.txt};
        \logLogSlopeTriangle{0.90}{0.4}{0.1}{2}{black};
        \logLogSlopeTriangle{0.90}{0.4}{0.1}{3}{black};
        \logLogSlopeTriangle{0.90}{0.4}{0.1}{4}{black};
        \logLogSlopeTriangle{0.90}{0.4}{0.1}{5}{black};
      \end{loglogaxis}
    \end{tikzpicture}
    \subcaption{$\norm{\ph\uu[h]-u}$ vs. $h$}
  \end{minipage}
  \caption{Error vs. $h$ for the test case of Section~\ref{sec:poisson:num:3d.smooth}.\label{fig:poisson.3d.smooth:err.vs.h}}
\end{figure}

\begin{figure}\centering
  \ref{conv.legend.poisson.3d.smooth.time}
  \vspace{0.5cm}\\
  \begin{minipage}[b]{0.45\linewidth}\centering
    \begin{tikzpicture}[scale=0.65]
      \begin{loglogaxis}[ legend columns=-1, legend to name=conv.legend.poisson.3d.smooth.time ]
        \addplot table[x=t_tot,y=energy_error] {sinsinsin0.txt};
        \addplot table[x=t_tot,y=energy_error] {sinsinsin1.txt};
        \addplot table[x=t_tot,y=energy_error] {sinsinsin2.txt};
        \addplot table[x=t_tot,y=energy_error] {sinsinsin3.txt};
        \reverseLogLogSlopeTriangle{0.50}{0.4}{0.1}{1/3}{black};
        \reverseLogLogSlopeTriangle{0.50}{0.4}{0.1}{2/3}{black};
        \reverseLogLogSlopeTriangle{0.50}{0.4}{0.1}{1}{black};
        \reverseLogLogSlopeTriangle{0.50}{0.4}{0.1}{4/3}{black};
        \legend{$k=0$,$k=1$,$k=2$,$k=3$};
      \end{loglogaxis}
    \end{tikzpicture}
    \subcaption{$\norm{\GRADh(\ph\uu[h]-u)}$ vs. $t_{\rm tot}$}
  \end{minipage}
  \hspace{0.5cm}
  \begin{minipage}[b]{0.45\linewidth}\centering
    \begin{tikzpicture}[scale=0.65]
      \begin{loglogaxis}
        \addplot table[x=t_tot,y=L2_error] {sinsinsin0.txt};
        \addplot table[x=t_tot,y=L2_error] {sinsinsin1.txt};
        \addplot table[x=t_tot,y=L2_error] {sinsinsin2.txt};
        \addplot table[x=t_tot,y=L2_error] {sinsinsin3.txt};
        \reverseLogLogSlopeTriangle{0.50}{0.4}{0.1}{2/3}{black};
        \reverseLogLogSlopeTriangle{0.50}{0.4}{0.1}{1}{black};
        \reverseLogLogSlopeTriangle{0.50}{0.4}{0.1}{4/3}{black};
        \reverseLogLogSlopeTriangle{0.50}{0.4}{0.1}{5/3}{black};
        \legend{$k=0$,$k=1$,$k=2$,$k=3$};
      \end{loglogaxis}
    \end{tikzpicture}
    \subcaption{$\norm{\ph\uu[h]-u}$ vs. $t_{\rm tot}$}
  \end{minipage}
  \caption{Error vs. total computational time for the test case of Section~\ref{sec:poisson:num:3d.smooth}.\label{fig:poisson.3d.smooth:err.vs.t_tot}}
\end{figure}
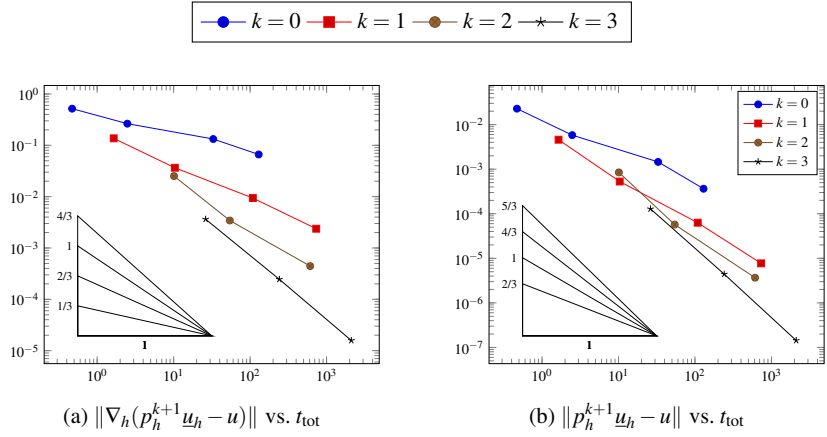

\subsubsection{Three-dimensional case with adaptive mesh refinement}\label{sec:poisson:num:3d.adap}

The third test case, known as Fichera corner benchmark, is taken from~\cite{Di-Pietro.Specogna:16} and is based on the exact solution of~\cite{Fichera:75} on the etched three-dimensional domain $\Omega=(-1,1)^3\setminus[0,1]^3$:
$$
u(\vec{x})= \sqrt[4]{x_1^2+x_2^2+x_3^2},
$$
with right-hand side $f(\vec{x})=-\nicefrac34(x_1^2+x_2^2+x_3^2)^{-\nicefrac34}$.
In this case, the gradient of the solution has a singularity in the origin which prevents the method from attaining optimal convergence rates even for $k=0$.
In Fig.~\ref{fig:poisson:3d.adap} we show a computation comparing the numerical error versus $N_{\rm dof}$ (cf.~\eqref{eq:Ndof}) for the Fichera problem on uniformly and adaptively refined mesh sequences for polynomial degrees up to 3.
Clearly, the order of convergence is limited by the solution regularity when using uniformly refined meshes, while using adaptively refined meshes we recover optimal orders of convergence of $N_{\rm dof}^{\nicefrac{(k+1)}{d}}$ and $N_{\rm dof}^{\nicefrac{(k+2)}{d}}$  (with $d=3$) for the energy- and $L^2$-errors, respectively.

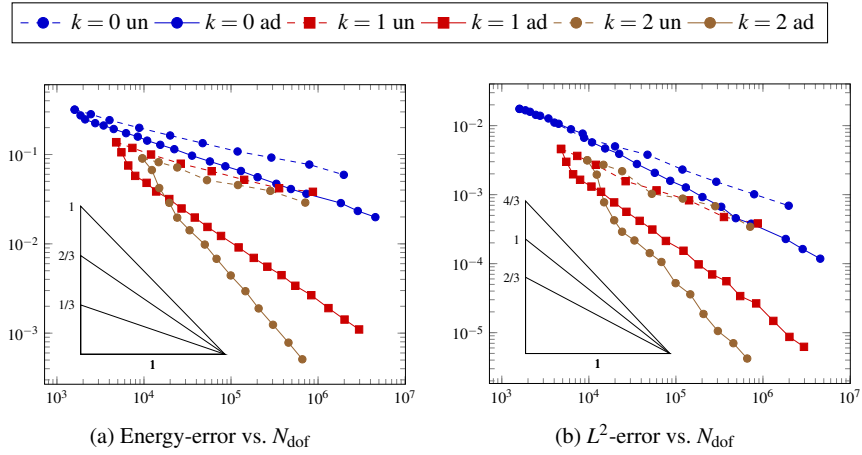
\begin{figure}\centering
  \ref{conv.legend.poisson.3d.adap}
  \vspace{0.5cm}\\
  \begin{minipage}[b]{0.45\linewidth}\centering
    \begin{tikzpicture}[scale=0.70]
      \begin{loglogaxis}[ legend columns=-1, legend to name={conv.legend.poisson.3d.adap} ]
        \addplot[color=blue!80!black,dashed,mark=*,mark options={solid}] table[x=ndofs,y=energy_error] {fichera0s.txt};
        \addplot[color=blue!80!black,mark=*,mark options={solid}] table[x=ndofs,y=energy_error] {fichera0a.txt};
        \addplot[color=red!80!black,dashed,mark=square*,mark options={solid}] table[x=ndofs,y=energy_error] {fichera1s.txt};
        \addplot[color=red!80!black,mark=square*,mark options={solid}] table[x=ndofs,y=energy_error] {fichera1a.txt};
        \addplot[color=brown!80!black,dashed,mark=otimes*,mark options={solid}] table[x=ndofs,y=energy_error] {fichera2s.txt};
        \addplot[color=brown!80!black,mark=otimes*,mark options={solid}] table[x=ndofs,y=energy_error] {fichera2a.txt};        
        \reverseLogLogSlopeTriangle{0.50}{0.4}{0.1}{1/3}{black};
        \reverseLogLogSlopeTriangle{0.50}{0.4}{0.1}{2/3}{black};
        \reverseLogLogSlopeTriangle{0.50}{0.4}{0.1}{1}{black};         
        \legend{$k=0$ un,$k=0$ ad,$k=1$ un,$k=1$ ad,$k=2$ un,$k=2$ ad};
      \end{loglogaxis}
    \end{tikzpicture}
    \subcaption{Energy-error vs. $N_{\rm dof}$}
  \end{minipage}
  \hspace{0.5cm}
  \begin{minipage}[b]{0.45\linewidth}
    \begin{tikzpicture}[scale=0.70]
      \begin{loglogaxis}
        \addplot[color=blue!80!black,dashed,mark=*,mark options={solid}] table[x=ndofs,y=L2_error] {fichera0s.txt};
        \addplot[color=blue!80!black,mark=*,mark options={solid}] table[x=ndofs,y=L2_error] {fichera0a.txt};
        \addplot[color=red!80!black,dashed,mark=square*,mark options={solid}] table[x=ndofs,y=L2_error] {fichera1s.txt};
        \addplot[color=red!80!black,mark=square*,mark options={solid}] table[x=ndofs,y=L2_error] {fichera1a.txt};
        \addplot[color=brown!80!black,dashed,mark=otimes*,mark options={solid}] table[x=ndofs,y=L2_error] {fichera2s.txt};
        \addplot[color=brown!80!black,mark=otimes*,mark options={solid}] table[x=ndofs,y=L2_error] {fichera2a.txt};
        \reverseLogLogSlopeTriangle{0.50}{0.4}{0.1}{2/3}{black};
        \reverseLogLogSlopeTriangle{0.50}{0.4}{0.1}{1}{black};
        \reverseLogLogSlopeTriangle{0.50}{0.4}{0.1}{4/3}{black};
      \end{loglogaxis}
    \end{tikzpicture}
    \subcaption{$L^2$-error vs. $N_{\rm dof}$}
  \end{minipage}
  \caption{Error vs. $N_{\rm dof}$ for the test case of Section~\ref{sec:poisson:num:3d.adap}.\label{fig:poisson:3d.adap}}
\end{figure}


\section{A nonlinear example: The $p$-Laplace equation}\label{sec:plap}

We consider in this section an extension of the HHO method to the $p$-Laplace equation.
This problem will be used to introduce the techniques for the discretization and analysis of nonlinear operators, as well as a set of functional analysis results of independent interest.
An additional interesting point is that the $p$-Laplace problem is naturally posed in a non-Hilbertian setting.
This will require to emulate a Sobolev structure at the discrete level.

Let $p\in (1,+\infty)$ be fixed, and set $p'\eqbydef\frac{p}{p-1}$.
The $p$-Laplace problem reads: Find $u:\Omega\to\Real$ such that
\begin{equation}\label{eq:plap:strong}
  \begin{alignedat}{2}
    -\DIV(\vec{\sigma}(\GRAD u)) &= f &\qquad&\text{in $\Omega$},
    \\
    u &= 0 &\qquad&\text{on $\partial\Omega$},
  \end{alignedat}  
\end{equation}
where $f\in L^{p'}(\Omega)$ is a volumetric source term and the function $\vec{\sigma}:\Real^d\to\Real^d$ is such that
\begin{equation}\label{eq:plap:sigma}
  \vec{\sigma}(\vec{\tau}) \eqbydef |\vec{\tau}|^{p-2}\vec{\tau}.
\end{equation}
The $p$-Laplace equation is a generalization of the Poisson problem considered in Section~\ref{sec:basics}, which corresponds to the choice $p=2$.

Classically, the weak formulation of problem~\eqref{eq:plap:strong} reads:
Find $u\in W^{1,p}_0(\Omega)$ such that, for all $v\in W^{1,p}_0(\Omega)$,
\begin{equation}\label{eq:plap:weak}
  a(u,v) = \int_\Omega f(\vec{x})v(\vec{x}){\rm d}\vec{x},
\end{equation}
where the function $a:W^{1,p}(\Omega)\times W^{1,p}(\Omega)\to\Real$ is such that
\begin{equation}\label{eq:plap:a}
  a(u,v)\eqbydef\int_\Omega\vec{\sigma}(\GRAD u(\vec{x}))\SCAL\GRAD v(\vec{x}){\rm d}\vec{x}.
\end{equation}
From this point on, to alleviate the notation, we omit both the dependence of the integrand on $\vec{x}$ and the differential from integrals.

\subsection{Discrete $W^{1,p}$-norms and Sobolev embeddings}\label{sec:plap:sobolev.embeddings}

In Section~\ref{sec:basics}, the discrete space $\UhD$ and the norm $\norm[1,h]{{\cdot}}$ have played the role of the Hilbert space $H_0^1(\Omega)$ and of the seminorm $\seminorm[H^1(\Omega)]{{\cdot}}$, respectively (notice that $\seminorm[H^1(\Omega)]{{\cdot}}$ is a norm on $H^1_0(\Omega)$ by virtue of the continuous Poincar\'{e} inequality).
For the $p$-Laplace equation, $\UhD$ will replace at the discrete level the Sobolev space $W^{1,p}_0(\Omega)$.
A good candidate for the role of the corresponding seminorm $\seminorm[W^{1,p}(\Omega)]{{\cdot}}$ is the map $\norm[1,p,h]{{\cdot}}$ such that, for all $\uv[h]\in\Uh$,
\begin{equation}\label{eq:norm1p.h}
  \norm[1,p,h]{\uv[h]}^p\eqbydef\sum_{T\in\Th}\norm[1,p,T]{\uv}^p,
\end{equation}
where, for all $T\in\Th$,
\begin{equation}\label{eq:norm1p.T}
  \norm[1,p,T]{\uv}^p\eqbydef
  \norm[L^p(T)^d]{\GRAD v_T}^p
  + \sum_{F\in\Fh[T]}h_F^{1-p}\norm[L^p(F)]{v_F-v_T}^p.
\end{equation}
The power of $h_F$ in the second term ensures that both contributions have the same scaling.
When $p=2$, we recover the seminorm $\norm[1,h]{{\cdot}}$ defined by~\eqref{eq:norm1h}.

The following discrete Sobolev embeddings are proved in~\cite[Proposition~5.4]{Di-Pietro.Droniou:16}.
The proof hinges on the results of~\cite[Theorem~6.1]{Di-Pietro.Ern:10} for broken polynomial spaces (based, in turn, on the techniques originally developed in~\cite{Eymard.Gallouet.ea:10} in the context of finite volume methods).
Their role in the analysis of HHO methods for problem~\eqref{eq:plap:weak} is discussed in Remark~\ref{rem:role.sobolev}.

\begin{svgraybox}
  \begin{theorem}[Discrete Sobolev embeddings]\label{thm:sobolev}
    Let a polynomial degree $k\ge 0$ and an index $p\in (1,+\infty)$ be fixed.
    Let $(\Mh)_{h\in{\cal H}}$ denote a regular sequence of meshes in the sense of Definition~\ref{def:mesh.reg}.
    Let $1\le q\le \frac{dp}{d-p}$ if $1\le p< d$ and $1\le q<+\infty$ if $p\ge d$. Then, there exists a real number $C>0$ only depending on $\Omega$, $\varrho$, $l$, $p$, and $q$ such that, for all $\uv[h]\in\UhD$,
    \begin{equation}\label{eq:sobolev}
      \norm[L^q(\Omega)]{v_h}\le C\norm[1,p,h]{\uv[h]}.
    \end{equation}
  \end{theorem}
\end{svgraybox}

\begin{remark}[Discrete Poincar\'{e} inequality]
  The discrete Poincar\'{e} inequality~\eqref{eq:poincare} is a special case of Theorem~\ref{thm:sobolev} corresponding to $p=q=2$ (this choice is possible in any space dimension).
\end{remark}

\subsection{Discrete gradient and compactness}\label{sec:plap:compactness}

The analysis of numerical methods for linear problems is usually carried out in the spirit of the Lax--Richtmyer equivalence principle: ``For a consistent numerical method, stability is equivalent to convergence''; see for instance~\cite{Dahlquist:56} for a rigorous proof in the case of linear Cauchy problems.
When dealing with nonlinear problems, however, some form of compactness is also required; cf. Remark~\ref{rem:role.compactness} for further insight into this point.
In order to achieve it for problem~\eqref{eq:plap:weak}, we need to introduce a local gradient reconstruction slightly richer than $\GRAD\pT$; see~\eqref{eq:pT}.

Let a mesh element $T\in\Th$ be fixed.
By the principles illustrated in Section~\ref{sec:basics:local:ibp}, we define the local gradient reconstruction $\GT:\UT\to\Poly{k}(T)^d$ such that, for all $\uv[T]\in\UT$,
\begin{equation}\label{eq:GT}
  (\GT\uv[T],\vec{\tau})_T
  = -(v_T,\DIV\vec{\tau})_T + \sum_{F\in\Fh[T]}(v_F,\vec{\tau}\SCAL\normal_{TF})_F
  \quad\forall\vec{\tau}\in\Poly{k}(T)^d.
\end{equation}
Notice that here we reverted to the $L^2$-product notation instead of using integrals to emphasize the fact that the definition of $\GT$ is inherently $L^2$-based.
\begin{remark}[Relation between $\GT$ and $\pT$]
  Taking $\vec{\tau}=\GRAD w$ with $w\in\Poly{k+1}(T)$ in~\eqref{eq:GT} and comparing with~\eqref{eq:pT:1}, it is readily inferred that
  \begin{equation}\label{eq:GT.pT}
    (\GT\uv[T]-\GRAD\pT\uv[T],\GRAD w)_T=0\qquad\forall w\in\Poly{k+1}(T),
  \end{equation}
  i.e., $\GRAD\pT\uv[T]$ is the $L^2$-projection of $\GT\uv[T]$ on $\GRAD\Poly{k+1}(T)\subset\Poly{k}(T)^d$.
  In passing, we observe that for $k=0$, using the fact that $\GRAD\Poly{1}(T)=\Poly{0}(T)^d$,~\eqref{eq:GT.pT} implies that $\GT[0]\uv[T]=\GRAD\pT[1]\uv[T]$.
\end{remark}
Choosing a larger arrival space for $\GT$ has the effect of modifying the commuting property as follows (compare with~\eqref{eq:pT:commuting}): For all $v\in W^{1,1}(T)$,
\begin{equation}\label{eq:GT:commuting}
  (\GT\circ\IT) v = \vlproj[T]{k}(\GRAD v).
\end{equation}
At the global level, we define the operator $\Gh:\Uh\to\Poly{k}(\Th)^d$ such that, for all $\uv[h]\in\Uh$,
\begin{equation}\label{eq:Gh}
  \restrto{(\Gh\uv[h])}{T}\eqbydef\GT\uv[T]\qquad\forall T\in\Th.
\end{equation}

The commuting property~\eqref{eq:GT:commuting} is used in conjunction with the properties of the $L^2$-projector to prove the following lemma, which states the compactness of sequences of HHO functions uniformly bounded in a discrete Sobolev norm.
\begin{svgraybox}
  \begin{lemma}[Discrete compactness]\label{lem:compactness}
    Let a polynomial degree $k\ge 0$ and an index $p\in (1,+\infty)$ be fixed.
    Let $(\Mh)_{h\in{\cal H}}$ denote a regular sequence of meshes in the sense of Definition~\ref{def:mesh.reg}.
    Let $(\uv[h])_{h\in{\cal H}}\in (\UhD)_{h\in{\cal H}}$ be a sequence for which there exists a real number $C>0$ independent of $h$ such that
    $$
    \norm[1,p,h]{\uv[h]}\le C\qquad\forall h\in{\cal H}.
    $$
    Then, there exists $v\in W^{1,p}_0(\Omega)$ such that, up to a subsequence, as $h\to 0$,
    \begin{enumerate}[(i)]
    \item $v_h\to v$ and $\ph\uv[h]\to v$ strongly in $L^q(\Omega)$ for all $1\le q<\frac{dp}{d-p}$ if $1\le p< d$ and $1\le q<+\infty$ if $p\ge d$;
    \item $\Gh\uv[h]\to\GRAD v$ weakly in $L^p(\Omega)^d$.
    \end{enumerate}
  \end{lemma}
\end{svgraybox}

\subsection{Discrete problem and well-posedness}\label{sec:plap:discrete}

The discrete counterpart of the function $a$ defined by~\eqref{eq:plap:a} is the function $\mathrm{a}_h:\Uh\times\Uh\to\Real$ such that, for all $\uu[h],\uv[h]\in\Uh$,
\begin{equation}\label{eq:plap:ah}
  \mathrm{a}_h(\uu[h],\uv[h])
  \eqbydef
  \int_\Omega\vec{\sigma}(\Gh\uu[h])\SCAL\Gh\uv[h]
  + \sum_{T\in\Th}\mathrm{s}_T(\uu,\uv).
\end{equation}
Here, for all $T\in\Th$, $\mathrm{s}_T:\UT\times\UT\to\Real$ is a local stabilization function which can be obtained, e.g., by generalizing~\eqref{eq:sT.hho} to the non-Hilbertian setting:
\begin{multline}\label{eq:plap:sT.hho}
  \mathrm{s}_T(\uu,\uv)\eqbydef
  \\
  \sum_{F\in\Fh[T]}h_F^{1-p}\int_F|\delta_{TF}^k\uu-\delta_T^k\uu|^{p-2}
  (\delta_{TF}^k\uu-\delta_T^k\uu)(\delta_{TF}^k\uv-\delta_T^k\uv).
\end{multline}
The discrete problem reads:
Find $\uu[h]\in\UhD$ such that
\begin{equation}\label{eq:plap:discrete}
  \mathrm{a}_h(\uu[h],\uv[h]) = \int_\Omega fv_h\qquad\forall\uv[h]\in\UhD.
\end{equation}
The following result summarizes~\cite[Theorem~4.5, Remark~4.7, and Proposition~6.1]{Di-Pietro.Droniou:16}.
\begin{lemma}[Well-posedness]\label{lem:plap:well-posedness}
  Problem~\eqref{eq:plap:discrete} admits a unique solution, and there exists a real number $C>0$ independent of $h$, but possibly depending on $\Omega$, $d$, $\varrho$, and $k$, such that, denoting by $p'\eqbydef\frac{p}{p-1}$ the dual exponent of $p$, it holds that
  \begin{equation}\label{eq:plap:a-priori}
    \norm[1,p,h]{\uu[h]}\le C\norm[L^{p'}(\Omega)]{f}^{\frac{1}{p-1}}.
  \end{equation}
\end{lemma}
\begin{remark}[Role of the discrete Sobolev embeddings]\label{rem:role.sobolev}
  The discrete Sobolev embedding~\eqref{eq:sobolev} with $q=p$ is used in the proof of the a priori bound~\eqref{eq:plap:a-priori} to estimate the right-hand side of the discrete problem~\eqref{eq:plap:discrete} after selecting $\uv[h]=\uu[h]$ and using H\"{o}lder's inequality:
  $$
  \int_\Omega fu_h
  \le\norm[L^{p'}(\Omega)]{f}\norm[L^p(\Omega)]{u_h}
  \le\norm[L^{p'}(\Omega)]{f}\norm[1,p,h]{\uu[h]}.
  $$
\end{remark}
\subsection{Convergence and error analysis}\label{sec:plap:convergence}

The following theorem states the convergence of the sequence of solutions to problem~\eqref{eq:plap:discrete} on a regular mesh sequence.
Notice that convergence is proved for exact solutions that display only the minimal regularity $u\in W^{1,p}_0(\Omega)$ required by the weak formulation~\eqref{eq:plap:weak}.
This is an important point when dealing with nonlinear problems, for which further regularity can be hard to prove, and possibly requires assumptions on the data too strong to be matched in practical situations.

\begin{svgraybox}
  \begin{theorem}[Convergence]\label{thm:plap:convergence}
    Let a polynomial degree $k\ge 0$ and an index $p\in (1,+\infty)$ be fixed.
    Let $(\Mh)_{h\in{\cal H}}$ denote a regular sequence of meshes in the sense of Definition~\ref{def:mesh.reg}.
    Let $u\in W^{1,p}_0(\Omega)$ denote the unique solution to~\eqref{eq:plap:weak}, and denote by $(\uu[h])_{h\in{\cal H}}\in (\UhD)_{h\in{\cal H}}$ the sequence of solutions to~\eqref{eq:plap:discrete} on $(\Th)_{h\in{\cal H}}$.
    Then, as $h\to 0$, it holds
    \begin{enumerate}[(i)]
    \item $u_h\to u$ and $\ph\uu[h]\to u$ strongly in $L^q(\Omega)$ for all $1\le q<\frac{dp}{d-p}$ if $1\le p< d$ and $1\le q<+\infty$ if $p\ge d$;
    \item $\Gh\uu[h]\to\GRAD u$ strongly in $L^p(\Omega)^d$.
    \end{enumerate}
  \end{theorem}
\end{svgraybox}
\begin{remark}[Convergence by compactness]\label{rem:role.compactness}
  Convergence proofs by compactness such as that of Theorem~\ref{thm:plap:convergence} proceeed in three steps:
  \begin{inparaenum}[(i)]
  \item an energy estimate on the discrete solution is established;
  \item compactness of the sequence of discrete solutions is inferred from the energy estimate;
  \item the limit is identified as being a solution to the continuous problem.
  \end{inparaenum}
  In our context, the first point corresponds to the a priori bound~\eqref{eq:plap:a-priori}, while the second point relies on the compactness result of Lemma~\ref{lem:compactness}.
  The third step is carried out adapting the techniques of~\cite{Minty:63,Leray.Lions:65}.
\end{remark}
When dealing with high-order methods, it is also important to determine the convergence rates attained when the solution is regular enough (or when adaptive mesh refinement is used, cf. Section~\ref{sec:poisson:num:3d.adap}).
This makes the object of the following result, proved in~\cite[Theorem~7 and Corollary~10]{Di-Pietro.Droniou:17}.
\begin{svgraybox}
  \begin{theorem}[Energy error estimate]\label{thm:plap:en.err.est}
    Under the assumptions and notations of Theorem~\ref{thm:plap:convergence}, and further assuming the regularity $u\in W^{k+2,p}(\Omega)$ and $\vec{\sigma}(\GRAD u)\in W^{k+1,p'}(\Omega)^d$ with $p'\eqbydef\frac{p}{p-1}$, there exists a real number $C>0$ independent of $h$ such that the following holds: If $p\ge 2$,
    \begin{subequations}\label{eq:plap:en.err.est}
      \begin{multline}\label{eq:plap:en.err.est:p>=2}
        \norm[L^p(\Omega)^d]{\GRADh(\ph\uu[h]-u)} + \seminorm[\mathrm{s},h]{\uu[h]}
        \le\\
        C\left[
        h^{k+1}\seminorm[W^{k+2,p}(\Omega)]{u} + h^{\frac{k+1}{p-1}}\left(
        \seminorm[W^{k+2,p}(\Omega)]{u}^{\frac{1}{p-1}}
        + \seminorm[W^{k+1,p'}(\Omega)^d]{\vec{\sigma}(\GRAD u)}^{\frac{1}{p-1}}
        \right)
        \right],
      \end{multline}
      while, if $p<2$,
      \begin{multline}\label{eq:plap:en.err.est:p<2}
        \norm[L^p(\Omega)^d]{\GRADh(\ph\uu[h]-u)} + \seminorm[\mathrm{s},h]{\uu[h]}
        \le \\
        C\left(
        h^{(k+1)(p-1)}\seminorm[W^{k+2,p}(\Omega)]{u}^{p-1}
        + h^{k+1}\seminorm[W^{k+1,p'}(\Omega)^d]{\vec{\sigma}(\GRAD u)}
        \right),
      \end{multline}
    \end{subequations}
    where, recalling the definition~\eqref{eq:plap:sT.hho} of the local stabilization function, we have introduced the seminorm on $\Uh$ such that, for all $\uv[h]\in\Uh$, $\seminorm[\mathrm{s},h]{\uv[h]}^p\eqbydef\sum_{T\in\Th}\mathrm{s}_T(\uv,\uv).$
  \end{theorem}
\end{svgraybox}
\begin{remark}[Order of convergence]
  The asymptotic scaling for the approximation error in the left-hand side of~\eqref{eq:plap:en.err.est} is determined by the leading terms in the right-hand side.
  Using the Bachmann--Landau notation,
  \begin{equation}\label{eq:plap:en.err.est:asymptotic}
    \norm[L^p(\Omega)^d]{\GRADh(\ph\uu[h]-u)} + \seminorm[\mathrm{s},h]{\uu[h]}
    =\begin{cases}
    \mathcal{O}(h^{\frac{k+1}{p-1}}) &\text{if $p\ge 2$},
    \\
    \mathcal{O}(h^{(k+1)(p-1)}) & \text{if $p<2$}.
    \end{cases}
  \end{equation}
  For a discussion of these orders of convergence and a comparison with other methods studied in the literature, we refer the reader to~\cite[Remark~3.3]{Di-Pietro.Droniou:17}.
\end{remark}

\subsection{Numerical example}\label{sec:plap:num}

To illustrate the performance of the HHO method, we solve the $p$-Laplace problem corresponding to the exact solution
$$
u(\vec{x})=\exp(x_1+\pi x_2)
$$
for $p\in\{\nicefrac74,4\}$. This test is taken from ~\cite[Section~4.4]{Di-Pietro.Droniou:16} and~\cite[Section~3.5]{Di-Pietro.Droniou:17}.
The domain is again the unit square $\Omega=(0,1)^2$, and the volumetric source term $f$ is inferred from~\eqref{eq:plap:strong}.
The convergence results for the same triangular and polygonal mesh families of Section~\ref{sec:poisson:num:2d} (see Fig.~\ref{fig:meshes:triangular} and~\ref{fig:meshes:polygonal}) are displayed in Fig.~\ref{fig:plap:num}.
Here, the error is measured by the quantity $\norm[1,p,h]{\Ih u-\uu[h]}$, for which analogous estimates as those in Theorem~\ref{thm:plap:en.err.est} hold.
The error estimate seem sharp for $p=\nicefrac74$, and the asymptotic orders of convergence match the one predicted by the theory.
For $p=4$, better orders of convergence than the asymptotic ones in~\eqref{eq:plap:en.err.est:asymptotic} are observed.
One possible explanation is that the lowest-order terms in the right-hand side of~\eqref{eq:plap:en.err.est} are not yet dominant for the specific problem data and mesh.
Another possibility is that compensations occur among terms that are separately estimated in the proof.

\begin{figure}
  \centering
    \ref{conv.legend.plap}
  \vspace{0.5cm} \\
  \begin{minipage}{0.45\textwidth}
    \begin{tikzpicture}[scale=0.65]
      \begin{loglogaxis}[
          legend columns=-1,
          legend to name=conv.legend.plap,
          legend style={/tikz/every even column/.append style={column sep=0.35cm}}
        ]
        \addplot table[x=meshsize,y=err_p]{plap2_0_mesh1_pl2.dat};
        \addplot table[x=meshsize,y=err_p]{plap2_1_mesh1_pl2.dat};
        \addplot table[x=meshsize,y=err_p]{plap2_2_mesh1_pl2.dat};
        \addplot table[x=meshsize,y=err_p]{plap2_3_mesh1_pl2.dat};
        \logLogSlopeTriangle{0.90}{0.4}{0.1}{3/4}{black};
        \logLogSlopeTriangle{0.90}{0.4}{0.1}{3/2}{black};
        \logLogSlopeTriangle{0.90}{0.4}{0.1}{9/4}{black};
        \logLogSlopeTriangle{0.90}{0.4}{0.1}{3}{black};
        \legend{$k=0$,$k=1$,$k=2$,$k=3$}
      \end{loglogaxis}
    \end{tikzpicture}
    \subcaption{Triangular, $p=\nicefrac74$}
  \end{minipage}
  \begin{minipage}{0.45\textwidth}
    \begin{tikzpicture}[scale=0.65]
      \begin{loglogaxis}
        \addplot table[x=meshsize,y=err_p]{plap2_0_pi6_tiltedhexagonal_pl2.dat};
        \addplot table[x=meshsize,y=err_p]{plap2_1_pi6_tiltedhexagonal_pl2.dat};
        \addplot table[x=meshsize,y=err_p]{plap2_2_pi6_tiltedhexagonal_pl2.dat};
        \addplot table[x=meshsize,y=err_p]{plap2_3_pi6_tiltedhexagonal_pl2.dat};
        \logLogSlopeTriangle{0.90}{0.4}{0.1}{3/4}{black};
        \logLogSlopeTriangle{0.90}{0.4}{0.1}{3/2}{black};
        \logLogSlopeTriangle{0.90}{0.4}{0.1}{9/4}{black};
        \logLogSlopeTriangle{0.90}{0.4}{0.1}{3}{black};
      \end{loglogaxis}
    \end{tikzpicture}
    \subcaption{Hexagonal, $p=\nicefrac74$}
  \end{minipage} 
  \vspace{0.5cm}\\
  \begin{minipage}{0.45\textwidth}  
    \begin{tikzpicture}[scale=0.65]
      \begin{loglogaxis}
        \addplot table[x=meshsize,y=err_p4]{plap_0_mesh1_pg2.dat};
        \addplot table[x=meshsize,y=err_p4]{plap_1_mesh1_pg2.dat};
        \addplot table[x=meshsize,y=err_p4]{plap_2_mesh1_pg2.dat};
        \addplot table[x=meshsize,y=err_p4]{plap_3_mesh1_pg2.dat};
        \logLogSlopeTriangle{0.85}{0.4}{0.1}{1/3}{black};
        \logLogSlopeTriangle{0.85}{0.4}{0.1}{2/3}{black};
        \logLogSlopeTriangle{0.85}{0.4}{0.1}{1}{black};
        \logLogSlopeTriangle{0.85}{0.4}{0.1}{4/3}{black};        
      \end{loglogaxis}
    \end{tikzpicture}
    \subcaption{Triangular, $p=4$}    
  \end{minipage}  
  \begin{minipage}{0.45\textwidth}  
    \begin{tikzpicture}[scale=0.65]
      \begin{loglogaxis}
        \addplot table[x=meshsize,y=err_p4]{plap_0_pi6_tiltedhexagonal_pg2.dat};
        \addplot table[x=meshsize,y=err_p4]{plap_1_pi6_tiltedhexagonal_pg2.dat};
        \addplot table[x=meshsize,y=err_p4]{plap_2_pi6_tiltedhexagonal_pg2.dat};
        \addplot table[x=meshsize,y=err_p4]{plap_3_pi6_tiltedhexagonal_pg2.dat};
        \logLogSlopeTriangle{0.85}{0.4}{0.1}{1/3}{black};
        \logLogSlopeTriangle{0.85}{0.4}{0.1}{2/3}{black};
        \logLogSlopeTriangle{0.85}{0.4}{0.1}{1}{black};
        \logLogSlopeTriangle{0.85}{0.4}{0.1}{4/3}{black};        
      \end{loglogaxis}
    \end{tikzpicture}
    \subcaption{Hexagonal, $p=4$}    
  \end{minipage}  
  \caption{$\norm[1,p,h]{\Ih u-\uu[h]}$ vs. $h$ for the test case of Section~\ref{sec:plap:num}.\label{fig:plap:num}}
\end{figure}


\section{Diffusion-advection-reaction}\label{sec:adr}

In this section we extend the HHO method to the scalar diffusion-advection-reaction problem: Find $u:\Omega\to\Real$ such that
\begin{subequations}
  \begin{alignat*}{2}
    \DIV(-\diff\GRAD u + \vel u) + \reac u &= f &\qquad& \text{in $\Omega$,} 
    \\   
    u &= 0 &\qquad& \text{on $\partial\Omega$,}
  \end{alignat*}
\end{subequations}
where
\begin{inparaenum}[(i)]
\item $\diff:\Omega\to\Real_+^*$ is the diffusion coefficient, which we assume piecewise constant on a fixed partition of the domain $P_\Omega=\{\omega\}$ and uniformly elliptic;
\item $\vel\in{\rm Lip}(\Omega)^d$ (hence, in particular, $\vel \in W^{1,\infty}(\Omega)^d$) is the advective velocity field, for which we additionally assume, for the sake of simplicity, $\DIV\vel\equiv0$;
\item $\reac\in L^\infty(\Omega)$ is the reaction coefficient such that $\reac\ge\reac_0>0$ a.e. in $\Omega$ for some real number $\reac_0$;
\item $f\in L^2(\Omega)$ is the volumetric source term.
\end{inparaenum}

Having assumed $\diff$ uniformly elliptic, the following weak formulation classically holds:
Find $u\in H_0^1(\Omega)$ such that
\begin{equation}\label{eq:advdiffreac:weak}
  a_{\diff,\vel,\reac}(u,v) = (f,v)\qquad\forall v\in H_0^1(\Omega),
\end{equation}
where the bilinear form $a_{\diff,\vel,\reac}:H^1(\Omega)\times H^1(\Omega)\to\Real$ is such that
$$
a_{\diff,\vel,\reac}(u,v)\eqbydef
a_{\diff}(u,v) + a_{\vel,\reac}(u,v),
$$
and the diffusive and advective-reactive contributions are respectively defined by
$$
  a_{\diff}(u,v)\eqbydef(\diff\GRAD u, \GRAD v),\qquad
  a_{\vel,\reac}(u,v)\eqbydef \tfrac12 (\vel \SCAL \GRAD u, v) - \tfrac12 (u, \vel \SCAL \GRAD v) 
  + (\reac u,v).
$$

The first novel ingredient introduced in this section is the robust HHO discretization of first-order terms.
Problem~\eqref{eq:advdiffreac:weak} is characterized by the presence of spatially varying coefficients, which can give rise to different regimes in different regions of the domain.
In practice, one is typically interested in numerical methods that handle in a robust way locally dominant advection, corresponding to large values of a local P\'{e}clet number.
As pointed out in~\cite{Di-Pietro.Ern.ea:08}, this requires that the discrete counterpart of the bilinear form $a_{\vel,\reac}$ satisfies a stability condition that guarantees well-posedness even in the absence of diffusion.
This is realized here combining a reconstruction of the advective derivative obtained in the HHO spirit with an upwind stabilization that penalizes the differences between face- and element-based DOFs.

The second novelty introduced in this section is a formulation of diffusive terms with weakly enforced boundary conditions.
A relevant feature of problem~\eqref{eq:advdiffreac:weak} is that boundary layers can appear in the vicinity of the outflow portion of $\partial\Omega$ when the diffusion coefficient takes small values.
To improve the numerical approximation in this situation, one can resort to weakly enforced boundary conditions, which do not constrain the numerical solution to a fixed boundary value.

The following material is closely inspired by~\cite{Di-Pietro.Droniou.ea:15}, where locally vanishing diffusion is treated (see Remark~\ref{rem:locally.degenerate}), and more general formulations for the advective stabilization term are considered. 

\subsection{Discretization of diffusive terms with weakly enforced boundary conditions}\label{sec:adr:weak.bc}

To avoid dealing with jumps of the diffusion coefficient inside the elements when writing the HHO discretization of problem~\eqref{eq:advdiffreac:weak} on a mesh $\Mh=(\Th,\Fh)$, we make the following
\begin{assumption}[Compatible mesh]
  The mesh $\Mh=(\Th,\Fh)$ is compatible with the diffusion coefficient, i.e., for all $T\in\Th$, there exists a unique subdomain $\omega\in P_\Omega$ such that $T\subset\omega$.
  For all $T\in\Th$ we set, for the sake of brevity, $\diff[T]\eqbydef\restrto{\diff}{T}$.
\end{assumption}

Letting $\zeta>0$ denote a user-dependent boundary penalty parameter, we define the discrete diffusive bilinear form $\mathrm{a}_{\diff,h}:\Uh\times\Uh\to\Real$ such that
\begin{equation}\label{eq:avuh}
  \begin{aligned}
    &\mathrm{a}_{\diff,h}(\uu[h],\uv[h])
    \eqbydef
    \sum_{T\in\Th}\diff[T]\mathrm{a}_T(\uu,\uv)
    \\
    &\quad + \hspace{-1ex}\sum_{F \in \Fhb}
    \left\{-(\diff[T_F] \GRAD p_{T_F}^{k+1} \uu ,v_F)_F
    + (u_F, \diff[T_F] \GRAD p_{T_F}^{k+1} \uv)_F
    +  \frac{\zeta\diff[T_F]}{h_F} (u_F,v_F)_F
    \right\}, 
  \end{aligned}
\end{equation}
where, for all mesh elements $T\in\Th$, $\mathrm{a}_T$ is the local diffusive bilinear form defined by~\eqref{eq:aT} and, for all boundary faces $F\in\Fhb$, $T_F$ denotes the unique mesh element such that $F\subset\partial T_F$.
The terms in the second line of~\eqref{eq:avuh} are responsible for the weak enforcement of boundary conditions \`{a} la Nitsche.

Define the diffusion-weighted norm on $\Uh$ such that, for all $\uv[h]\in\Uh$, letting $\norm[\mathrm{a},T]{\uv}^2\eqbydef\mathrm{a}_T(\uv,\uv)$,
$$
\norm[\diff,h]{\uv[h]}^2\eqbydef
\sum_{T\in\Th}\diff[T]\norm[\mathrm{a},T]{\uv[T]}^2 + \sum_{F\in\Fhb}\frac{\diff[T_F]}{h_F}\norm[F]{v_F}^2.
$$
It is a simple matter to check that, for all $\zeta\ge 1$, we have the following coercivity property for $\mathrm{a}_{\diff,h}$:
For all $\uv[h]\in\Uh$,
\begin{equation}\label{eq:coer:adiffh}
  \norm[\diff,h]{\uv[h]}^2\le\mathrm{a}_{\diff,h}(\uv[h],\uv[h]).
\end{equation}

\subsection{Discretization of advective terms with upwind stabilization}\label{sec:adr:advection}

We introduce the ingredients for the discretization of first-order terms: a local advective derivative reconstruction and an upwind stabilization term penalizing the differences between face- and element-based DOFs.

\subsubsection{Local contribution}
Let a mesh element $T\in\Th$ be fixed.
By the principles illustrated in Section~\ref{sec:basics:local:ibp}, we define the local discrete advective derivative reconstruction $\Gvel{T}:\UT\to\Poly{k}(T)$ such that, for all $\uv[T]\in\UT$,
$$
(\Gvel{T}\uv[T],w)_T
= -(v_T,\vel\SCAL\GRAD w)_T + \sum_{F\in\Fh[T]}((\vel\SCAL \normal_{TF})v_F,w)_F
\quad\forall w \in\Poly{k}(T).
$$
The local advective-reactive bilinear form $\mathrm{a}_{\vel,\reac,T}:\UT\times\UT\to\Real$ is defined as follows:
\begin{equation}\label{eq:abetaT}
  \mathrm{a}_{\vel,\reac,T}(\uu,\uv)
  \eqbydef
  \frac12 (\Gvel{T}\uu,v_T)_T - \frac12 (u_T,\Gvel{T}\uv)_T + \mathrm{s}_{\vel,T}(\uu,\uv) + (\reac u_T,v_T)_T,
\end{equation}
where the bilinear form
\begin{equation}\label{eq:svelT}
  \mathrm{s}_{\vel,T}(\uu,\uv)\eqbydef
  \frac12\sum_{F \in \Fh[T]}(| \vel\SCAL\normal_{TF} | (u_F-u_T),v_F-v_T)_F,
\end{equation}
can be interpreted as an upwind stabilization term.
\begin{remark}[Element-face upwind stabilization]
  Upwinding is realized here by penalizing the difference between face- and element-based DOFs.
  This is a relevant difference with respect to classical (cell-based) finite volume and discontinuous Galerkin methods, where jumps of element-based DOFs are considered instead.
  With the choice~\eqref{eq:svelT} for the stabilization term, the stencil remains the same as for a pure diffusion problem, and static condensation of element-based DOFs in the spirit of Section~\ref{sec:basics:local:implementation} remains possible.
  In the context of the lowest-order Hybrid Mixed Mimetic methods, face-element upwind terms have been considered in~\cite{Beirao-da-Veiga.Droniou.ea:11}.
\end{remark}
To express the stability properties of $\mathrm{a}_{\vel,\reac,T}$, we define the local seminorm such that, for all $\uv\in\UT$,
$$
\norm[\vel,\reac,T]{\uv}^2 \eqbydef 
\frac12\sum_{F \in \Fh[T]} \norm[F]{| \vel\SCAL\normal_{TF} |^{\nicefrac12} (v_F-v_T)}^2 
+ \tauref^{-1}\norm[T]{v_T}^2,
$$
where, letting ${\rm L}_{\beta,T} \eqbydef \max_{1\le i\le d} \norm[L^\infty(T)^d]{\GRAD\beta_i}$, we have introduced the reference time
$$
\tauref \eqbydef \{ \max ( \norm[L^\infty(T)]{\reac},{\rm L}_{\beta,T})  \}^{-1}.
$$
Notice that the map $\norm[\vel,\reac,T]{{\cdot}}$ is actually a norm on $\UT$ provided that $\restrto{\vel}{F}\SCAL\normal_{TF}$ is nonzero a.e. on each $F\in\Fh[T]$.
For all $\uv\in\UT$, letting $\uu=\uv$ in~\eqref{eq:abetaT}, it can be easily checked that the following coercivity property holds:
\begin{equation}\label{eq:coer:abetaT}
  \min(1,\tauref\reac_0)\norm[\vel,\reac,T]{\uv}^2\le\mathrm{a}_{\vel,\reac,T}(\uv,\uv).
\end{equation}

\subsubsection{Global advective-reactive bilinear form} 
The global advective-reactive bilinear form is given by 
\begin{equation}\label{eq:abetah}
  \mathrm{a}_{\vel,\reac,h}(\uu[h],\uv[h])
  \eqbydef
  \sum_{T\in\Th}\mathrm{a}_{\vel,\reac,T}(\uu,\uv) + \frac12\sum_{F \in \Fhb}(|\vel\SCAL\normal|u_F,v_F)_F,
\end{equation}
where the first term results from the assembly of elementary contributions, while the second term is responsible for the enforcement of the boundary condition on the inflow portion of $\partial\Omega$. 
\begin{remark}[Link with the advective-reactive bilinear form of~\cite{Di-Pietro.Droniou.ea:15}]
  The bilinear form $\mathrm{a}_{\vel,\reac,h}$ defined by~\eqref{eq:abetah} admits the following equivalent reformulation, which corresponds to~\cite[Eq.~(16)]{Di-Pietro.Droniou.ea:15} when the upwind stabilization discussed in Section~4.2 therein is used:
  \begin{multline}\label{eq:abetah.tilde}
    \mathrm{a}_{\vel,\reac,h}(\uu[h],\uv[h])
    = \sum_{T\in\Th}\bigg( -(u_T,\Gvel{T}\uv)_T + \sum_{F \in \Fh[T]}((\vel\SCAL\normal_{TF})^- (u_F-u_T),v_F-v_T)_F\bigg)
    \\
    + \sum_{T\in\Th} (\reac u_T,v_T)_T + \sum_{F \in \Fhb}((\vel\SCAL\normal)^+u_F,v_F)_F,
  \end{multline}
  where, for any real number $\alpha$, we have set $\alpha^\pm\eqbydef\frac12\left(|\alpha|\pm\alpha\right)$.
  As a matter of fact, recalling the discrete integration by parts formula~\cite[Eq. (35)]{Di-Pietro.Droniou.ea:15},
  \begin{align*}
    \sum_{T\in\Th} (u_T,\Gvel{T}\uv)_T = &
    - \sum_{T\in\Th} (\Gvel{T}\uu,v_T)_T
    - \sum_{T\in\Th}\sum_{F\in\Fh[T]} ( (\vel\SCAL\normal_{TF})(u_F-u_T),v_F-v_T )_F \\
    &+\sum_{F\in\Fhb} ( (\vel\SCAL\normal_{TF})u_F,v_F)_F , 
  \end{align*}
  we can reformulate the first term in the right-hand side of~\eqref{eq:abetah.tilde} as follows:
  \begin{align*}
    \sum_{T\in\Th} -(u_T,\Gvel{T}\uv)_T & =
    \sum_{T\in\Th} \bigg(- \frac12(u_T,\Gvel{T}\uv)_T - \frac12(u_T,\Gvel{T}\uv)_T \bigg)\\
    & =  \sum_{T\in\Th}\bigg(- \frac12(u_T,\Gvel{T}\uv)_T + \frac12  (\Gvel{T}\uu,v_T)_T \bigg) \\
    &\qquad + \frac12 \sum_{T\in\Th}\sum_{F\in\Fh[T]} ( (\vel\SCAL\normal_{TF})(u_F-u_T),v_F-v_T )_F
    \\
    &\qquad - \frac12 \sum_{F\in\Fhb} ( (\vel\SCAL\normal_{TF})u_F,v_F)_F. 
  \end{align*}	
  Inserting this equality into~\eqref{eq:abetah.tilde} and rearranging the terms we recover~\eqref{eq:abetah}.
  The formulation~\eqref{eq:abetah} highlights two key properties of the bilinear form $\mathrm{a}_{\vel,\reac,h}$: its positivity and the skew-symmetric nature of the consistent term.
  The reformulation~\eqref{eq:abetah.tilde}, on the other hand, has a more familiar look for the reader accustomed to upwind stabilization terms.
\end{remark}
Define the global advective-reactive norm such that, for all $\uv[h]\in\Uh$,
$$
\norm[\vel,\reac,h]{\uv[h]}^2 \eqbydef \sum_{T\in\Th} \norm[\vel,\reac,T]{\uv}^2 
+ \frac12 \sum_{F\in\Fhb}  \norm[F]{| \vel\SCAL\normal |^{\nicefrac12} v_F}^2.
$$
The following coercivity result for $\mathrm{a}_{\vel,\reac,h}$ follows from~\eqref{eq:coer:abetaT}:
For all $\uv[h]\in\Uh$ 
\begin{equation}\label{eq:coer:abetah}
  \min_{T\in\Th}(1, \tauref\reac_0) \norm[\vel,\reac,h]{\uv[h]}^2 \leq \mathrm{a}_{\vel,\reac,h}(\uv[h],\uv[h]).
\end{equation}

\subsection{Global problem and inf-sup stability}\label{sec:adr:global.problem}

We can now define the global bilinear form $\mathrm{a}_{\diff,\vel,\reac,h}: \Uh\times\Uh\to\Real$ combining the diffusive and advective-reactive contributions defined above:
$$
\mathrm{a}_{\diff,\vel,\reac,h}(\uu[h],\uv[h]) \eqbydef \mathrm{a}_{\diff,h}(\uu[h],\uv[h]) + \mathrm{a}_{\vel,\reac,h}(\uu[h],\uv[h]). 
$$
The HHO approximation of~\eqref{eq:advdiffreac:weak} then reads:
Find $\uu[h]\in\Uh$ such that, for all $\uv[h]\in\Uh$, 
\begin{equation}\label{eq:advdiffreac:discrete}
  \mathrm{a}_{\diff,\vel,\reac,h}(\uu[h],\uv[h]) = (f,v_h).
\end{equation}

Let us examine stability.
In view of~\eqref{eq:coer:adiffh} and~\eqref{eq:coer:abetah}, the bilinear form $\mathrm{a}_{\diff,\vel,\reac,h}$ is clearly coercive with respect to the norm
$$
\norm[\flat,h]{\uv[h]}^2\eqbydef\norm[\diff,h]{\uv[h]}^2 + \norm[\vel,\reac,h]{\uv[h]}^2,
$$
which guarantees that problem~\eqref{eq:advdiffreac:discrete} has a unique solution.
This norm, however, does not convey any information on the discrete advective derivative.
A stronger stability result is stated in the following lemma, where we consider the augmented norm
$$
\norm[\sharp,h]{\uv[h]}^2\eqbydef\norm[\flat,h]{\uv[h]}^2
+ \sum_{T\in\Th,\velref\not\equiv 0}h_T\velref^{-1}\norm[T]{\Gvel\uv}^2,
$$
with $\velref\eqbydef\norm[L^\infty(T)^d]{\vel}$ denoting the reference velocity on $T$.
\begin{lemma}[Inf-sup stability of $\mathrm{a}_{\diff,\vel,\reac,h}$]\label{lem:advdiffreac:stab}
  Assume that $\zeta\ge 1$ and that, for all $T\in\Th$,
  \begin{equation}\label{eq:stability.cond}
    h_T \max({\rm L}_{\beta,T},\reac_0) \le \velref. 
  \end{equation}
  Then, there exists a real number $C>0$, independent of $h,\diff,\vel$ and $\reac$, but possibly depending on $d$, $\varrho$, and $k$ such that, for all $\uw[h]\in\Uh$,
  $$	
  C\min_{T\in\Th}(1, \tauref\reac_0) \norm[\sharp,h]{ \uw[h] } \le 
  \sup_{\uv[h]\in\Uh\backslash{\{\underline{0}_h\}}} \frac{\mathrm{a}_{\diff,\vel,\reac,h}(\uw[h],\uv[h])}{\norm[\sharp,h]{\uv[h]}}. 
  $$
\end{lemma}

\begin{remark}[Condition~\eqref{eq:stability.cond}]
  Condition~\eqref{eq:stability.cond} means
  \begin{inparaenum}[(i)]
  \item that the advective field is well-resolved by the mesh and
  \item that reaction is not dominant.
  \end{inparaenum}
\end{remark}

\subsection{Convergence}\label{sec:adr:convergence}

For each mesh element $T\in\Th$, we introduce the local P{\'e}clet number such that
$$
\Pe_T \eqbydef \max_{F\in\Fh[T]} \frac{h_F\norm[L^\infty(F)]{\restrto{\vel}{F}\SCAL\normal_{TF}}}{\diff[F]},
$$
where $\diff[F] \eqbydef \min_{T\in\Th[F]} \diff[T]$. 
For the mesh elements where diffusion dominates we have $\Pe_T\le h_T$, for those where advection dominates we have $\Pe_T\ge 1$, while intermediate regimes correspond to $\Pe_T\in(h_T,1)$.

The following error estimate accounts for the variation of the convergence rate according to the value of the local P\'{e}clet number, showing that diffusion-dominated elements contribute with a term in $\mathcal{O}(h_T^{k+1})$ (as for a pure diffusion problem), whereas convection-dominated elements contribute with a term in $\mathcal{O}(h_T^{k+\nicefrac12})$ (as for a pure advection problem).
\begin{svgraybox}
  \begin{theorem}[Energy error estimate]\label{thm:advdiffreac:en.err.est}
    Let $u$ solve \eqref{eq:advdiffreac:weak} and $\uu[h]$ solve \eqref{eq:advdiffreac:discrete}. 
    Under the assumptions of Lemma~\ref{lem:advdiffreac:stab}, 
    and further assuming the regularity $\restrto{u}{T}\in H^{k+2}(T)$ for all $T\in\Th$, 
    there exists a real number $C>0$ independent of $h,\diff,\vel$, and $\reac$, but possibly depending on $\rho,d$, and $k$, such that 
    $$
    \begin{aligned}
      C\min_{T\in\Th}(1, \tauref\reac_0) \norm[\sharp,h]{ \uhu - \uu[h] }
      \le\Bigg\{ \sum_{T\in\Th} \bigg[  
    	\left( \diff[T]\norm[H^{k+2}(T)]{u}^2 + \tauref^{-1}\norm[H^{k+1}(T)]{u}^2\right) &h_{T}^{2(k+1)}
        \\	    		
        +\velref \min(1,\Pe_T) \norm[H^{k+1}(T)]{u}^2 &h_{T}^{2k+1}
    	\bigg]
      \Bigg\}^{\nicefrac12}. 
    \end{aligned}
    $$
  \end{theorem}
\end{svgraybox}

\begin{remark}[Extension to locally vanishing diffusion]\label{rem:locally.degenerate}
  It has been showed in~\cite{Di-Pietro.Droniou.ea:15} that the error estimate of Theorem~\ref{thm:advdiffreac:en.err.est} extends to locally vanishing diffusion provided that we conventionally set $\Pe_T=+\infty$ for any element $T\in\Th$ such that $\diff[F]=0$ for some $F\in\Fh[T]$.
\end{remark}

\subsection{Numerical example}\label{sec:adr:numerical.example}
To illustrate the performance of the HHO method, we solve in the unit square $\Omega=(0,1)^2$ the Dirichlet problem corresponding to the solution~\eqref{eq:poisson:num:2d:ex.sol} with $\vel(\vec{x})= (\nicefrac12-x_2,x_1-\nicefrac12)$, $\mu\equiv 1$, and a uniform diffusion coefficient $\diff$ taking values in $\{1,\pgfmathprintnumber{1e-3},0\}$.
We take triangular and predominantly hexagonal meshes, as depicted in Figures~\ref{fig:meshes:triangular} and~\ref{fig:meshes:polygonal} respectively.
\begin{figure}\centering
  \ref{conv.legend.adr.2d}
  \vspace{0.5cm}\\
  \begin{minipage}[b]{0.32\linewidth}\centering
    \begin{tikzpicture}[scale=0.45]
      \begin{loglogaxis}[ legend columns=-1, legend to name=conv.legend.adr.2d ]       
        \addplot table[x=meshsize,y=err_sharp] {adho_1_0_mesh1.dat};
        \addplot table[x=meshsize,y=err_sharp] {adho_1_1_mesh1.dat};
        \addplot table[x=meshsize,y=err_sharp] {adho_1_2_mesh1.dat};
        \addplot table[x=meshsize,y=err_sharp] {adho_1_3_mesh1.dat};
        \logLogSlopeTriangle{0.90}{0.4}{0.1}{1}{black};
        \logLogSlopeTriangle{0.90}{0.4}{0.1}{2}{black};
        \logLogSlopeTriangle{0.90}{0.4}{0.1}{3}{black};
        \logLogSlopeTriangle{0.90}{0.4}{0.1}{4}{black};
        \legend{$k=0$,$k=1$,$k=2$,$k=3$};
      \end{loglogaxis}
    \end{tikzpicture}
    \subcaption{$\kappa=1$, triangular}
  \end{minipage}
  \begin{minipage}[b]{0.32\linewidth}\centering
    \begin{tikzpicture}[scale=0.45]
      \begin{loglogaxis}
        \addplot table[x=meshsize,y=err_sharp] {adho_0.001_0_mesh1.dat};
        \addplot table[x=meshsize,y=err_sharp] {adho_0.001_1_mesh1.dat};
        \addplot table[x=meshsize,y=err_sharp] {adho_0.001_2_mesh1.dat};
        \addplot table[x=meshsize,y=err_sharp] {adho_0.001_3_mesh1.dat};
        \logLogSlopeTriangle{0.90}{0.4}{0.1}{1}{black};
        \logLogSlopeTriangle{0.90}{0.4}{0.1}{2}{black};
        \logLogSlopeTriangle{0.90}{0.4}{0.1}{3}{black};
        \logLogSlopeTriangle{0.90}{0.4}{0.1}{4}{black};
      \end{loglogaxis}
    \end{tikzpicture}
    \subcaption{$\kappa=\pgfmathprintnumber{1e-3}$, triangular}
  \end{minipage}
  \begin{minipage}[b]{0.32\linewidth}\centering
    \begin{tikzpicture}[scale=0.45]
      \begin{loglogaxis}
        \addplot table[x=meshsize,y=err_sharp] {adho_0_0_mesh1.dat};
        \addplot table[x=meshsize,y=err_sharp] {adho_0_1_mesh1.dat};
        \addplot table[x=meshsize,y=err_sharp] {adho_0_2_mesh1.dat};
        \addplot table[x=meshsize,y=err_sharp] {adho_0_3_mesh1.dat};
        \logLogSlopeTriangle{0.90}{0.4}{0.1}{1/2}{black};
        \logLogSlopeTriangle{0.90}{0.4}{0.1}{3/2}{black};
        \logLogSlopeTriangle{0.90}{0.4}{0.1}{5/2}{black};
        \logLogSlopeTriangle{0.90}{0.4}{0.1}{7/2}{black};
      \end{loglogaxis}
    \end{tikzpicture}
    \subcaption{$\kappa=0$, triangular}
  \end{minipage}
  \vspace{0.5cm}\\
  \begin{minipage}[b]{0.32\linewidth}\centering
    \begin{tikzpicture}[scale=0.45]
      \begin{loglogaxis}
        \addplot table[x=meshsize,y=err_sharp] {adho_1_0_pi6_tiltedhexagonal.dat};
        \addplot table[x=meshsize,y=err_sharp] {adho_1_1_pi6_tiltedhexagonal.dat};
        \addplot table[x=meshsize,y=err_sharp] {adho_1_2_pi6_tiltedhexagonal.dat};
        \addplot table[x=meshsize,y=err_sharp] {adho_1_3_pi6_tiltedhexagonal.dat};
        \logLogSlopeTriangle{0.90}{0.4}{0.1}{1}{black};
        \logLogSlopeTriangle{0.90}{0.4}{0.1}{2}{black};
        \logLogSlopeTriangle{0.90}{0.4}{0.1}{3}{black};
        \logLogSlopeTriangle{0.90}{0.4}{0.1}{4}{black};
      \end{loglogaxis}
    \end{tikzpicture}
    \subcaption{$\kappa=1$, polygonal}
  \end{minipage}
  \begin{minipage}[b]{0.32\linewidth}\centering
    \begin{tikzpicture}[scale=0.45]
      \begin{loglogaxis}
        \addplot table[x=meshsize,y=err_sharp] {adho_0.001_0_pi6_tiltedhexagonal.dat};
        \addplot table[x=meshsize,y=err_sharp] {adho_0.001_1_pi6_tiltedhexagonal.dat};
        \addplot table[x=meshsize,y=err_sharp] {adho_0.001_2_pi6_tiltedhexagonal.dat};
        \addplot table[x=meshsize,y=err_sharp] {adho_0.001_3_pi6_tiltedhexagonal.dat};
        \logLogSlopeTriangle{0.90}{0.4}{0.1}{1}{black};
        \logLogSlopeTriangle{0.90}{0.4}{0.1}{2}{black};
        \logLogSlopeTriangle{0.90}{0.4}{0.1}{3}{black};
        \logLogSlopeTriangle{0.90}{0.4}{0.1}{4}{black};
      \end{loglogaxis}
    \end{tikzpicture}
    \subcaption{$\kappa=\pgfmathprintnumber{1e-3}$, polygonal}
  \end{minipage}
  \begin{minipage}[b]{0.32\linewidth}\centering
    \begin{tikzpicture}[scale=0.45]
      \begin{loglogaxis}
        \addplot table[x=meshsize,y=err_sharp] {adho_0_0_pi6_tiltedhexagonal.dat};
        \addplot table[x=meshsize,y=err_sharp] {adho_0_1_pi6_tiltedhexagonal.dat};
        \addplot table[x=meshsize,y=err_sharp] {adho_0_2_pi6_tiltedhexagonal.dat};
        \addplot table[x=meshsize,y=err_sharp] {adho_0_3_pi6_tiltedhexagonal.dat};
        \logLogSlopeTriangle{0.90}{0.4}{0.1}{1/2}{black};
        \logLogSlopeTriangle{0.90}{0.4}{0.1}{3/2}{black};
        \logLogSlopeTriangle{0.90}{0.4}{0.1}{5/2}{black};
        \logLogSlopeTriangle{0.90}{0.4}{0.1}{7/2}{black};
      \end{loglogaxis}
    \end{tikzpicture}
    \subcaption{$\kappa=0$, polygonal}
  \end{minipage}    
  \caption{$\norm[\sharp,h]{\Ih u - \uu[h]}$ vs. $h$ for the test case of Section~\ref{sec:adr:numerical.example}.\label{fig:convergence:adr}}
\end{figure}
The convergence results are depicted in Figure~\ref{sec:adr:numerical.example}. We observe that the convergence rate decreases with $\diff$, with a loss slightly less than the half order predicted by the error estimate of Theorem~\ref{thm:advdiffreac:en.err.est}.


\begin{acknowledgement}
  This work was funded by Agence Nationale de la Recherche grant HHOMM (ref. ANR-15-CE40-0005-01).
\end{acknowledgement}
\vspace{-1cm}


\begin{footnotesize}
  \bibliographystyle{plain}
  \bibliography{hho_sema-simai}
\end{footnotesize}

\end{document}